\documentclass[10pt]{article}

\usepackage[utf8]{inputenc}
\usepackage[T1]{fontenc}
\usepackage[utf8]{inputenc}
\usepackage{newtxtext}
\usepackage{latexsym}
\usepackage{mathrsfs}
\usepackage{typearea}
\typearea{12}

\usepackage{amsmath}
\usepackage{cases}
\usepackage{amssymb}
\usepackage{mathrsfs}
\usepackage[pdftex]{graphicx}
\usepackage[all]{xy}
\usepackage{color}
\usepackage{usebib}
\usepackage{comment}
\usepackage{hyperref}

\hypersetup{%
 setpagesize=false,%
 bookmarks=true,%
 bookmarksdepth=tocdepth,%
 bookmarksnumbered=true,%
 colorlinks=false,%
 pdftitle={},%
 pdfsubject={},%
 pdfauthor={},%
 pdfkeywords={}}
\usepackage{xcolor}
\hypersetup{
    colorlinks=true,
    citecolor=blue,
    linkcolor=black,
    urlcolor=orange,
}

\bibinput{article}

\makeatletter

\@addtoreset{equation}{section}
\makeatother
\def\tilde{\widetilde}
\usepackage{amsthm}
\newtheorem{thm}{Theorem}[section]
\newtheorem{dfn}[thm]{Definition}
\newtheorem{lem}[thm]{Lemma}
\newtheorem{prp}[thm]{Proposition}
\newtheorem{rmk}[thm]{Remark}

\newtheorem{cor}[thm]{Corollary}

\newtheorem{thm'}{Theorem}
\newtheorem{conj'}[thm']{Conjecture}
\newtheorem{ques'}[thm']{Question}

\title{Instanton knot invariants with rational holonomy parameters and an application for torus knot groups}
\author{Hayato Imori}
\date{}

\begin{document}
\nocite{*}
\maketitle
\begin{abstract}
There are several knot invariants in the literature that are defined using singular instantons. 
Such invariants provide strong tools to study the knot group and give topological applications.
For instance, it gives powerful tools to study the topology of knots in terms of representations of fundamental groups.
In particular, it is shown that any traceless representation of the torus knot group can be extended to any concordance from the torus knot to another knot.
Daemi and Scaduto proposed a generalization that is related to a version of the Slice-Ribbon conjecture to torus knots.
The results of this paper provide further evidence towards the positive answer to this question.
The method is a generalization of Daemi-Scaduto's equivariant singular instanton Floer theory following Echeverria's earlier work.
Moreover, the irreducible singular instanton homology of torus knots for all but finitely many rational holonomy parameters are determined as $\mathbb{Z}/4$-graded abelian groups.
\end{abstract}
\tableofcontents

\maketitle

\section{Introduction}
\subsection{Background}
Floer homology is an infinite dimensional analogue of Morse homology. In the context of gauge theory, instanton Floer homology \cite{Flo88}, Heegaard Floer homology  \cite{OS04b} and monopole Floer homology \cite{KrMr07} have provided strong topological invariants for low-dimensional manifolds.
Knot invariants  have been also developed in Floer theories. This list of knot invariants includes knot Floer homology introduced by Ozsv{\'a}th -Szab{\'o} \cite{OS04a} and Rasmussen \cite{Ras03}  in Heegaard Floer theory and Kronheimer-Mrowka \cite{kronheimer2010knots} in monopole Floer theory.
In the field of  instanton Floer theory, invariants of knot constructed by Floer \cite{floer1990instanton} and Braam-Donaldson \cite{braam1995floer} via framed surgery of knots. 
It is conjectured that their instanton knot invariants are related to knot invariants in Ozsv{\'a}th-Szab{\'o} and Rasmussen \cite{Ras03} by Kronheimer-Mrowka \cite{kronheimer2010knots}.
Collin-Steer \cite{collin1999instanton} and Kronheimer-Mrowka \cite{KM11} developed other type invariants for knots. 
While knot invariants in \cite{floer1990instanton} and \cite{braam1995floer} are related to invariants of 3-manifold via surgery along knot, knot invariants in \cite{collin1999instanton} and \cite{KM11} are related to 3-manifold's invariants via branched covering.

The advantage of instanton invariants is that it is directly related to fundamental groups of knot complement. 
For example, Kronheimer-Mrowka \cite{kronheimer2004dehn} show that the knot group $\pi_{1}(S^{3}\setminus K)$ for a non-trivial knot $K\subset S^{3}$ admits non-abelian representation $\pi_{1}(S^{3}\setminus K)\rightarrow SU(2)$. 
This is a  refinement of the result by Papakyriakopoulos \cite{papakyriakopoulos1957dehn} which states that $K\subset S^{3}$ is unknot if only if $\pi_{1}(S^{3}\setminus K)$ is infinitely cyclic.  
A concordance analogue of the result by Kronheimer-Mrowka \cite{kronheimer2004dehn} was given by Daemi-Scaduto \cite{DS2} using a version of instanton Floer theory.
Daemi-Scaduto \cite{DS2} also show the following statement which is specific to torus knots. 
\begin{thm'}  \label{DSthm8} {\rm (\cite[Theorem 8]{DS2} )}
Let $S:T_{p,q}\rightarrow K$ be a given smooth concordance. Then any traceless $SU(2)$-representation of $\pi_{1}(S^{3}\setminus T_{p, q})$ extends over the concordance complement. 
\end{thm'}
Here $T_{p, q}$ denotes the $(p, q)$-torus knot in $S^{3}$, where $p$ and $q$ are positive coprime integers. 
An $SU(2)$-traceless representation of $\pi_{1}(S^{3}\setminus K)$  is a  $SU(2)$-representation of $\pi_{1}(S^{3}\setminus K)$ which sends a homotopy class of meridian $\mu_{K}$ of $K$ to a traceless element in $SU(2)$.
The motivation of this theorem is related to a version of the Slice-Ribbon conjecture. 
A concordance $S:K\rightarrow K'$ is called ribbon concordance if the projection $S^{3}\times [0, 1]\supset S\rightarrow [0, 1]$ is a Morse function without any local maximums. 
Consider a knot $K$ which is concordant to unknot $U$.
The Slice-Ribbon conjecture proposed by Fox \cite{fox1962some} states that there is a ribbon concordance from $U$ to $K$ under this assumption. 
A generalization of the Slice-Ribbon conjecture by Daemi-Scaduto \cite{DS2} is 
\begin{conj'}\label{DSconj}
Let $K$ be a knot which is concordant to the $(p, q)$-torus knot $T_{p, q}$. Then there is a ribbon concordance from $T_{p, q}$ to $K$.
\end{conj'}
A necessary condition to show that a concordance $S:K\rightarrow K'$ is ribbon can be  stated in terms of representations of knot groups.
For a topological space $X$, we write $\mathcal{R}(X, SU(2))$ for the $SU(2)$-character variety of $X$ (i.e. the space of conjugacy classes of $SU(2)$-representations of $\pi_{1}(X)$).
\begin{thm'}{\rm (\cite{Gor}, \cite{DL})}\label{thm'}
Let $S:K\rightarrow K'$ be a ribbon concordance between two knots. Then the inclusion $i:S^{3}\setminus K\rightarrow S^{3}\times [0, 1]$ induces a surjection $i^{*} :\mathcal{R}(S^{3}\times [0, 1]\setminus S,SU(2))\rightarrow \mathcal{R}(S^{3}\setminus K, SU(2))$.   
\end{thm'}
Hence Theorem \ref{DSthm8} gives a piece of evidence towards Conjecture  \ref{DSconj}. 
The traceless condition  on representations of $\pi_{1}(S^{3}\setminus T_{p, q})$ arises from the specific type of knot invariants developed in \cite{DS1}. 
In light of Theorem \ref{thm'} and Conjecture \ref{DSconj}, it is natural to ask the following question.
\begin{ques'}\label{Ques}
Can we drop the traceless condition in Theorem \ref{DSthm8}?
\end{ques'}
In this paper, we will affirmatively solve this question. 
To explain our strategy, let us describe technical backgrounds of the Daemi-Scaduto's work. 
It mainly consists of three ingredients, singular gauge theory, equivariant Floer theory and the Chern-Simons filtration.
\\

Firstly, let us explain the notion of singular connections.
Let $K\subset Y$ be a knot in a 3-manifold. 
Roughly speaking, an $SU(2)$-singular connection $A$ is  an $SU(2)$-connection defined over the knot complement with the holonomy condition
\begin{equation}\lim_{r\rightarrow 0}{\rm Hol}_{\mu(r)}(A)\sim 
\left[\begin{array}{cc}
e^{2\pi i \alpha}&0\\
0&e^{-2\pi i \alpha}
\end{array}
\right].\label{holonomy}\end{equation}
Here $\mu(r)$ is a radius $r$ meridian of  $K\subset Y$ and  $\alpha$ is a fixed parameter in $(0, \frac{1}{2})$. $\sim$ denotes two matrices are conjugate in $SU(2)$. 
The parameter $\alpha$ is called the holonomy parameter of the singular connection $A$.  
In particular, a singular flat $SU(2)$-connection corresponds to an $SU(2)$-representation of $\pi_{1}(Y\setminus K)$ which sends the meridian $\mu_{K}$ of knot $K$ to an element which is conjugate to the matrix in (\ref{holonomy}).
Kronheimer-Mrowka developed singular version of Yang-Mills gauge theory in \cite{KM93, KM95, KM11}. 
Such kind of Floer homology theories constructed via singular connections are called singular instanton homology.
Singular gauge theory has different features compared to non-singular ones. 
In fact, singular Floer homology cannot defined over the coefficient ring $\mathbb{Z}$ for a general holonomy parameter $\alpha$.
To be more precise, singular instanton Floer homology is defined over $\mathbb{Z}$ only for $\alpha=\frac{1}{4}$.
This is called {\it the monotonicity condition}.
Most of works in singular instanton homology including \cite{DS1} and \cite{DS2} impose the monotonicity condition.
This is the reason that the statement of Theorem \ref{DSthm8} includes the traceless condition.

Next, we discuss the equivariant Floer theory.
 Fr{\o}yshov developed the homology cobordism invariant in \cite{fro02} and \cite{fro04}
 based on the equivariant Floer theory for integral homology 3-spheres, which was introduced by Donaldson \cite{Don}.
 The equivariant Floer theory introduced by Daemi-Scaduto \cite{DS1} produces invariants for a knot $K$ in an integral homology 3-sphere $Y$, and this can be regarded as 
 the counterpart of F{\o}yshov's work in singular gauge theory.
 Daemi-Scaduto's construction uses in a crucial way the $U(1)$-reducible singular flat connection $\theta$ which corresponds to the conjugacy class of  representation 
 \[\pi_{1}(Y\setminus K)\rightarrow H_{1}(Y\setminus K; \mathbb{Z})\rightarrow SU(2)\]
 whose image of the meridian $\mu_{K}$ of $K\subset Y$ is trace-free. 
 Here $\pi_{1}(Y\setminus K)\rightarrow H_{1}(Y\setminus K; \mathbb{Z})$ is the abelianization.
 In this situation, the construction which is similar to Floer's instanton homology \cite{Flo88} produces a chain complex $C_{*}(Y, K)$ for a knot in an integral homology 3-sphere.
 Its homology group $I_{*}(Y, K)$ can be interpreted as a categorification of knot signature for the case $Y=S^{3}$.
 Daemi-Scaduto \cite{DS1} also introduced chain complexes which have the following form \[\tilde{C}_{*}(Y, K):=C_{*}(Y, K)\oplus C_{*-1}(Y, K)\oplus \mathbb{Z}.\]
 Such objects are called $\mathcal{S}$-complex. 
 This can be interpreted as a version of $S^{1}$-equivariant Floer theory. 
 Let $\mathcal{B}(Y, K)$ be a configuration space of singular connections over $(Y, K)$ with a holonomy parameter $\alpha=\frac{1}{4}$. 
 Then there is a configuration space $\mathcal{B}(Y, K)_{0}$ of framed connections. The Chern-Simons functional on $\mathcal{B}(Y, K)$ lifts to $\mathcal{B}(Y, K)_{0}$ in an equivariant way. An $\mathcal{S}$-complex $\widetilde{C}_{*}(Y, K)$ is  related to the lifted $S^{1}$-equivariant Chern-Simons functional on $\mathcal{B}(Y, K)_{0}$.
 
 Another feature of Daemi-Scaduto's construction is the Chern-Simons filtration of $\mathcal{S}$-complexes.
 While usual instanton Floer theory is the analogue of Morse theory on the configuration space, its filtered version can be seen the Morse theory on the universal covering of the configuration space.
 The Chern-Simons filtration gives more refined structures on $\mathcal{S}$-complexes.
The counterpart idea in non-singular instanton Floer theory was used by Daemi  \cite{Dae2020} and Nozaki-Sato-Taniguchi  \cite{nozaki2019filtered}, which provided homology cobordism invariants. 
  
Any of the above version of singular instanton Floer theories can be extended to an arbitrary holonomy parameters if the integer coefficient ring is replaced with a Novikov ring $\Lambda$ by the Echeverria's work \cite{E19}.
To be more precise, the holonomy parameter should satisfy the technical condition $\Delta_{(Y, K)}(e^{4\pi i \alpha})\neq 0$, where $\Delta_{(Y, K)}$ is the Alexander polynomial for $K\subset Y$.
  One of the flavors of Echeverria's Floer homology is a categorification of the Tristram-Levine signature when $Y=S^{3}$.
  For a knot $K$ in an integral homology 3-sphere $Y$, the Tristram-Levine signature is given by
  \[\sigma_{\alpha}(Y, K):={\rm sign}[(1-e^{4\pi i \alpha})V+(1-e^{-4\pi i \alpha})V^{T}]\]
   where $V$ is a Seifert matrix form of $K\subset Y$. 
 For the case $Y=S^{3}$, we omit $Y$ from the notation.
 
 Our strategy to drop the traceless condition from Theorem \ref{DSthm8} is constructing a family of $\mathcal{S}$-complexes for general holonomy parameters.
 \\
 
\subsection{Summary of Results}\label{summary}

First of all, we state the main theorem of this paper which gives the positive answer to Question \ref{Ques}. 
\begin{thm}\label{main thm}
For a given knot $K$ and a smooth concordance $S: T_{p, q}\rightarrow K$, any $SU(2)$-representation of 
 $\pi_{1}(S^{3}\setminus T_{p, q})$ extends to an $SU(2)$-representation of $\pi_{1}((S^{3}\times [0, 1])\setminus S)$.
\end{thm}
The proof of Theorem \ref{main thm} requires the special property that all generators of singular instanton homology for torus knots have odd gradings.  The outline of the proof is as follows. 
After extending the condition of \cite{DS1}, we define analogues knot Floer theory of \cite{DS1} for  all  $\alpha\in \mathcal{I}$,  where $\mathcal{I}$ is a dense subset of $[0, \frac{1}{2}]$. 
This means that all $SU(2)$-representations of $\pi_{1}(S^{3}\setminus T_{p, q})$ with the holonomy parameter $\alpha\in \mathcal{I}$ extend to the concordance complement. The limiting argument shows that this extension property is true for all $SU(2)$-representations of $\pi_{1}(S^{3}\setminus T_{p, q})$  with any holonomy parameter $\alpha \in [0, \frac{1}{2}]$.

As described above, singular instanton knot homology \cite{E19} and its equivariant counterparts are key tools for the proof of Theorem.
Let us review essential properties of these objects employed in this paper.
We consider the Novikov ring $\Lambda^{\mathbb{Z}[T^{-1},T]\!]}$  which is given by
\[\Lambda^{\mathbb{Z}[T^{-1},T]\!]}:=\left\{\sum_{r\in \mathbb{R}}p_{r}\lambda^{r}\ \middle|\  p_{r}\in \mathbb{Z}[T^{-1}, T]\!], \forall C>0, \#\{p_{r}\neq 0\}_{r>C}<\infty\right\}.\]
Let $\alpha$ be a parameter in $(0, \frac{1}{2})$. We introduce the following subring $\mathscr{R}_{\alpha}$ of $\Lambda^{\mathbb{Z}[T^{-1}, T]\!]}$.
\[\mathscr{R}_{\alpha}:=\begin{cases}
\mathbb{Z}[\xi^{\pm 1}_{\alpha}][\![\lambda^{-1},\lambda]&(\alpha\leq\frac{1}{4})\\
\mathbb{Z}[\lambda^{\pm 1}][\![\xi^{-1}_{\alpha}, \xi_{\alpha}]&(\alpha>\frac{1}{4})
\end{cases}\]
where $\xi_{\alpha}=\lambda^{2\alpha}T^{2}$.
The geometric aspect of the subring $\mathscr{R}_{\alpha}$ is described in Subsection \ref{Floer homology}.
\begin{thm}
Let $\mathscr{S}$ be an algebra over $\mathscr{R}_{\alpha}$.
Let $K\subset Y$ be an oriented knot in an integral homology 3-sphere. Choose a holonomy parameter $\alpha\in \mathbb{Q}\cap (0, \frac{1}{2})$ so that $\Delta_{(Y, K)}(e^{4\pi i \alpha})\neq 0$. 
Then we can associate a $\mathbb{Z}/4$-graded module $I^{\alpha}_{*}(Y, K;\Delta_{\mathscr{S}})$ over $\mathscr{S}$.  
Moreover, if $\mathscr{S}$ is an integral domain, we can associate a $\mathbb{Z}/4$-graded $\mathcal{S}$-complex
$(\tilde{C}^{\alpha}_{*}(Y, K;\Delta_{\mathscr{S}}), \tilde{d}, \chi)$ to a given triple $(Y, K, \alpha)$ with $\Delta_{(Y, K)}(e^{4\pi i \alpha})\neq 0$.
\end{thm}
The precise definition of $\mathcal{S}$-complex can be seen in Subsection \ref{4-1}.
We call $I^{\alpha}_{*}(Y, K; \Delta_{\mathscr{S}})$ {\it irreducible singular instanton knot homology over $\mathscr{S}$ with a holonomy parameter $\alpha$}.
For the case $Y=S^{3}$, we drop $Y$ from these notations.
The difference of the construction of our Floer homology $I^{\alpha}_{*}(Y, K; \Delta_{\mathscr{S}})$ and $I_{*}(Y, K, \alpha)$ introduced by \cite{E19} is the choice of local coefficients.
The construction of $(C^{\alpha}_{*}(Y, K), d)$ and $(\tilde{C}^{\alpha}_{*}(Y, K), \tilde{d}, \chi)$ depend on additional data (metric and perturbation), however, their chain homotopy classes in the sense of $\mathcal{S}$-complex are independent of such choices. 

\begin{rmk}
{\rm
For the coefficient $\mathscr{S}=\mathscr{R}_{\alpha}$, the Floer chain complex $C^{\alpha}_{*}(Y, K;\Delta_{\mathscr{S}})$ and $\mathcal{S}$-complex $\tilde{C}^{\alpha}_{*}(Y, K;\Delta_{\mathscr{S}})$ has $\mathbb{Z}\times \mathbb{R}$-bigraded structure if we fix auxiliary date.
Moreover, they have a filtered structure induced from $\mathbb{R}$-grading.
The precise description of $\mathbb{Z}\times \mathbb{R}$-bigrading and the filtered structure are contained in subsection \ref{Floer homology} and \ref{vmu}.}
\end{rmk}

The following statement describes behavior of $\mathcal{S}$-complex under the connected sum.
\begin{thm}\label{connected sum S-cpx}
Let $\mathscr{S}$ be an integral domain over $\mathscr{R}_{\alpha}$.
Let $K\subset Y$ and $K'\subset Y'$be a two oriented knots in integral homology 3-spheres. Fix a holonomy parameter $\alpha\in \mathbb{Q}\cap (0, \frac{1}{2})$ such that $\Delta_{(Y, K)}(e^{4\pi  i \alpha})\cdot\Delta_{(Y', K')}(e^{4\pi i \alpha})\neq 0$. 
Then there is a chain homotopy equivalence of $\mathcal{S}$-complexes
\[\tilde{C}^{\alpha}_{*}(Y\#Y', K\#K';\Delta_{\mathscr{S}})\simeq\tilde{C}^{\alpha}_{*}(Y, K;\Delta_{\mathscr{S}})\otimes_{\mathscr{S}}\tilde{C}^{\alpha}_{*}(Y', K';\Delta_{\mathscr{S}}).\]
\end{thm}
The precise definition of {\it a chain homotopy equivalence of $\mathcal{S}$-complex} can be seen in Subsection \ref{Scpx}.
This is a generalization of connected sum theorem by Daemi-Scaduto \cite{DS1}. 
The method of the proof of \cite[Theorem 6.1 ]{DS1} cannot be directly adapted to prove Theorem \ref{connected sum S-cpx} since we have to deal with  the non-monotonicity situation which arises for general holonomy parameters.
\begin{rmk}
{\rm
For the  coefficient ring $\mathscr{R}_{\alpha}$, the connected sum theorem for $\mathcal{S}$-complex as in Theorem \ref{connected sum S-cpx} is still hold as a $\mathbb{Z}\times \mathbb{R}$-bigraded complex.}
\end{rmk}

 As described in \cite{DS1}, we can associate an integer valued invariant which is called Fr{\o}yshov type invariant to a given $\mathcal{S}$-complex.
 Our construction of $\mathcal{S}$-complex provides  an integer-valued invariant $h_{\mathscr{S}}^{\alpha}(Y, K)$ for a knot in a homology 3-sphere $(Y, K)$. 
  We call $h^{\alpha}_{\mathscr{S}}(Y, K)$ {\it the Fr{\o}yshov invariant for $(Y, K)$ over $\mathscr{S}$ with a holonomy parameter $\alpha$ }.
   We drop $Y$ from the notation when $Y=S^{3}$.
   Note that Echeverria \cite{E19} also introduced Fr{\o}yshov type invariant denoted by $h(Y, K, \alpha)$, which is constructed from singular instanton Floer homology with a different local coefficient system from our setting.
  The invariant $h^{\alpha}_{\mathscr{S}}(Y, K)$ satisfies the following properties.
 \begin{thm}\label{Froyshov prp}
  Let $(Y, K)$ and $(Y', K')$ are two pairs of integral homology 3-spheres and knots. 
  Assume that $\alpha\in\mathbb{Q}\cap(0, \frac{1}{2})$ satisfies $\Delta_{(Y, K)}(e^{4\pi i \alpha})\neq 0$ and $\Delta_{(Y', K')}(e^{4\pi i \alpha})\neq 0$.
  Then 
\[h^{\alpha}_{\mathscr{S}}(Y\#Y', K\#K')=h^{\alpha}_{\mathscr{S}}(Y, K)+h^{\alpha}_{\mathscr{S}}(Y', K')\]
Moreover, if $(Y, K)$ and $(Y', K')$ are homology concordant, then
\[h^{\alpha}_{\mathscr{S}}(Y, K)=h^{\alpha}_{\mathscr{S}}(Y', K').\]
\end{thm}

 Let us consider a knot in $S^3$. 
 It is shown that  Fr{\o}yshov type invariant in \cite{DS1} reduces to knot signature (\cite[Theorem 7]{DS2}).
 The invariant $h_{\mathscr{S}}^{\alpha}$ reduces to the Tristram-Levine signature as follows.
 \begin{thm}\label{Froyshov}
Let $\mathscr{S}$ be an integral domain over $\mathscr{R}_{\alpha}$.
 For any knot $K\subset S^{3}$ and for a holonomy parameter  $\alpha\in (0, \frac{1}{2})\cap\mathbb{Q}$ with $\Delta_{K}(e^{4\pi i \alpha})\neq 0$, the following equality holds;
 \[h_{\mathscr{S}}^{\alpha}(K)=-\frac{1}{2}\sigma_{\alpha}(K).\]

 \end{thm}
 For a given knot $K\subset S^3$ and integer $l$, we define a knot $l K\subset S^{3}$ so that
 \[
 l K:=
 \begin{cases}
 {\#_{l}K}&(l>0)\\
 U\text{ (unknot)}&  (l=0)\\
 {\#_{-l}(-K)}&(l<0)
 \end{cases}
 \]
where $-K$ is the mirror of $K$ with the reverse orientation.
More strongly, $\mathcal{S}$-complexes have the following structure theorem.
 \begin{thm}\label{str}
Let $\mathscr{S}$ be an integral domain over $\mathscr{R}_{\alpha}$.
 For a knot $K$ in $S^{3}$ and for a holonomy parameter $\alpha\in \mathbb{Q}\cap (0, \frac{1}{2})$ with $\Delta_{K}(e^{4\pi i \alpha})\neq 0$, there is an two-bridge torus knot $T_{2, 2n+1}$ such that $\Delta_{T_{2,2n+1}}(e^{4\pi i \alpha})\neq 0$, $\sigma_{\alpha}(T_{2, 2n+1})=-2$ and the relation
 \[\tilde{C}^{\alpha}_{*}(K;\Delta_{\mathscr{S}})\simeq \tilde{C}^{\alpha}_{*}(lT_{2, 2n+1};\Delta_{\mathscr{S}})\]
holds, where $l=-\frac{1}{2}\sigma_{\alpha}(K)$.
 \end{thm}
For the proof of Theorem \ref{str}, it is essential to observe behaviors of morphisms of $\mathcal{S}$-complexes induced from cobordisms between pairs $(Y, K)$ and $(Y', K')$.
In \cite{DS2},  techniques in \cite{Kr97} are used to describe behaviors of morphisms of $\mathcal{S}$-complexes  for the case $\alpha=\frac{1}{4}$. 
However, such techniques do not directly adapt to prove Theorem \ref{str} because of the lack of the monotonicity condition.

Theorem \ref{connected sum S-cpx} and Theorem \ref{str} imply the Euler characteristic formula
\begin{equation}\chi(I^{\alpha}(K, \Delta_{\mathscr{S}}))=\frac{1}{2}\sigma_{\alpha}(K),\label{Euler char}\end{equation}
as described in subsection \ref{eulerchar}.
Since our grading convention of generators coincides with that of \cite{E19}, the above argument also gives an alternative proof of the Euler characteristic formula in  \cite[Theorem 17]{E19} for the case $Y=S^{3}$.

Next, we focus on $(p, q)$-torus knot $ T_{p, q}$ in 3-sphere. We always  assume that $p$ and $q$ are positive coprime integers. The following is a characteristic property of the torus knot and a key-lemma for the proof of Theorem \ref{main thm}.
Let $\mathcal{R}_{\alpha}(Y\setminus K, SU(2))$ be the space of conjugacy classes of $SU(2)$-representations of $\pi_{1}(Y\setminus K)$ with a holonomy parameter $\alpha$.
Let $\mathcal{R}^{*}_{\alpha}(Y\setminus K, SU(2))$ be its irreducible part.
\begin{thm}\label{absolute counting}
For any $\alpha\in [0, \frac{1}{2}]$ with $\Delta_{T_{p, q}}(e^{4\pi i \alpha})\neq 0$,
\[|\mathcal{R}^{*}_{\alpha}(S^{3}\setminus T_{p,q }, SU(2))|=-\frac{1}{2}\sigma_{\alpha}(T_{p, q}).\]
\end{thm}
Here $|S|$ for a set $S$ denotes the size of this set.
In \cite{her97}, Herald introduced the {\it signed count} of elements in the character variety $\mathcal{R}^{*}_{\alpha}(S^3\setminus K, SU(2))$ for a general knot $K$ with a fixed holonomy parameter.
One first perturbs $\mathcal{R}^{*}_{\alpha}(S^3\setminus K, SU(2))$ into a discrete set $\mathcal{R}^{*,h}_{\alpha}(S^{3}\setminus K, SU(2))$ and then associates a sign to each element of this set.
The sum of these signs is Herald's signed count of $\mathcal{R}^{*}_{\alpha}(S^3\setminus K, SU(2))$, which we denote by $\#\mathcal{R}^{*}_{\alpha}(S^3\setminus K, SU(2))$.
In general, we have:
\[\#\mathcal{R}^{*}_{\alpha}(S^{3}\setminus K, SU(2))=-\frac{1}{2}\sigma_{\alpha}(K)\]
 See \cite[Corollary 0.2 ]{her97} and \cite{Lin} for the case $\alpha=\frac{1}{4}$.
In the case $K=T_{p, q}$, the character variety $\mathcal{R}^{*}_{\alpha}(S^3\setminus T_{p, q}, SU(2))$ is already discrete and one does not make any perturbation. Theorem \ref{absolute counting} implies that  all elements of $\mathcal{R}^{*}_{\alpha}(S^{3}\setminus T_{p, q}, SU(2))$ have positive signs.

 Theorem \ref{absolute counting} implies that $C^{\alpha}_{*}(T_{p, q};\Delta_{\mathscr{S}})$ is supported only on odd graded components.
In particular, its homology groups are isomorphic to chain complexes,\[I^{\alpha}_{*}(T_{p, q};\Delta_{\mathscr{S}})\cong C^{\alpha}_{*}(T_{p, q};\Delta_{\mathscr{S}})\]
since all differentials of chain complexes are trivial.
This can be  interpreted as the counterpart of the computation of instanton homology of Brieskorn homology $3$-spheres in \cite{FS90}.

Theorem \ref{Froyshov} implies that ${\rm rank}C^{\alpha}_{*}(T_{p, q};\Delta_{\mathscr{S}})= h_{\mathscr{S}}^{\alpha}(T_{p, q})$ and by the definition of the invariant $h_{\mathscr{S}}^{\alpha}$ we also have the following statement 

\begin{thm}\label{sub thm}

Let $\mathscr{S}$ be an algebra over $\mathscr{R}_{\alpha}$.
For $\alpha\in(0, \frac{1}{2})\cap\mathbb{Q}$ with $\Delta_{T_{p, q}}(e^{4\pi i \alpha})\neq 0$, there is an isomorphism,
\[I^{\alpha}_{*}(T_{p, q}; \Delta_{\mathscr{S}})\cong \mathscr{S}_{(1)}^{\lceil-\sigma_{\alpha}(T_{p, q})/4\rceil}\oplus \mathscr{S}_{(3)}^{\lfloor-\sigma_{\alpha}(T_{p, q})/4\rfloor}\]
as a $\mathbb{Z}/4$-graded abelian group.
 
\end{thm}
Theorem \ref{sub thm} describes the grading of generators of Floer chain and it is independent of the choice of local coefficient systems. 
A similar structure theorem holds for singular instanton knot homology introduced by \cite{E19}.
\\

Theorem \ref{str} imply that  $\mathcal{S}$-complexes for knots are determined by the Tristram-Levine signature without the $\mathbb{Z}\times \mathbb{R}$-grading structure. 
On the other hand, the $\mathbb{R}$-grading from the Chern-Simons filtration can be expected to have stronger information of the knot concordance.
In an upcoming work \cite{DISST} and relying on the results of the present paper, we will introduce a generalization of the $\Gamma$-invariant of \cite{DS1} for rational holonomy parameters, which can be regarded as a gauge theoretic refinement of the Tristram-Levine signature.
The techniques from this paper are also used in a future work of Daemi and Scaduto to construct families of hyperbolic knots that are minimal with respect to the ribbon partial order \cite{Gor}.

\subsection{Outline of the paper}
In Section 2, we review the background of $SU(2)$-singular gauge theory for rational holonomy parameters. 
We also introduce the generalized definition of negative definite cobordism in this section.
In Section 3, we construct Floer chain groups and $\mathcal{S}$-complex parametrized by holonomy parameters $\alpha$ and introduce the Fr{\o}yshov type invariant. The argument is almost parallel to \cite{DS1} and \cite{DS2}, however, we need a careful choice of the local coefficient systems if we introduce bigraded structure on the Floer chain complex.
The proof of Theorem \ref{Froyshov prp} is also contained. 
In Section 4, we prove the Tristram-Levine signature formula for torus knots (Theorem \ref{absolute counting}). 
In the proof of Theorem \ref{absolute counting}, we use the correspondence of singular flat connections and non-singular flat connections over the branched covering space.
We also use the pillowcase picture of the $SU(2)$-character variety for the knot complement space.
In Section 5, we prove Theorem \ref{Froyshov}, \ref{str}, and \ref{sub thm} in Subsection \ref{eulerchar}, and finally we give the proof of our main theorem (Theorem \ref{main thm}) in Subsection \ref{self-con}. 
The bigraded structure of $\mathcal{S}$-complex plays an important role in the proof of the main theorem. 
Appendix consists of the proof of the connected sum theorem (Theorem\ref{connected sum S-cpx}).
\\

{\bf Acknowledgments.}
The author would like to thank Aliakbar Daemi for his introduction to singular instanton knot homology, helpful suggestions, and answering many questions on papers \cite{DS1} and \cite{DS2}. The author would also like to thank Kouki Sato and  Masaki Taniguchi for helpful discussions.
This work is supported by JSPS KAKENHI Grant Number JP 21J20203.

\section{Backgrouds on singular instantons}
In this section, we review singular gauge theory mainly developed by Kronheimer-Mrowka \cite{KM11}. 
We also give a generalization of the setting of singular $SU(2)$-gauge theory adopted by Daemi-Scaduto \cite{DS1, DS2}.

\subsection{The space of singular connections}\label{sp of sing conn}
We review the construction of singular instantons over a closed pair of 4-manifold and surface. Let $X$  be a closed oriented smooth  4-manifold and $S$ be a closed oriented embedded surface in $X$. Let $N$ be a tubular neighborhood of $S$ in $X$.  We identify $N$ with a disk bundle over $S$, and $\partial N$ with a circle bundle over $S$. Let $\eta$ be a connection 1-form on a circle bundle $\partial N$. This means that $\eta$ is $U(1)$-invariant. We fix the decomposition of the $SU(2)$-bundle $E\rightarrow X$ over the embedded surface $S$, as $E|_{S}=L\oplus L^*$, where $L$ is a $U(1)$-bundle over $S$. This decomposition extends to $N$.
We define two topological invariants,
$$k=c_{2}(E)[X]$$
$$l=-c_{1}(L)[S]$$
$k$ is called the instanton number, and $l$ is called the monopole number. 
\\

Next, we fix a connection $A_{0}$ over $X$ of the form 
\[
A_{0}|_{N}=\left[\begin{array}{cc}
b&0\\
0&-b

\end{array}
\right]
\]
Here $b$ is a connection over $L$. This means that $A^{0}$ reduces to a $U(1)$-connection over $S$.  We give the polar coordinate $(r, \theta)\in D^{2}$ on each fiber of $N$. Let $\eta$ be a 1-form obtained by pulled-back 1-form on  $\partial N$ which coincides $d\theta$ on each fiber, and $\psi$ be a cut-off function supported on $N$. We define the singular base connection $A^{\alpha}$ as follows:
\[A^{\alpha}=A_{0}+ i\psi\left[\begin{array}{cc}
\alpha&0\\
0&-\alpha

\end{array}\right]\eta,
\]
where $\alpha\in (0, \frac{1}{2})$. $\alpha$ is called the holonomy parameter. Recall that $\eta$ is defined only on $N\setminus S$, but extends $0$ to $X\setminus S$ after cutting off by $\psi$. $A^{\alpha}$ is a connection over $X\setminus S$. Let $\mathfrak{g}_{E}$ be the adjoint bundle of $E$. For $a\in \Omega^{1}(X\setminus S, \mathfrak{g}_{E})$, $A^{\alpha}+a$ is called a singular connection. 
\\
Before defining the space of singular connections, we have to introduce functional spaces.
We fix the orbifold metric on $X$, which can be written in the form
$$g^{\nu}=du^{2}+dv^{2}+dr^{2}+\frac{r^{2}}{\nu^{2}}d\theta^{2}$$on $N$, where $u, v$ are local coordinate of $S$. We say that this orbifold metric has a cone angle $\frac{2\pi}{\nu}$. $(X,g^{\nu})$  has a local structure $U/\mathbb{Z}_{\nu}$ near the singular locus $S$, where $U$ is an open set in $\mathbb{R}^{4}$.   The model connection $A^{\alpha}$ induces an $SO(3)$-adjoint connection on $\mathfrak{g}_{E}$.  We define covariant derivative $\check{\nabla}_{A^{\alpha}}$ on bundle $\Lambda^{m}\otimes \mathfrak{g}_{E}$ using the adjoint connection of $A^{\alpha}$ and Levi-Civita connection with respect to  the metric $g^{\nu}$.
Let $F\rightarrow X$ be a orbifold vector bundle.
The Sobolev space $\check{L}^{p}_{m, A^{\alpha}}(X\setminus S, F)$ is defined as the completion of spaces of smooth sections of $F\rightarrow X$ by the norm
$$\|s\|^{p}_{\check{L}^{p}_{m, A^{\alpha}}}=\sum_{i=0}^{m}\int_{X\setminus S}|\check{\nabla}_{A^{\alpha}}s|^{p}{\rm dvol}_{g^{\nu}}.$$
If we use orbifold metrics, the "Fredholm package" works. 
Let $d^{+}_{A^{\alpha}}:\Omega^{1}(X\setminus S, \mathfrak{g}_{E})\rightarrow \Omega^{+}(X\setminus S, \mathfrak{g}_{E})$ be a linearized anti-self-dual operator defined by the metric $g^{\nu}$,  and $d^{*}_{A^{ \alpha}}: \Omega^{1}(X\setminus S, \mathfrak{g}_{E})\rightarrow \Omega^{0}(X\setminus S, \mathfrak{g}_{E})$ be the formal adjoint of the covariant derivative for the metric $g^{\nu}$. Consider the elliptic operator $D_{A^{\alpha}}=-d^{*}_{A^{ \alpha}}\oplus d^{+}_{A^{ \alpha}}$  acting on the Sobolev space
\begin{equation}\check{L}^{p}_{m, A^{ \alpha}}(X\setminus S, \Lambda^{1}\otimes \mathfrak{g}_{E})\rightarrow \check{L}^{p}_{m-1, A^{ \alpha}}(X\setminus S, (\Lambda^{0}\oplus \Lambda^{+})\otimes \mathfrak{g}_{E}).\label{deform}
\end{equation}
\begin{prp}\label{Fredholm}
Let $\alpha$ be a rational holonomy parameter of the form $\alpha=p/q\in (0, \frac{1}{2})\cap \mathbb{Q}$. Choose cone angle $\frac{2\pi}{\nu}$ of orbifold metric so that $\frac{2\nu p}{q}\in \mathbb{Z}$. 
Then the operator $D_{A^{\alpha}}$ and its formal adjoint are Fredholm,  and the Fredholm alternative holds.
\end{prp}
Let $A^{\alpha}_{\rm ad}$ be the  adjoint of a singular connection $A^{\alpha}$ and $\pi :U\rightarrow U/\mathbb{Z_{\nu}}$ be an orbifold chart with respect to the orbifold metric $g^{\nu}$.
If $\nu\in \mathbb{Z}_{>0}$ is chosen as in Proposition \ref{Fredholm}, the lift of adjoint connection of $\pi^{*}A^{\alpha}$ has an asymptotically trivial holonomy along a small linking  of the singular locus. Thus $\pi^{*}A^{\alpha}$  extends smoothly over $U$. This means that $A^{\alpha}_{\rm ad}$ defines an orbifold connection.  All analytical argument reduces to orbifold setting.
From now on, we always fix $\nu$ as in Proposition \ref{Fredholm} for a given rational holonomy parameter. 

Assume that $m>2$. The space of singular connections with a holonomy parameter $\alpha\in (0, \frac{1}{2})$ is given by,
$$\mathcal{A}(X, S, \alpha)=\{ A^{\alpha}+ a| a \in \check{L}^{2}_{m, A^{\alpha}}(X\setminus S, \Lambda^{1}\otimes\mathfrak{g}_{E})\}.$$
This is an affine space as non-singular case. Notice that $\mathcal{A}(X, S, \alpha)$ is independent of the choice of the base connection $A^{\alpha}$.
We also introduce the group of gauge transformations,
$$\mathcal{G}(X, S)=\{g \in {\rm Aut}(E)|\  g \in \check{L}^{2}_{m+1, A^{\alpha}}(X\setminus S, {\rm End}(E))\}.$$
  There is a smooth action of $\mathcal{G}(X, S)$ on $\mathcal{A}(X, S, \alpha)$ and we can take the quotient.
$$\mathcal{B}(X, S, \alpha)=\mathcal{A}(X, S, \alpha)/\mathcal{G}(X, S).$$
A singular connection with 0-dimensional stabilizer for the action of $\mathcal{G}(X, S)$ is called irreducible connection.
A singular connection  is called reducible if it is not irreducible.
The quotient space $\mathcal{B}(X, S,\alpha)$ has a smooth Banach manifold structure except for orbits of reducible connections.  
The set of gauge equivalence classes of solutions for the anti-self-dual equation
\[M^{\alpha}(X, S)=\{[A]\in \mathcal{B}(X, S,\alpha)|F^{+}_{A}=0\}\]
is called the moduli space of singular anti-self-dual connections.
$M^{\alpha}(X, S)_{d}$ denotes the subset of $M^{\alpha}(X, S)$ with expected dimension $d$.
For a generic orbifold metric with a fixed cone angle, the irreducible part of $M^{\alpha}(X, S)_{d}$ is a smooth  manifold of dimension $d={\rm ind} (d^{*}_{A}\oplus d^{+}_{A})$, where $[A]\in M^{\alpha}(X, S)_{d}$.
If $M^{\alpha}(X, S)_{d}$ consists of reducible connections, we modify the dimension of moduli space so that $d={\rm ind}(d^{*}_{A}\oplus d^{+}_{A})+{\rm dim}H^{0}_{A}$, where $H^{i}_{A}$ is an $i$-th cohomology group of the deformation complex.
The index of the ASD-operator $d^{*}_{A}\oplus d^{+}_{A}$ is given by 
\[{\rm ind }(d^{*}_{A}\oplus d^{+}_{A})=8k+4l-3(1-b^{1}+b^{+})-2(g(S)-1),\]
where $g(S)$ is the genus of the surface $S$.
The index formula for a closed pair $(X, S)$ does not depend on the holonomy parameter $\alpha$.
On the other hand, the energy integral $\kappa(A)=\|F_{A}\|_{\check{L}^{2}}$  for an ASD-connection $A$ is given by
\[\kappa(A)=k+2\alpha l-\alpha^{2}S\cdot S.\]
Throughout this paper, we assume that an integer $\nu>0$ is chosen large enough for a fixed holonomy parameter $\alpha\in \mathbb{Q}\cap (0, \frac{1}{2})$, under the condition $2\alpha \nu\in \mathbb{Z}$. 
Such choice of $\nu$ is related to the bubbling and compactification of moduli spaces.
The details are described in \cite{KM93, KM95}.

\subsection{The Chern-Simons functional\label{chern-simons}}
We discuss singular connections over 3-manifolds.
Let $Y$ be an oriented integral homology 3-sphere and $K$ be an oriented knot in $Y$. Let $E$ be an $SU(2)$-bundle over $Y$. This is always topologically trivial. We fix the reduction of $E$ to a line bundle over $K$ as $E|_{K}=L\oplus L^{*}$.  We fix orbifold metric $g^{\nu}$ along $K$as the similar way in subsection \ref{sp of sing conn}. 
For a fixed $\alpha\in \mathbb{Q}\cap (0, \frac{1}{2})$, we choose $\nu$ as in Proposition \ref{Fredholm}.
 We can similarly define the space of singular connections and gauge transformations.
$$\mathcal{A}(Y, K, \alpha)=\{ A^{\alpha}+ a|\  a \in \check{L}^{2}_{m, A^{\alpha}}(Y\setminus K, \mathfrak{g}_{E})\},$$
$$\mathcal{G}(Y, K)=\{g \in {\rm Aut}(E)|\  g \in \check{L}^{2}_{m+1, A^{\alpha}}(Y\setminus K, {\rm End}(E))\}.$$

We define the quotient
$$\mathcal{B}(Y, K,\alpha)=\mathcal{A}(Y, K, \alpha)/\mathcal{G}(Y, K).$$
We use the notation $\mathcal{A}_{m}(Y, K, \alpha)$ if we emphasize that the space of singular connection is defined by completion of Sobolev norm of $\check{L}^{2}_{m}$.

We  describe the topology of $\mathcal{G}(Y, K)$ and $\mathcal{B}(Y, K, \alpha)$.
There are other two kinds of group of gauge transformation,
$$\mathcal{G}_{K}=\{\ g\in {\rm Aut}(L|_{K})\},$$
$$\mathcal{G}^{K}(Y, K)=\{g\in {\rm Aut }(E)|\ g|_{K}={\rm id}\}.$$
Then there is an exact sequence,
$$1\rightarrow \mathcal{G}^{K}(Y, K)\rightarrow  \mathcal{G}(Y, K)\rightarrow \mathcal{G}_{K}\rightarrow 1.$$
There is a map $\mathcal{G}(Y, K)\rightarrow \mathbb{Z}\oplus \mathbb{Z}$ given by $d(g)=(k, l)$, where $k={\rm deg}(g:Y\rightarrow SU(2))$ and $l={\rm deg}(g|_{K}:K\rightarrow U(1))$, and this map induces an isomorphism $$\pi_{0}(\mathcal{G}(Y, K))\cong\mathbb{Z}\oplus\mathbb{Z}.$$
Using the homotopy exact sequence induced from a fibration $\mathcal{G}(Y, K)\rightarrow \mathcal{A}(Y, K,\alpha)\rightarrow \mathcal{B}(Y, K,\alpha)$, we also have an isomorphism $$\pi_{1}(\mathcal{B}(Y, K,\alpha))\cong \mathbb{Z}\oplus \mathbb{Z}.$$ 

We define a  $L^{2}$-inner product on tangent spaces of $\mathcal{A}(Y, K, \alpha)$ as follows:
$$\langle a, b\rangle =\int_{{Y}\setminus K}-{\rm tr}(a\wedge *b).$$The $*$-operator is given by the orbifold metric $g^{\nu}$.
The Chern-Simons functional $CS :\mathcal{A}(Y, K, \alpha)\rightarrow \mathbb{R}$ is given by the formal gradient
$${\rm grad}(CS)_{A}=\frac{1}{4\pi^2}*F_{A}$$  with respect to the above $L^{2}$-inner product on $T\mathcal{A}(Y, K, \alpha)$. This uniquely determines $CS$ up to constant. $A\in \mathcal{A}(Y, K, \alpha)$ is a critical point of $CS$ if only if $F_{A}=0$. The critical point set of $CS$ is a space of flat connections on $Y\setminus K$ such that their holonomy along the meridian are conjugate to
\[\left[\begin{array}{cc}
e^{2\pi i \alpha}&0\\
0&e^{-2\pi i \alpha}
\end{array}\right].
\]
Let ${\rm Crit}$ be a critical point set of the Chern-Simons functional $CS:\mathcal{A}(Y, K, \alpha)\rightarrow \mathbb{R}$ and ${\rm Crit}^{*}={\rm Crit}\cap \mathcal{A}^{*}(Y, K,  \alpha)$. Let $\mathfrak{C}(Y, K, \alpha)$ and $\mathfrak{C}^{*}(Y, K, \alpha)$  be images of  ${\rm Crit}$ and ${\rm Crit}^{*}$ by the natural projection $\mathcal{A}(Y, K, \alpha)\rightarrow \mathcal{B}(Y, K, \alpha)$. Then 

$$\mathfrak{C}(Y, K,\alpha)=\mathcal{R}_{\alpha}(Y\setminus K, SU(2)),$$
$$\mathfrak{C}^{*}(Y,K,\alpha)=\mathcal{R}^{*}_{\alpha}(Y\setminus K, SU(2))$$by the holonomy correspondence of flat connections and representations of the fundamental group. 

We have to perturb  Chern-Simons functional to achieve transversality. 
This is done by introducing a cylinder function associated to a perturbation $\pi \in \mathcal{P}$
$$f_{\pi}: \mathcal{A}(Y, K, \alpha) \rightarrow \mathbb{R},$$ 
which we will describe its construction in the subsection \ref{holp}. 
Let ${\rm Crit}_{\pi}$ be a critical point set of $CS+f_{\pi}$ and ${\rm Crit}^{*}_{\pi}={\rm Crit}_{\pi }\cap \mathcal{A}^{*}(Y, K, \alpha)$.
Their orbit of gauge transformations are denoted by $\mathfrak{C}_{\pi}(Y, K, \alpha)$ and $\mathfrak{C}^{*}_{\pi}(Y, K, \alpha)$. 

 We define (perturbed) topological energy $\mathcal{E}_{\pi}(\gamma)$ of a path $\gamma: [0, 1]\rightarrow \mathcal{A}(Y, K, \alpha)$ as
 
 \begin{equation}
\mathcal{E}_{\pi}(\gamma)=2\{(CS+f_{\pi})(\gamma(1))-(CS+f_{\pi})(\gamma(0))\}.\label{energy}
\end{equation}
 We also define (perturbed) Hessian of $A\in \mathcal{A}(Y, K, \alpha)$ as
 $${\rm Hess}_{A,\pi}(a)=*d_{A}a+DV_{\pi}|_{A}(a),$$where $V_{\pi}$ is a gradient of $f_{\pi}$, and $DV_{\pi}|_{A}$ is its derivative at $A$.

  For each $A\in \mathcal{A}(Y, K, \alpha)$, we can regard Hessian as an operator, 
  $${\rm Hess}_{A,\pi}: \check{L}^{2}_{m, A^{\alpha}}({Y}\setminus K, {\Lambda }^{1}\otimes{\mathfrak{g}}_{E})\rightarrow \check{L}^{2}_{m-1, A^{\alpha}}({Y}\setminus K,{ \Lambda}^{1}\otimes {\mathfrak{g}}_{E}). $$
 \begin{dfn}
  $A\in {\rm Crit}^{*}_{\pi}$ is called non-degenerate if ${\rm Hess}_{A,\pi}|_{{\rm Ker}(d_{A}^{*})}$ is invertible.
 \end{dfn}
  This means that Hessian is non-degenerate to the vertical direction of the gauge orbit.
 For irreducible critical points of unperturbed Chern-Simons functional, there is a following criterion for the non-degeneracy condition

\begin{prp}{\rm ( \cite[Lemma 3.13]{KM11})}\label{grcoh non-deg}
A critical point  $A\in {\rm Crit}^{*}$ is non-degenerate if only if the kernel of the map
$$H^{1}(Y\setminus K; {\rm ad}{\rho})\rightarrow H^{1}(\mu_{K}; {\rm ad}{\rho} )$$is zero, where this map is induced by the natural embedding $\mu_{K}\hookrightarrow Y\setminus K$ and $\rho: \pi_{1}(Y\setminus K)\rightarrow SU(2)$ is a representation corresponding to a flat connection $A$.
\end{prp}
We say that $[A]\in \mathfrak{C}_{\pi}$ is non-degenerate if one of its representatives $A\in {\rm Crit}_{\pi}$ (and hence all) are non-degenerate.

The non-degeneracy condition at the reducible critical point  is given by a constrain on the holonomy parameter. 
Let $\theta_{\alpha}$ be a gauge equivalence class of reducible flat connection corresponding to a conjugacy class of $SU(2)$-representation of $\pi_{1}(Y\setminus K)$ which factors through the abelianization $H_{1}(Y\setminus K, \mathbb{Z})$ and has a holonomy parameter $\alpha$.
Since $Y$ is an integral homology $3$-sphere, such $\theta_{\alpha}$ uniquely exists.
The following is obtained as a corollary of \cite[Lemma 3.13]{KM11}
\begin{prp}{\rm (\cite[Lemma 15]{E19})}\label{non-deg red}
The unique flat reducible $\theta_{\alpha}$ is isolated and non-degenerate if only if $\Delta_{(Y, K)}(e^{4\pi i \alpha})\neq 0$. 
\end{prp}
Let us fix the definition of the Chern-Simons functional.
We fix a reducible flat connection $\tilde{\theta}_{\alpha}$ which represents $\theta_{\alpha}$ and put a condition $CS(\tilde{\theta}_{\alpha})=0$.
Then the $\mathbb{R}$-valued functional $CS$  is determined up to the choice of a representative of $\theta_{\alpha}$.
From now on, we fix a representative $\tilde{\theta}_{\alpha}$ for each pair $(Y, K)$.

\subsection{The flip symmetry}
The flip symmetry is an involution which acts on a family of configuration spaces $\bigcup_{\alpha\in(0,\frac{1}{2})\cap\mathbb{Q}}\mathcal{B}(Y, K, \alpha)$. 
The flip symmetry changes holonomy conditions as $\alpha\mapsto \frac{1}{2}-\alpha$. 
The four dimensional version is introduced in \cite{KM93}, and the three dimensional version is similarly defined in \cite{DS1}.  
We generalize the three dimensional version of flip symmetry as follows. 
Let $\chi\in H^{1}(Y\setminus K, \mathbb{Z}_{2})\cong\mathbb{Z}_{2}$ be a generator. Since $H^{1}(Y\setminus K,\mathbb{Z}_{2})={\rm Hom}(\pi_{1}(Y\setminus K), \mathbb{Z}_{2})$,  we can regard $\chi$ as a representation $\chi: \pi(Y\setminus K)\rightarrow \mathbb{Z}_{2}$.
  The representation $\chi$ satisfies $\chi(\mu_{K})=-1$. 
  This representation forms a flat line bundle  $L_{\chi}$ over $Y\setminus K$ with a flat connection corresponding to $\chi$. Since $L_{\chi}$ is a trivial line bundle and there is an isomorphism $E|_{Y\setminus K}\cong E|_{Y\setminus K}\otimes L_{\chi}$, we regard a connection $A\otimes \chi$ on $ E|_{Y\setminus K}\otimes L_{\chi}$ as a connection  on $E|_{Y\setminus K}$. Thus the action of $\chi \in H^{1}(Y\setminus K, \mathbb{Z}_{2})$ on $\bigcup_{\alpha\in(0,\frac{1}{2})\cap\mathbb{Q}}\mathcal{B}(Y, K, \alpha)$ is defined by
  $$\chi[A]=[A\otimes \chi].$$
  This action is called the flip symmetry. 
  The flip symmetry gives the identification  \[\mathcal{B}(Y,K,\alpha)\cong \mathcal{B}(Y, K,\frac{1}{2}-\alpha),\] in particular, it defines the involution on $\mathcal{B}(Y, K,\frac{1}{4})$.
  
  The flip symmetry can be restricted to the space $\bigcup_{\alpha\in (0, \frac{1}{2})\cap\mathbb{Q}}\mathcal{R}_{\alpha}(Y\setminus K, SU(2))$.
  In this case, the action of $\chi \in H^{1}(Y\setminus K, \mathbb{Z}_{2})$ is simply described as 
  $\chi[\rho]=[\rho\cdot \chi]$ where $ \rho\cdot \chi : \pi_{1}(Y\setminus K)\rightarrow SU(2)$ is an $SU(2)$-representation defined as $(\rho\cdot \chi)(g):=\rho(g)\chi(g)$ for $g\in \pi_{1}(Y\setminus K)$.
  If $\rho$ satisfies \[\rho(\mu_{K})\sim
  \left[\begin{array}{cc}
  e^{2\pi i \alpha}&0\\
  0&e^{-2\pi i \alpha}
  \end{array}\right]
  \]
  then 
 \[(\rho\cdot \chi)(\mu_{K})\sim\left[
 \begin{array}{cc}
 e^{2\pi i (\frac{1}{2}-\alpha)}&0\\
 0&e^{-2\pi i (\frac{1}{2}-\alpha)}
 \end{array}\right].
 \]

\subsection{Holonomy perturbations\label{holp}}
In this subsection, we review the constructions and properties of the perturbation term of the Chern-Simons functional introduced by Kronheimer-Mrowka \cite{KM11} and also used in the construction by Daemi-Scaduto \cite{DS1}.  

Let $q: S^{1}\times D^{2}\rightarrow Y\setminus K $ be a smooth immersion of a solid torus. $(s, z)\in S^{1}\times D^{2}$ denotes its coordinate regarding $S^{1}$ as $\mathbb{R}/\mathbb{Z}$ and $D^{2}$ as the unit disk in $\mathbb{C}$. Let $G_{E}\rightarrow Y$ be a bundle of group whose sections are gauge transformations of $E$. This is defined by $G_{E}=P\times_{SU(2)} SU(2)$  where $P$ is a corresponding $SU(2)$-bundle and $SU(2)$ act as obvious way on $P$ and conjugate on $SU(2)$-factor. ${\rm Hol}_{q}(A): D^{2}\rightarrow q^{*}(G_{E}) $ is a section of $q^{*}(G_{E})$ which assigns the holonomy ${\rm Hol}_{q(-, z)}(A)$ of connection $A\in \mathcal{A}(Y, K, \alpha)$ along the loop $q(-, z): S^{1}\rightarrow Y\setminus K$ to each $z\in D^{2}$. 
Next, we repeat above constructions for $r$-tuple of  smooth immersions of solid torus.
$$\mathbf{q}=(q_{1}, \cdots, q_{r}).$$
Assume that  there is a positive number $\eta >0$ such that 
\begin{equation}
q_{i}(s, z)=q_{j}(s, z)\  {\rm for\  all}\  (s, z)\in [-\eta, \eta]\times D^{2}. \label{cond1}
\end{equation}
Then there are identifications,
$$q_{i}^{*}(G_{E})\cong q_{j}^{*}(G_{E})$$
over $[-\eta, \eta]\times D^{2}$. 
$q^{*}(G^{r}_{E})$ denotes fiber product of $q_{1}^{*}(G_{E}),\cdots , q_{r}^{*}(G_{E})$ over $[-\eta, \eta]\times D^{2}$. Then we can construct a section  ${\rm Hol}_{\mathbf{q}}(A): D^{2}\rightarrow q^{*}(G_{E}^{r})$ which assigns $$({\rm Hol}_{q_{1}(-,z)}(A),\cdots , {\rm Hol}_{q_{r}(-, z)}(A)) \in SU(2)^{r}$$ for each $z \in D^{2}$. 
Next, we choose a smooth function on $SU(2)^{r}$ which is invariant under  the diagonal action of $SU(2)$ on $SU(2)^{r}$ by the adjoint action on each factors. Then this smooth function induces $h: q^{*}(G_{E})\rightarrow \mathbb{R}$.
\begin{dfn}
${\rm Hol}_{\mathbf{q}}(A)$ and $h$ are as above. Let $\mu$ be a $2$-form on $D^{2}$ such that $\int_{D^{2}}\mu =1$.
A smooth function $f: \mathcal{A}(Y, K, \alpha) \rightarrow \mathbb{R}$  of the form
$$f(A)=\int _{D^{2}}h({\rm Hol}_{\mathbf{q}}(A))\mu$$
is called a cylinder function.
\end{dfn}
Cylinder functions are determined by choices of $r$-tuples $\mathbf{q}$ and functions $h$. Note that the construction of cylinder functions is gauge invariant.
Let $\mathcal{P}$ be space of perturbations.  See \cite{KM11} for the detailed description.
For each $\pi \in \mathcal{P}$, we can associate a cylinder function $f_{\pi}$.
We call $f_{\pi}$ the holonomy perturbation and  $CS+f_{\pi}$ the perturbed Chern-Simons functional.

\begin{prp}\label{non-deg}
There is a residual subset of the Banach space of perturbations $\mathcal{P}'\subset\mathcal{P}$ such that, for  any sufficiently small $\pi \in \mathcal{P}'$, the perturbed Chern-Simons functional $CS+f_{\pi}$ has non-degenerate critical point set ${\rm Crit}^{*}_{\pi}$ and its image $\mathfrak{C}^{*}_{\pi}$ in $\mathcal{B}^{*}(Y, K,\alpha)$ is a finite point set.
Moreover,  the reducible critical point $\theta_{\alpha}$ is unmoved under the perturbation and is non-degenerate if $\Delta_{(Y, K)}(e^{4\pi i \alpha})\neq 0$.
\end{prp}
\begin{proof}
The finiteness property follows from a similar argument in \cite[Lemma 3.8]{KM11}.
The non-degeneracy condition follows from the fact that $f_{\pi}$ is dense in $C^{\infty}(S)$ for any compact finite dimensional submanifold $S\subset \mathcal{B}^{*}(Y, K,\alpha)$. See \cite[Section 5]{Don} for the details.
The argument in \cite[Subsection 2.4]{DS1} is  adapted to show that, for suitable choices of $SU(2)$-invariant smooth function $h$, the unique flat reducible $\theta_{\alpha}$ is unmoved under small perturbations.
By Proposition \ref{non-deg red}, the unique flat reducible $\theta_{\alpha}$ is still isolated and non-degenerate for such perturbations under the condition $\Delta_{(Y, K)}(e^{4\pi i \alpha})\neq 0$. 
\end{proof}

\subsection{The moduli space over the cylinder}
We discuss trajectories for perturbed gradient flow.
Let $(Z, S)=\mathbb{R}\times (Y, K)$ be a cylinder equipped with a product metric $g^{\nu}_{Y}+dt$. We introduce moduli spaces of  instantons over a cylinder. $\mathbb{E}$ denotes the pullback of the $SU(2)$-bundle $E\rightarrow Y$ by the projection $\mathbb{R}\times Y\rightarrow Y$. Consider  a connection $A$ on $\mathbb{E}$ of the form $A=B(t)+Cdt$, where $B({t})$ is a $t$-dependent singular connection on $Y\setminus K$.
Let $\beta_{0}$ and $\beta_{1}$ be elements in $\mathfrak{C}^{*}_{\pi}$. Let $B_{0}$ and $B_{1}$ be their representative in gauge equivariant classes.
Consider a singular connection $A_{0}$ over a cylinder $(Z, S)$ such that
$$A_{0}|_{(Y\setminus K)\times \{t\}}=B_{1}\ {\rm for}\ t \gg0,$$
$$A_{0}|_{(Y\setminus K)\times\{t\}}=B_{0}\  {\rm for}\  t\ll0.$$
$A_{0}$ defines a path $\gamma: \mathbb{R}\rightarrow \mathcal{B}(Y, K, \alpha)$ by sending $t$ to $[B(t)]$. $z$ denotes its relative homotopy class in $\pi_{1}(\mathcal{B}(Y, K, \alpha);\beta_{0}, \beta_{1})$.

Then we define a space of singular connection indexed by $z$.
$$\mathcal{A}_{z}(Z, S; B_{0}, B_{1})=\{A| A-A_{0} \in \check{L}^{2}_{m,A_{0}}({Z}\setminus S, \Lambda^{1}\otimes \mathfrak{g}_{\mathbb{E}})\}$$We also define the group of  gauge transformations,
$$\mathcal{G}_{z}(Z, S)=\{g\in {\rm Aut}(\mathbb{E})|\nabla^{k}_{A_{0}}g\in \check{L}^{2}(Z\setminus S, {\rm End}(\mathbb{E})),\ k=1, \cdots , m+1\}.$$

The group $\mathcal{G}_{z}(Z, S)$ acts on $\mathcal{A}_{z}(Z, S)$. Taking the quotient, we get the configuration space
$\mathcal{B}_{z}(Z, S; \beta_{0}, \beta_{1})$. 
We introduce the moduli space of (perturbed) instantons over a cylinder associated to the perturbed Chern-Simons functional $CS+f_{\pi}$. This is the moduli space of solutions of perturbed ASD-equation
$$M^{\pi}_{z}(\beta_{0}, \beta_{1})=\{[A]\in \mathcal{B}_{z}(Z, S;\beta_{0}, \beta_{1})| F^{+}_{A}+\hat{V_{\pi}}(A)=0\}.$$
Here $\hat{V}_{\pi}$ is a term arising from the perturbation $f_\pi $. 
The perturbed version of ASD-complex is given by
$$\Omega^{0}({Z\setminus S},\mathfrak{g}_{\mathbb{E}})\xrightarrow{d_{A}}\Omega^{1}({Z\setminus S},\mathfrak{g}_{\mathbb{E}}) \xrightarrow{d_{A}^{+}+D\hat{V}_{\pi}}\Omega^{+}({Z\setminus S},\mathfrak{g}_{\mathbb{E}}). $$
We consider a Fredholm operator 
$${D}_{A, \pi}: \check{L}^{2}_{m, A_{0}}({Z}\setminus S, \Lambda^{1}\otimes\mathfrak{g}_{\mathbb{E}})\rightarrow \check{L}^{2}_{m-1, A_{0}}({Z}\setminus S, (\Lambda^{0}\oplus \Lambda^{+})\otimes\mathfrak{g}_{\mathbb{E}})$$
given by ${D}_{A,\pi}=-d_{A}^{*}\oplus (d^{+}_{A}+D\hat{V}_{\pi})$.
We define the relative $\mathbb{Z}$-grading for $\beta_{0}, \beta_{1}\in\mathfrak{C}^{*}_{\pi}(Y, K, \alpha)$ as follows:
$${\rm gr}_{z}(\beta_{0}, \beta_{1})= {\rm ind}({D}_{A,\pi}),$$
where $z$ is a path represented by $A$.
Note that ${\rm ind}D_{A, \pi}$ is independent of the choice of the perturbation $\pi$ since the term $D\hat{V}_{\pi}$ is a compact operator.
The following proposition gives the well-defined mod 4 grading on the critical point set.
\begin{prp}{\rm ( \cite[Lemma3.1]{KM11})}
Let $z\in \pi_{1}(\mathcal{B}(Y, K,\alpha))$ be a homotopy class represented by a path which connects $B$ and $g^{*}(B)$, where $\beta=[B]$ and $d(g)=(k, l)$.
For $\beta\in \mathfrak{C}^{*}$ and a homotopy class $z\in \pi_{1}(\mathcal{B}(Y, K, \alpha); \beta)$, we have $${\rm gr}_{z}(\beta, \beta)=8k+4l.$$
\end{prp}

The mod 4 of $\mathbb{Z}$-grading is independent of the choice of the homotopy class $z$ and hence we can write
$${\rm gr}(\beta_{0}, \beta_{1})\equiv{\rm gr}_{z}(\beta_{0}, \beta_{1})\ {\rm (mod 4)}.$$
We also define the absolute $\mathbb{Z}$-grading by 
$${\rm gr}_{z}(\beta, \theta_{\alpha})={\rm ind}({D}_{A,\pi}:\phi\check{L}^{2}_{m }\rightarrow \phi\check{L}^{2}_{m-1}),$$
where $\phi\check{L}^{2}_{m}$ is a weighted Sobolev space with a weight function $\phi$ which agree with $e^{-\delta|t|}$ over two ends of the cylinder.  Here $\delta>0$ is chosen small enough. 
Similarly, we can define the mod 4 grading
$${\rm gr}(\beta)\equiv{\rm gr}_{z}(\beta, \theta_{\alpha}) \ {\rm (mod 4)}.$$
The moduli space $M^{\pi}_{z}(\beta_{0},\beta_{1})$ is called regular if the operator ${D}_{A,\pi}$ is surjective for all $[A]\in M^{\pi}_{z}(\beta_{0}, \beta_{1})$. For a generic choice of a perturbation, the moduli space  $M^{\pi}_{z}(\beta_{0}, \beta_{1})$ is regular and smooth manifold of dimension ${\rm gr}_{z}(\beta_{0}, \beta_{1})$. 
We explicitly write $M^{\pi}_{z}(\beta_{0}, \beta_{1})_{d}$  if the moduli space $M^{\pi}_{z}(\beta_{0}, \beta_{1})$ has dimension $d$.
We also write $\breve{M}^{\pi}_{z}(\beta_{0}, \beta_{1})_{d-1}=M^{\pi}_{z}(\beta_{0}, \beta_{1})_{d}/\mathbb{R}$.
The argument in \cite[Section 5]{Don} is adapted to our situation, and we have the following properties:
\begin{prp}\label{regular pert}
Let $\alpha\in (0, \frac{1}{2})$ be a holonomy parameter with $\Delta_{K}(e^{4\pi i \alpha})\neq 0$ and $\pi_{0}\in \mathcal{P}'$ be a small perturbation such that $\mathfrak{C}_{\pi_{0}}=\{\theta_{\alpha}\}\sqcup \mathfrak{C}^{*}_{\pi_{0}}$ consists of finitely many non-degenerate points. Then there is a small perturbation $\pi \in \mathcal{P}'$ such that
\item[\rm (i)]$f_{\pi}=f_{\pi_{0}}$ in a neighborhood of $\mathfrak{C}_{\pi_{0}}$,
\item[\rm (ii)] $\mathfrak{C}_{\pi}=\mathfrak{C}_{\pi_{0}}$,
\item[\rm (iii)] $M^{\pi}_{z}(\beta_{1}, \beta_{2})$ is regular for all homotopy class $z$ and critical points $\beta_{1}, \beta_{2}$.
\end{prp}
\begin{proof}
First, we fix a perturbation $\pi_{0}\in \mathcal{P}$ as in Proposition \ref{non-deg}.
Then, for each homotopy class $z$, we can find a perturbation $\pi_{z}\in \mathcal{P}$ which is supported away from critical points and  the corresponding moduli space is regular.
This essentially follows from the argument in \cite[Section 5]{Don}.
Since the subset $\mathcal{P}_{z}$ of regular perturbations as above forms open dense subset in $\mathcal{P}$, we can find a desired perturbation  $\pi$ in the countable intersection $\cap_{z}\mathcal{P}_{z}$.
\end{proof}
From now on, we assume that perturbation $\pi \in \mathcal{P}$ is always taken so that it satisfies properties in Proposition \ref{regular pert} and we drop $\pi$ from the notation $M^{\pi}_{z}(\beta_{1}, \beta_{2})$.

\subsection{Compactness}
Consider a relative homotopy class $z\in \pi_{1}(\mathcal{B}(Y,K, \alpha), \beta_{1}, \beta_{2})$.
If $\beta_{1}=\beta_{2}$ then we assume that $z$ is a non-trivial homotopy class. 
Elements in $\breve{M}_{z}(\beta_{1}, \beta_{2})$ are called unparamtrized trajectories. 
\begin{dfn}
A collection of unparametrized trajectories $([A_{1}], \cdots, [A_{l}]) \in \breve{M}_{z_{1}}(\beta_{1}, \beta_{2})\times \cdots \times \breve{M}_{z_{l}}(\beta_{l-1}, \beta_{l})$ is called broken trajectories from $\beta_{1}$ to $\beta_{l}$. If the composition of paths $z_{1}\circ\cdots \circ z_{l}$ is contained in the homotopy class $z$, then $\breve{M}^{+}_{z}(\beta_{1}, \beta_{l})$ denotes the space of unparametrized broken trajectories from $\beta_{1}$ to $\beta_{l}$.
\end{dfn}
The compactness property of moduli spaces are follows. 
See also \cite[Proposition 3.22]{KM11}. 
\begin{prp}
Let $\beta_{1}, \beta_{2}\in \mathfrak{C}_{\pi}$ and assume that ${\rm dim}{M}_{z}(\beta_{1}, \beta_{2})<4$. Then the space of unparametrized broken trajectories $\breve M^{+}_{z}(\beta_{1}, \beta_{2})$ is compact.
\end{prp}
We can assign the energy $\mathcal{E}_{\pi }(z)$ to a homotopy class $z$. 
In  singular gauge theory for general holonomy parameters, the counting $\# \bigcup_{z} \breve{M}_{z}(\beta_{1}, \beta_{2})$ with ${\rm gr}_{z}(\beta_{1}, \beta_{2})=1$ can be infinite.  Instead, we use the following finiteness result. 
\begin{prp}{\rm (\cite[Proposition 3.23]{KM11})}
For a given constant $C>0$, there are only finitely many critical points $\beta_{1}, \beta_{2}$ and homotopy classes $z\in \pi_{1}(\mathcal{B}; \beta_{1}, \beta_{2})$ such that the moduli space $M_{z}(\beta_{1}, \beta_{2})$ is non-empty and $\mathcal{E}_{\pi}(z)<C$. 
\end{prp}
Thus $$\bigcup_{\substack{z\\ \mathcal{E}_{\pi}(z)<C}}\breve{M}_{z}(\beta_{1}, \beta_{2})_{0}$$is a finite point set for any $C>0$.

The gluing formula of index tells us 
\begin{equation}
{\rm gr}_{z}(\beta_{0}, \theta_{\alpha})+1+{\rm gr}_{z'}(\theta_{\alpha}, \beta_{1})={\rm gr}_{z'\circ z}(\beta_{0},\beta_{1})\label{adding formula for grading}
\end{equation}
 since $\theta_{\alpha}$ has a stabilizer $S^{1}$. 
From this relation, we conclude that any broken trajectories in $M^{+}_{z}(\beta_{0}, \beta_{1})_{d}$ do not factor through $\theta_{\alpha}$ if the dimension of $M_{z}(\beta_{0}, \beta_{1})_{d}$ is less than $3$.

\subsection{Cobordisms}\label{cobordism}
Let $(W, S)$ be a pair of oriented 4-manifold and embedded oriented surface such that $\partial W=Y'\sqcup(-Y)$ and $\partial S=K\sqcup K'$.
We call  $(W, S)$ the cobordism of pairs and write $(W, S):(Y, K)\rightarrow (Y', K')$.
Set 
\[(W^{+}, S^{+}):=\mathbb{R}_{\leq 0}\times (Y, K)\cup (W, S)\cup\mathbb{R}_{\geq 0} \times (Y', K').\]
We fix a metric on $W^{+}\setminus S^{+}$ with a cone angle $\frac{2\pi}{\nu}$ and cylindrical forms on each ends. 
Let $\beta\in \mathcal{B}(Y, K, \alpha)$ and $\beta'\in\mathcal{B}(Y', K')$ be given connections and  we choose a singular $SU(2)$-connection $A_{0}$ on $(W^{+}, S^{+})$ which has a limiting connection $\beta, \beta'$(up to gauge transformations) on each ends of $(W^{+}, S^{+})$. $z$ denotes a homotopy class of $A$.
We define a space of connections and the group of gauge transformations as follows:
\[\mathcal{A}_{z}(W, S;\beta,\beta'):=\{A|A-A_{0}\in \check{L}^{2}_{m-1, A_{0}}(W^{+}\setminus S^{+},\mathfrak{g}_{E}\otimes \Lambda^{1}) \},\]
\[\mathcal{G}_{z}(W, S):=\{g\in {\rm Aut}(E)|\nabla_{A_{0}}^{i}g\in\check{L}^{2}(W^{+}\setminus S^{+}; {\rm End}(E)), i=1\cdots m \}.\]
We also define the quotient 
\[\mathcal{B}_{z}(W, S;\beta, \beta'):=\mathcal{A}_{z}(W, S;\beta, \beta')/\mathcal{G}_{z}(W, S).\]
$\mathcal{B}(W, S;\beta, \beta')$ denotes the union of $\mathcal{B}_{z}(W, S;\beta, \beta')$ for all paths.
The perturbed ASD equation on $(W, S)$ has the form $F_{A}^{+}+U_{\pi_{W}}=0$ where $U_{\pi_{W}}$ is a $t$-dependent perturbation. More concretely this can be described as the following. (The argument is based on \cite{KM11}.)
Let $\pi,\pi_{0}\in \mathcal{P}_{Y}$ be two holonomy perturbations on $\mathbb{R}\times (Y, K)$. The perturbed ASD equation on $\mathbb{R}_{\leq 0}\times (Y, K)$ has the following form
\[F_{A}^{+}+\psi(t)\hat{V}_{\pi}+\psi_{0}(t)\hat{V}_{\pi_{0}}=0\]
 where  $\psi(t)$ is a smooth cut-off function such that $\psi(t)=1$ if $t<-1$ and $\psi(t)=0$ at  $t=0$. $\psi_{0}$ is a smooth function supported on $(-1, 0)\times Y$.
We choose $\pi\in \mathcal{P}$ so that $\mathfrak{C}_{\pi}$ satisfies properties in Proposition \ref{non-deg} and Proposition \ref{regular pert}. The perturbation terms can be described similarly on another end.
For generic choices of $\pi_{0}$ and $\pi_{0}'\in \mathcal{P}_{Y'}$, the irreducible part of the perturbed ASD-moduli space 
\[M_{z}(W, S; \beta, \beta')\subset \mathcal{B}_{z}(W, S;\beta, \beta')\]is a smooth manifold.
Consider the ASD-operator
\begin{equation}{D}_{A}=-d^{*}_{A}\oplus d^{+}_{A}:\phi\check{L}^{2}_{m,A_{0}}(W\setminus S, \mathfrak{g}_{{E}}\otimes\Lambda^{1})\rightarrow \phi\check{L}^{2}_{m-1,A_{0}}(W\setminus S, \mathfrak{g}_{{E}}\otimes(\Lambda^{0}\oplus \Lambda^{+}))\label{ASD-ope on cob}
\end{equation}
where $\phi$ is a weight function.  
If one of limiting connection of $A_{0}$ is irreducible then we choose $\phi\equiv1$ on that end of $(W^{+}, S^{+})$.
If  $A_{0}$ has a reducible limiting connection then we choose $\phi= e^{-\delta|t|}$ on that end, where $\delta>0$ is small enough.
$M(W, S;\beta, \beta')_{d}$ denotes the union of the moduli spaces $M_{z}(W, S;\beta, \beta')$ with ${\rm ind}D_{A}=d$.

\begin{dfn}

We define the  topological energy of $A\in \mathcal{B}(W,S;\beta, \beta')$ as
$$\kappa(A):=\frac{1}{8\pi^{2}}\int_{W^{+}\setminus S^{+}}{\rm Tr}(F_{A}\wedge F_{A}),$$
and the monopole number of $A$ as
$$\nu(A):=\frac{i}{\pi}\int_{S^{+}}\Omega-2\alpha S\cdot S.$$where \[F_{A}|_{S^{+}}=\left[\begin{array}{cc}
\Omega&0\\
0&-\Omega
\end{array}\right].
\]
\end{dfn}
For the cylinder $(W, S)=[0, 1]\times (Y, K)$ and trivial perturbation $\pi=0$, the topological energy $\kappa$ is related to energy $\mathcal{E}$ of Chern-Simons functional as $2\kappa(A)=\mathcal{E}(A)$.
Consider an $SU(2)$ connection $B$ on $(Y, K)$ and a connection $A$ over the cylinder $\mathbb{R}\times(Y, K)$ which asymptotic to $B$ at  and a fixed reducible flat connection $\tilde{\theta}_{\alpha}$ such that $CS(\tilde{\theta}_{\alpha})=0$.
Then $CS(B)=\kappa(A)$ by the construction.

Similarly we define an $\mathbb{R}$-valued function ${\rm hol}_{K}:\mathcal{A}(Y, K,\alpha)\rightarrow \mathbb{R}$ as follows;
\begin{dfn}
Let $A$ be an $SU(2)$ connection over the cylinder $[0, 1]\times (Y, K)$ as above. We define
${\rm hol}_{K}(B):=-\nu(A).$
\end{dfn}

If $z$ is a path on $(W, S)$ which is represented by a connection $A$, then we write $\kappa(z)$ for $\kappa(A)$ and $\nu(z)$ for $\nu(A)$ since these numbers are independent of the choice of $A$.

Let $(X, \Sigma)$ be a pair of four-manifold and embedded surface with boundary $\partial X=Y$, $\partial \Sigma=K $ where $K$ is an oriented knot in an oriented integral homology 3-sphere $Y$.  We assume that $[\Sigma]=0$. Let $\Theta_{\alpha}$ be a singular flat reducible  connection with a holonomy parameter $\alpha=\frac{n}{m}$ and whose  lift to the $m$-fold  cyclic branched covering $\tilde{X}_{m}$ is a trivial connection.
We  write $H^{i}(X\setminus \Sigma; \Theta_{\alpha})$ for $i$-th cohomology of $X\setminus \Sigma$ with the local coefficient system twisted by $\Theta_{\alpha}$. 
\begin{lem}\label{sig formula}
We define $\chi(X\setminus \Sigma;\Theta_{\alpha})=\sum_{i}(-1)^{i}{\rm dim}H^{i}(X\setminus \Sigma; \Theta_{\alpha})$ and $\sigma(X\setminus \Sigma;\Theta_{\alpha})={\rm dim}H^{+}(X\setminus \Sigma; \Theta_{\alpha})-{\rm dim}H^{-}(X\setminus \Sigma;\Theta_{\alpha})$. 

Then
\[\chi(X\setminus \Sigma; \Theta_{\alpha})=\chi(X)-\chi(\Sigma),\]
\[\sigma(X\setminus\Sigma; \Theta_{\alpha})=\sigma(X)+\sigma_{\alpha}(Y, K).\]

\end{lem}
\begin{proof}
Consider a rational holonomy parameter of the form $\alpha=\frac{n}{m}\in\mathbb{Q}$. We take a $m$-fold branched covering $\pi:\tilde{X}_{m}\rightarrow X$ whose branched locus is $\Sigma$. 
The pull-back of singular flat connection $\Theta_{\alpha}$ extends as a trivial flat connection over $\tilde{X}_{m}$. Let $\tau:\tilde{X}_{m}\rightarrow \tilde{X}_{m}$ be a covering transformation. Then its induced action $\tilde{\tau}$ on the pulled-back bundle $ \underline{\mathbb{C}}$ is the multiplication by $e^{\frac{4\pi i n}{m}}$. 
Since the index of the twisted de Rham operator \[d_{\Theta_{\alpha}}+d^{*}_{\Theta_{\alpha}}:\Omega^{\rm even}(X\setminus \Sigma;\Theta_{\alpha})\rightarrow \Omega^{\rm odd}(X\setminus \Sigma; \Theta_{\alpha})\]
coincides with the index of \begin{equation}d+d^{*}:\Omega^{\rm even}(\tilde{X}_{m};\underline{\mathbb{C}})^{\tilde{\tau}}\rightarrow \Omega^{\rm odd}(\tilde{X}_{m};\underline{\mathbb{C}})^{\tilde{\tau}}\label{equiv de rham}\end{equation}
where $\Omega^{*}(\tilde{X}_{m};\underline{\mathbb{C}})^{\tilde{\tau}}=\{\omega\in \Omega^{*}(\tilde{X}_{m};\underline{\mathbb{C}})|\omega(\tau(x))=\tilde{\tau}(\omega(x))\}$. 
The index of (\ref{equiv de rham}) is given by $\chi(X)-\chi(\Sigma)$. 
This can be seen by taking cell complex $C_{*}(\tilde{X}_{m})$ of $\tilde{X}_{m}$ in $\tau$-equivariant way. Then there are decompositions
\[C_{*}(\tilde{X}_{m})=C_{*}(\Sigma)\oplus C_{*}(\tilde{X}_{m}\setminus \Sigma),\]
\[C_{*}(\tilde{X}_{m}\setminus \Sigma)=\bigoplus_{i=1}^{n}C_{i},\]where each $C_{i}$ is isomorphic to a copy of $C_{*}(X\setminus \Sigma)$. 
Since $\tau_{*}$ acts as the identity on $C_{*}(\Sigma)$ and a cyclic way on $C_{*}(\tilde{X}_{m}\setminus \Sigma)=\bigoplus_{i=1}^{n}C_{i}$, all eigenspaces  of the action $\tau_{*}$ on $C_{*}(\tilde{X}_{m}\setminus \Sigma)$ are isomorphic. 
On the other hand, there is the identity $\chi(\tilde{X}_{m})=m\chi(X)+\chi(\Sigma)$. 
Thus the $\tau$-invariant  index of de Rham operator is given by $\chi(X)-\chi(\Sigma)$.

Similarly, the index of  signature operator twisted by local coefficient $\Theta_{\alpha}$ coincides with the index of the
signature operator over $\tilde{X}_{m}$ which is restricted to $e^{\frac{4\pi i n}{m}}$ -eigenspaces. 
Such signature is equal to $\sigma(X)+\sigma_{\frac{n}{m}}(Y, K)$ by the formula in \cite{Viro}.
\end{proof}
\begin{prp}\label{index formula}
Let $(W, S):(Y, K)\rightarrow (Y', K')$ be a cobordism of pairs  and $[A]$ be an element of $\mathcal{B}(W,S ;\theta_{\alpha}, \theta'_{\alpha})$.  
Then the index of the ASD operator ${D}_{A}$ is given by
$${\rm ind}{D}_{A}=8\kappa(A)+2(4\alpha-1)\nu(A)-\frac{3}{2}(\sigma(W)+\chi(W))+\chi(S)+8\alpha^2S\cdot S+\sigma_{\alpha}(Y, K)-\sigma_{\alpha}(Y', K')-1.$$
\end{prp}
\begin{proof}
Let $X$ be a compact four-manifold with $\partial X=Y$ and  $\Sigma\subset Y$ be a Seifert surface of a knot $K$. 
Pushing $\Sigma$ into the interior of $X$, we obtain a pair $(X,\Sigma )$ whose boundary is $(Y, K)$. Moreover $[\Sigma]=0$ in $H_{2}(X;\mathbb{Z})$. Similarly we can construct another pair $(X', \Sigma')$ such that $(\partial X', \partial\Sigma')=(Y', K')$. 

Set
\[(\bar{W}, \bar{S}):=(X, \Sigma)\cup_{(Y, K)}(W, S)\cup_{(Y', K')}(X', \Sigma').\]$(\bar{W}, \bar{S})$ is a closed pair of four-manifold and embedded surface. Let $A_{1}$ and $A_{2}$ be singular  flat reducible connections over $(X, \Sigma)$ and $(X, \Sigma')$ which are extensions of $\theta_{\alpha}$ and $\theta'_{\alpha}$ respectively. 
Let $A$ be a connection which represents an element of $\mathcal{B}(W, S; \theta_{\alpha}, \theta'_{\alpha})$. We consider a  connection $A'=A_{1}\#_{\theta_{\alpha}}A\#_{\theta'_{\alpha}}A_{2}$ over $(\bar{W}, \bar{S})$ obtained by the gluing.

Using the gluing formula for index, we have 
$${\rm ind}{D}_{A'}={\rm ind}{D}_{A_{1}}+{\rm ind}{D}_{A}+{\rm ind}{D}_{{A}_{2}}+2$$
where $A'$ is a singular connection on the closed pair $(\bar{W}, \bar{S})$ obtained by gluing $A_{1}$, $A_{2}$, and $A$.
Since $A_{1}$ is reducible, there is a decomposition  $A_{1}=\mathbf{1}\oplus B_{\alpha}$ with respect to a decomposition of adjoint bundle $\underline{\mathbb{R}}\oplus L^{\otimes 2}$, where $\mathbf{1}$ denotes the trivial connection. The deformation complex for ${D}_{A_{1}}$ decompose into the following two parts
\begin{equation}
\Omega^{0}(X)\xrightarrow{d}\Omega^{1}(X)\xrightarrow{d^{+}}\Omega^{+}(X) \label{trivial asd cpx}
\end{equation}
and
\begin{equation}
\Omega^{0}(X\setminus \Sigma; {\rm ad}{B_{\alpha}})\xrightarrow{d_{B_{\alpha}}}\Omega^{1}(X\setminus \Sigma;{\rm ad}{{B}_{\alpha}})\xrightarrow{d^{+}_{B_{\alpha}}}\Omega^{+}(X\setminus \Sigma; {\rm ad}{B_{\alpha}}).\label{asd w loc coeff}
\end{equation}
The index of (\ref{trivial asd cpx}) is given by $-\frac{1}{2}(\sigma(X)+\chi(X))-\frac{1}{2}$. On the other hand, the index of (\ref{asd w loc coeff}) is given by $-\sigma(X\setminus \Sigma; B_{\alpha})-\chi(X\setminus \Sigma; B_{\alpha})$. Using two formulae $\sigma(X\setminus \Sigma; B_{\alpha})=\sigma(X)+\sigma_{\alpha}(Y, K)$ and $\chi(X\setminus \Sigma,B_{\alpha})=\chi(X)-\chi(\Sigma)$ in Lemma \ref{sig formula}, we obtain
\[
{\rm ind}{D}_{A_{1}}=-\frac{3}{2}(\sigma(X)+\chi(X))-\sigma_{\alpha}(Y, K)+\chi(\Sigma)-\frac{1}{2}.
\]
Similarly we have 
\[{\rm ind}{D}_{A_{2}}=-\frac{3}{2}(\sigma(X')+\chi(X'))+\sigma_{\alpha}(Y',K')+\chi(\Sigma')-\frac{1}{2}
\]since $\sigma_{\alpha}(-Y', K')=-\sigma_{\alpha}(Y', K')$.
The index formula for a closed pair in \cite{KM93} gives
\[
{\rm ind}{D}_{\bar{A}}=8\kappa(\bar{A})+2(4\alpha-1)\nu(\bar{A})-\frac{3}{2}(\sigma(\bar{W})+\chi(\bar{W}))+\chi(\bar{S})+8\alpha^{2}S\cdot S+2.
\]Hence we have a desired formula
\[{\rm ind}{D}_{A}=8\kappa(A)+2(4\alpha-1)\nu(A)-\frac{3}{2}(\sigma(W)+\chi(W))+\chi(S)+8\alpha^2S\cdot S+\sigma_{\alpha}(Y, K)-\sigma_{\alpha}(Y', K')-1.\]

\end{proof}
\begin{rmk}
{\rm The index formula in Proposition \ref{index formula} recovers \cite[Lemma 2.26]{DS1} when $\alpha=
\frac{1}{4}$.}
\end{rmk}
Let $A_{L}$ be a reducible connection corresponding to a decomposition $E=L\oplus L^{*}$.
\begin{dfn}
We call $A_{L}$ minimal if it minimizes ${\rm ind}{D}_{A_{L}}$. 
\end{dfn}
Our definition of minimal reducible coincides with \cite[Subsection 2.3]{DS2} if $\alpha=\frac{1}{4}$.

Let us describe relations between $CS$ and $\kappa$, and $\nu$ and ${\rm hol}_{K}$ over cobordisms.
Consider a connection $A$ over a cobordism $(W, S):(Y,K)\rightarrow (Y', K')$ whose limiting  connections are $B$ on $(Y, K)$ and $B'$ on $(Y', K')$.
Then the following statement holds.
\begin{lem}\label{kappa and nu}
Fix a reducible connection $A_{L}$ over $(W, S)$. Then there exist $k, l\in \mathbb{Z}$ such that
\begin{eqnarray*}\kappa(A)-\kappa(A_{L})&=&CS(B)-CS(B')+k+2\alpha l,\\
\nu(A)-\nu(A_{L})&=& {\rm hol}_{K'}(B')-{\rm hol}_{K}(B)-2l.
\end{eqnarray*}
\end{lem}
\begin{proof}
Recall that $\mathbb{R}$-valued functions $CS$ and ${\rm hol}$ are fixed by choosing reducible flat connections $\tilde{\theta}_{\alpha}$ and $\tilde{\theta'}_{\alpha}$ over each pairs $(Y, K)$ and $(Y',K')$.
If we choose reducible connection $A_{L_{0}}$ so that it has two reducible limits  $\tilde{\theta}_{\alpha}$ and $\tilde{\theta'}_{\alpha}$, then we have
\begin{eqnarray*}
\kappa(A)+CS(B')&=&CS(B)+\kappa(A_{L_{0}})\\
\nu(A)-{\rm hol}_{K'}(B')&=&-{\rm hol}_{K}(B)+\nu(A_{L_{0}})
\end{eqnarray*}
by the construction.
If we change $A_{L}$ to other homotopy classes of reducible connections, terms $k+2\alpha l$ and $-2l$ appear by the gauge transformation.
\end{proof}
For a cobordism of pairs $(W, S)$ and fixed holonomy parameter $\alpha$,  we introduce real values $\kappa_{0}(W, S,\alpha)$ and $\nu_{0}(W, S,\alpha)$ as follows.
\begin{dfn}
\begin{eqnarray*}
\kappa_{0}(W, S,\alpha)&:=&\min\{\kappa(A_{L})\vert A_{L}\   \text{{\rm minimal reducible}}\}\\
\nu_{0}(W, S,\alpha)&:=&\begin{cases}
\nu(A_{L})\  \text{\rm  where }A_{L} \text{\rm  is a minimal reducible with } \kappa_{0}=\kappa(A_{L})& (\alpha\neq \frac{1}{4}),\\
\min\{\nu(A_{L})|A_{L} \text{{\rm minimal reducible}}\}&(\alpha=\frac{1}{4}).
\end{cases}\end{eqnarray*}
\end{dfn}
Note that the homotopy class of  path $z:\theta_{\alpha}\rightarrow \theta'_{\alpha}$ represented by a minimal reducible $A_{L}$ with $\kappa_{0}=\kappa(A_{L}) $ is uniquely determined when $\alpha\neq\frac{1}{4}$. 
If $\alpha=\frac{1}{4}$ then homotopy classes of paths represented by minimal reducibles are not unique but finitely exists.
In particular, $\nu_{0}(W, S,\alpha)$ is well-defined.
\begin{rmk}{\rm
If the cobordism of pairs $(W, S)$ has a flat minimal reducible with a holonomy parameter $\alpha$ then $\kappa_{0}(W, S, \alpha)=\nu_{0}(W, S,\alpha)=0$.}
\end{rmk}
We write $\kappa_{0}=\kappa_{0}(W, S,\alpha)$ and $\nu_{0}=\nu_{0}(W, S,\alpha)$ for short.

\begin{dfn}\label{negativedef}
Let $(W, S):(Y, K)\rightarrow (Y', K')$ be a cobordism of pairs where $ K$ and $K'$ are oriented knots in integral homology 3-spheres $Y$, $Y'$. Let $\mathscr{S}$ be an integral domain over $\mathscr{R}_{\alpha}$.
A cobordism of pairs  $(W, S)$ is called negative definite  over $\mathscr{S}$ if it satisfies the followings
\begin{enumerate}
\item $b^{1}(W)=b^{+}(W)=0$,
\item The index of the minimal reducibles are $-1$,
\item \[\eta^{\alpha}(W, S):=\sum_{A_{L}: {\rm minimal}}(-1)^{c_{1}(L)^{2}}\lambda^{\kappa_{0}-\kappa (A_{L})}T^{\nu(A_{L})-\nu_{0}}\]  is a non-zero element in $\mathscr{S}$.
\end{enumerate}
\end{dfn}
\begin{rmk}
{\rm Our definition of negative definite cobordism coincides with that of \cite{DS2} when $\alpha=\frac{1}{4}$ since instantons have minimal energy if only if they have minimal index.}
\end{rmk}
Let $(W_{1}, S_{1}):(Y_{1}, K_{1})\rightarrow (Y', K')$ and $(W_{2}, S_{2}):(Y', K')\rightarrow (Y_{2}, K_{2})$ be negative definite cobordisms. 
Note that their composition $(W_{2}\circ W_{1}, S_{2}\circ S_{1}):(Y_{1}, K_{1})\rightarrow (Y_{2}, K_{2})$  is also a negative definite cobordism.

A cylinder $[0, 1]\times (Y, K)$ and a homology concordance $(Y, K)\rightarrow (Y', K')$ are examples of negative definite cobordisms.
The following is also a basic example of negative definite cobordism.
Let $K_{+}$ be a knot in $S^3$ which has at least one positive crossing.
Let  $K_{-}$ be a knot which is obtained by replacing one positive crossing in the knot $K_{+}$ by one negative crossing. See Figure \ref{fig:Kpm}. 

\begin{figure}
    \centering
    \includegraphics[scale=0.55]{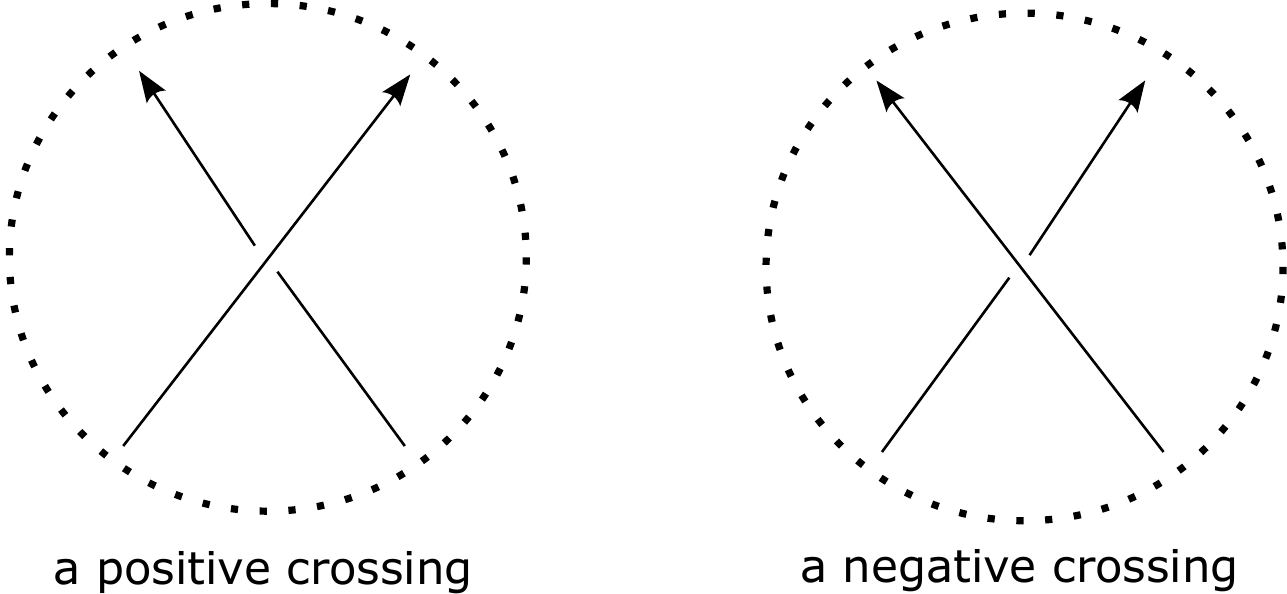}
    \caption{Crossings of knot}
    \label{fig:Kpm}
\end{figure}
 Since $S^{3}$ is simply connected, $K_{+}$ and $K_{-}$ are homotopic. Approximating homotopy from $K_{-}$ to $K_{+}$ by a smooth map, we get a smoothly  immersed surface $S\subset [0, 1]\times S^{3}$ such that $S\cap \{0\}\times S^{3}=K_{+}$ and $S\cap \{1\}\times S^{3}=K_{+}$. 
 Furthermore, we assume that $S$ has a transversal self-intersection point.
 Let ${S'}:K_{+}\rightarrow K_{-}$ be an inverse cobordism of $S$. 
 $S$ has a positive self-intersection point in $[0,1]\times S^{3}$. 
 Blowing up this self-intersection point, we obtain a new cobordism of pairs
 \begin{eqnarray}(\overline{\mathbb{CP}^{2}}\#([0,1]\times S^{3}), \overline{S}): (S^{3}, K_{-})\rightarrow (S^{3}, K_{+}).\label{cob1}\end{eqnarray}
 $\overline{S}$ is an embedded surface in $\overline{\mathbb{CP}^{2}}\#([0,1]\times S^{3})$ obtained by resolving a self-intersection of $S$ and it represents a homology class $$2e\in H^{2}(\overline{\mathbb{CP}^{2}};\mathbb{Z})\cong H^{2}(\overline{\mathbb{CP}^{2}}\#([0,1]\times S^{3});\mathbb{Z}),$$ where $e$ is an element represented by an exceptional curve.
 Similarly, we obtain a  cobordism of pairs
 \begin{equation}(\overline{\mathbb{CP}^{2}}\#([0, 1]\times S^{3}),\overline{S}'):(S^{3},K_{+})\rightarrow (S^{3}, K_{-}).\label{Cob2}
 \end{equation}
 Cobordisms of pairs $(W, \overline{S}):(S^{3}, K_{-})\rightarrow (S^{3}, K_{+})$ and $(W', \overline{S}'):(S^{3}, K_{+})\rightarrow (S^{3},K_{-})$  constructed as above are called {\it cobordism of positive /negative crossing change} respectively.
 \begin{prp}\label{neg def pair}
 Fix a holonomy parameter $\alpha\in (0, \frac{1}{2})\cap\mathbb{Q}$ with $\Delta_{K_{+}}(e^{4\pi i \alpha})\neq0$ and $\Delta_{K_{-}}(e^{4\pi i \alpha})\neq 0$.
 Let $\mathscr{S}$ be an integral domain over $\mathscr{R}_{\alpha}$.
 We assume that $\sigma_{\alpha}(K_{+})=\sigma_{\alpha}(K_{-})$.
 Then cobordisms of positive and negative crossing change are negative definite over $\mathscr{S}$.
 \end{prp}
 \begin{proof}
Firstly, we show that (\ref{cob1}) is a negative definite pair. Put $W=\overline{\mathbb{CP}^{2}}\#([0, 1]\times S^{3})$. 
Then it is clear that $W$ satisfies condition (1) in Definition \ref{negativedef} since $H^{1}(W;\mathbb{Z})=0$ and $H^{2}(W;\mathbb{Z})=\mathbb{Z}$.
Let $A_{m}$ be a $U(1)$-reducible instanton corresponding to an element $m\in \mathbb{Z}=H^{2}(W;\mathbb{Z})$. 
Then 
\begin{eqnarray*}
\bar{\kappa}(A_{m})&=&-(c_{1}(L_{m})+\alpha\overline{S})^{2}\\
&=&(m+2\alpha)^{2}
\end{eqnarray*}
where $L_{m}$ be a line bundle such that $c_{1}(L_{m})[e]=-m$. We also have 
\begin{equation*}
\nu(A_{m})=2c_{1}(L_{m})[\overline{S}]=-4m.
\end{equation*}
The index computation yields that
\begin{eqnarray*}
{\rm ind}({D}_{A_{m}})&=&8(m+2\alpha)^{2}+2(4\alpha-1)\cdot (-4m)-32\alpha^{2}+\sigma_{\alpha}(K_{+})-\sigma_{\alpha}(K_{-})-1\\
&=&8m(m+1)+\sigma_{\alpha}(K_{-})-\sigma_{\alpha}(K_{+})-1.
\end{eqnarray*}Thus $m=0, -1$ minimizes ${\rm ind}{D}_{A_{m}}$ , and this means that $A_{0}$ and $ A_{-1}$ are minimal reducibles.
Since $\sigma_{\alpha}(K_{+})=\sigma_{\alpha}(K_{-})$ by our assumption, the index for minimal instantons are $-1$ for the first case. Thus $(W, \overline{S})$ satisfies condition (2) in Definition \ref{negativedef}. Since $\bar{\kappa}(A_{m})=m^{2}+4\alpha m+4\alpha^{2}$,  we have 
\[
\eta^{\alpha}(W, \overline{S})=\begin{cases}
1-\lambda^{4\alpha-1}T^{4}&(\alpha\leq \frac{1}{4})\\
\lambda^{1-4\alpha}T^{-4}-1&(\alpha>\frac{1}{4}).
\end{cases}
\]
Since $1-\lambda^{4\alpha-1}T^{4} $ is invertible when $\alpha\neq \frac{1}{4}$ and non-zero when $\alpha=\frac{1}{4}$ by the assumption, $\eta^{\alpha}(W, \overline{S})$ is non-zero in $\mathscr{S}$. Hence $(W, \overline{S})$ satisfies condition (3) in Definition \ref{negativedef} and $(W, \overline{S})$ is a negative definite pair. 

It is also obvious that $(W, \overline{S}')$ satisfies condition (1) in Definition \ref{negativedef}. Since $\overline{S}'$ has a trivial homology class in $H_{2}(W;\mathbb{Z})$, minimal reducibles are only trivial one with index $-1$. Hence $\eta^{\alpha}(W, \overline{S}')=1\neq 0\in \mathscr{S}$. 
\end{proof}
Next, we discuss transversality of moduli spaces at reducibles.
Following \cite{Kr05, CDX}, we introduce the perturbation supported on the interior of cobordism.
Let $\mathcal{I}$ be an infinite countable set of indexes and consider the following data.
\begin{itemize}
\item{} A collection of  embedded 4-balls $\{B_{i}\}_{i\in \mathcal{I}}$ in $W^{+}\setminus S^{+}$.
\item{} A  collection of submersions $q_{i}:S^{1}\times B_{i}\rightarrow W^{+}\setminus S^{+}$ such that $q_{i}(1, \cdot)$ is the identity.
\item  For any $x\in W\setminus S$, the set $\{q_{i, x}\vert i\in \mathcal{I}, x\in B_{i}\} $ is a $C^{1}$-dense subset in the space of loops based at $x\in W\setminus S$.
\end{itemize}

For each $i\in \mathcal{I}$, consider a  self-dual 2-form $\omega_{i}$ on $B_{i}$ with ${\rm supp}(\omega_{i})\subset B_{i}$.
These self-dual 2-forms $\omega_{i}$ can be regarded as a self-dual 2-forms on $W^{+}\setminus S^{+}$.
We define $V_{\omega_{i}}:\mathcal{A}_{z}(W, S; \beta, \beta')\rightarrow \Omega^{+}(W^{+}\setminus S^{+}; \mathfrak{su}(2))$ as the following way,
\[V_{\omega_{i}}(A):=\pi(\omega_{i}\otimes {\rm Hol}_{q_{i}}(A)),\]
where $\pi: SU(2)\rightarrow \mathfrak{su}(2)$ is a map given by $g\mapsto g-\frac{1}{2}{\rm tr}(g)1$.
The argument similar to \cite{Kr05} shows that there are constants $K_{n, i}$ and differentials of $V_{\omega_{i}}$ are satisfies the following inequality:
\[\|D^{n}V_{\omega_{i}}|(a_{1},\cdots, a_{n})\|_{\check{L}^{2}_{m, A_{0}}}\leq K_{n, i}\|\omega\|_{C^{l}}\prod_{i=1}^{n} \|a_{i}\|_{\check{L}^{2}_{m, A_{0}}},\]
where $A_{0}$ is a singular connection which represents the homotopy class $z$ and $l\geq3$.
We choose a family of positive constants $\{C_{i}\}$ so that 
\[C_{i}\geq {\sup}\{K_{n, i}|0\leq n\leq i\}.\]
Consider a family of self-dual 2-forms $\{\omega_{i}\}$ such that $\sum_{i\in \mathcal{I}}C_{i}\|\omega_{i}\|_{C^{l}}$ converges.
For such a choice of $\{\omega_{i}\}$, $V_{\mathbf{\omega}}:=\sum_{i\in \mathcal{I}}V_{\omega_{i}}\omega_{i}$ defines a smooth map \[\mathcal{A}_{z}(W,S;\beta, \beta')\rightarrow \phi\check{L}^{2}_{m}(W^{+}\setminus S^{+}, \Lambda^{+}\otimes \mathfrak{su}(2))\]
between Banach manifolds.

We define $\mathcal{J}:=\{(i, j)\in \mathcal{I}\times \mathcal{I}|i\neq j, B_{i,j}:=B_{i}\cap B_{j}\neq \emptyset\}$ and $q_{i, j}:B_{i, j}\rightarrow W^{+}\setminus S^{+}$ by 
\[q_{i, j}|_{\{x\}\times S^{1}}:=q_{i, x}*q_{j, x}*q_{i, x}^{-1}*q_{j, x}^{-1}\]
for each $(i, j)\in\mathcal{J}$. 
We choose a family of constants $\{C_{i, j}\}_{(i, j)\in \mathcal{J}}$ as before.
Let $\omega_{i, j}$ be a self-dual2-form on $B_{i, j}$.
We introduce a Banach space $\mathcal{W}$ which consists of sequences of self-dual 2-forms $\omega=\{\omega_{i}\}_{i\in \mathcal{I}}\cup\{\omega_{i, j}\}_{(i, j)\in \mathcal{J}}$ with the following weighted $l^{1}$-norm:
\[\|\omega\|_{\mathcal{W}}:=\sum_{i\in \mathcal{I}}C_{i}\|\omega_{i}\|_{C^{l}}+\sum_{(i, j)\in \mathcal{J}}C_{i, j}\|\omega_{i, j}\|_{C^{l}}\]
For each $\omega\in \mathcal{W}$, we define a perturbation term 
\[V_{\omega}(A):=\sum_{i\in \mathcal{I}}V_{\omega_{i}}(A)\otimes\omega_{i}+\sum_{(i, j)\in\mathcal{J}}V_{\omega_{i, j}}(A)\otimes\omega_{i, j}\]
which defines a smooth map $V_{\omega}:\mathcal{A}_{z}(W,S;\beta, \beta')\rightarrow \check{L}^{2}_{m, \epsilon}(W^{+}\setminus S^{+}, \Lambda^{+}\otimes \mathfrak{su}(2)).$
We call 
\[F^{+}_{A}+U_{\pi_{W}}(A)+V_{\omega}(A)=0\]
the secondary perturbed ASD-equation over a cobordism of pairs $(W, S):(Y, K)\rightarrow (Y', K')$. 
$M^{\pi_{W}, \omega}(W, S;\beta, \beta')$ denotes the moduli space of solutions for the secondary perturbed ASD-equation.
\begin{prp}\label{perturbation red}
Let $(W, S)$ be a cobordism of pairs such that $b^{1}(W)=b^{+}(W)=0$. 
Assume that a perturbation $\pi_{W}$ is chosen so that the perturbed ASD-equation
\[F^{+}_{A}+U_{\pi_{W}}(A)=0\]
cuts out  the irreducible part of the moduli space transversely.  
Let $A^{\rm ad}=\mathbf{1}\oplus B$  be an adjoint connection of abelian reducible ASD connection $[A]\in M(W, S; \theta_{\alpha}, \theta'_{\alpha})_{2d+1}$ with ${\rm ind}(d^{*}_{B}\oplus d^{+}_{B})\geq 0$.
Then for a small generic perturbation $\omega \in \mathcal{W}$, the secondary  perturbed ASD-equation  cuts out the irreducible part of the moduli space transversely.   Moreover, $M^{\pi_{W, \omega}}(W, S; \theta_{\alpha}, \theta'_{\alpha})_{2d+1}$ is regular at $[A]$ and has a neighborhood of $[A]$ which is homeomorphic to a cone on $\pm\mathbb{CP}^{d}$.
\end{prp}
\begin{proof}
For each connected component $M_{z}^{\pi_{W}, \omega}(W,S; \beta, \beta')$ of moduli spaces, the argument of \cite[Section 7]{CDX} is adapted to our case and reducible points are regular for generic perturbations.
Taking countable intersections of these subsets of  regular perturbations in $W$, we can  find  generic perturbations $\omega\in \mathcal{W}$ such that  the statement holds.
The claim about local structures around reducibles can be refined using a standard argument. 
See  \cite[Proposition 4.3.20]{DK90}, for example.
\end{proof}
The essentially same argument is used in \cite{DS2}.
From now on, we assume that perturbations over a cobordism of pairs $(W, S)$ are chosen so that they satisfy the statement of  Proposition \ref{perturbation red}.

\subsection{Orientation}
We see the orientation of moduli spaces over a cylinder based on \cite{KM11,DS1}. 
Consider a reference connection $A_{0}$ on $(W^{+}, S^{+})$ as described in subsection \ref{cobordism} and the ASD-operator (\ref{ASD-ope on cob}).
 If the weight function $\phi$ has the form $e^{-\delta |t|}$ on one end, the functional space $\phi\check{L}^{2}_{m, A_{0}}$ consists of  exponential decaying functions on that end.
 On the other hand, if  the weight function $\phi$ has the form $e^{\delta |t|}$ on one end, the functional space $\phi\check{L}^{2}_{m, A_{0}}$ allows exponential growth functions.
 The index of the operator $D_{A}$ depends on these choices of weighted functions.
 To distinct these two situations,  $\theta_{\alpha, \pm}$ denote the reducible flat limit $\theta_{\alpha}$ with weighted functions $e^{\pm \delta |t|}$.
  Let $z$ be a path along $(W, S)$ between two critical limits $\beta, \beta'$ on $(Y, K)$, $(Y', K')$. 
  The family index of $D_{\{A\}}$ defines a trivial line bundle ${\rm det \ ind}(D_{A})$ on each $\mathcal{B}_{z}(W, S,\beta, \beta')$.
  Let $\mathscr{O}_{z}[W, S;\beta_{0}, \beta_{1}]$ be a two point set of the orientation of the determinant line bundle ${\rm det\ ind}(D_{\{A\}})$.
 $\mathscr{O}_{z}[W, S; \beta, \beta']$ is a set of orientation of the moduli space $M^{\alpha}_{z}(W, S;\beta, \beta')$. 
 There is a transitive and faithful $\mathbb{Z}_{2}$-action on $\mathscr{O}_{z}[W, S;\beta, \beta']$. 
 For a composition of cobordisms $(W_{2}, S_{2})\circ(W_{1}, S_{1})$, there is  a pairing \[\Phi:\mathscr{O}_{z_{1}}[W_{1},S_{1};\beta, \beta']\otimes_{\mathbb{Z}_{2}}\mathscr{O}_{z_{2}}[W_{2}, S_{2};\beta',\beta'']\rightarrow \mathscr{O}_{z_{2}\circ z_{1}}[W_{2}\circ W_{1}, S_{2}\circ S_{1};\beta,\beta'']\] which is induced from the gluing formula of index. 
If we consider a gluing operation along the reducible connection $\theta_{\alpha}$, we choose $\beta'=\theta_{\alpha, +}$ at the first component and $\theta_{\alpha, -}$ at the second component.
Since there is a natural isomorphism between $\mathscr{O}_{z}[W, S;\beta, \beta']$ and $\mathscr{O}_{z'}[W, S; \beta, \beta']$, we omit $z$ from the above notations. We call an element of $\mathscr{O}[W, S;\theta_{\alpha \ +}, \theta'_{\alpha\ -}]$ a homology orientation of $(W, S)$.
For a given knot in integral homology 3-sphere $(Y, K)$, we use the following notations:
\[\mathscr{O}[\beta]:=\mathscr{O}[Y\times I, K\times I;\beta, \theta_{\alpha -} ]\]
if $\beta$ is irreducible, and
\[\mathscr{O}[\theta_{\alpha}]:=\mathscr{O}[Y\times I, K\times I;\theta_{\alpha +}, \theta_{\alpha -}].\]
There is an isomorphism  \[\mathscr{O}[W, S;\theta_{\alpha, +},\theta_{\alpha, -}]\vert_{[A_{0}]}\cong \Lambda^{\rm top}(H^{1}(W)\oplus H^{+}(W))\]and elements $o_{W}\in\mathscr{O}[W, S;\theta_{\alpha, +},\theta_{\alpha, -}] $ is called homology orientation.

Now, we describe how the orientation of the moduli space $M(W, S;\beta, \beta')$ is defined. 
Let $o_{W}\in \mathscr{O}[W, S;\theta_{\alpha, +}, \theta_{\alpha, -}]$ be a given homology orientation for $(W, S)$.
We fix elements $o_{\beta}\in\mathscr{O}[\beta]$ and $o_{\beta'}\in \mathscr{O}[\beta']$. Then the orientation $o_{(W, S;\beta, \beta')}\in \mathscr{O}[W, S;\beta,\beta']$ is fixed so that 
\[\Phi(o_{\beta}\otimes o_{W})=\Phi(o_{(W, S;\beta, \beta')}\otimes o_{\beta'}).\]

The moduli space $\breve{M}_{z}(\beta_{0}, \beta_{1})$ is oriented as the following way.
First, we fix orientations ${o}_{\beta_{0}} \in \mathscr{O}[\beta_{0}], {o}_{\beta_{1}}\in \Lambda[\beta_{1}]$. Then the orientation of $M(\beta_{0}, \beta_{1})$ is determined as above way. Note that there is an $\mathbb{R}$-action on $M_{z}(\beta_{1}, \beta_{2})$. Let $\tau_{s}(t, y)=(t-s, y)$ be a transition on the cylinder $(Y, K)\times \mathbb{R}$. Then the $\mathbb{R}$-action on $M(\beta_{0}, \beta_{1})$ is given by the pull-back $[A]\mapsto [\tau^{*}A]$. Finally, we orient $\breve{M}(\beta_{1}, \beta_{2})$ so that $\mathbb{R}\times \breve{M}(\beta_{1}, \beta_{2})=M(\beta_{1}, \beta_{2})$  is orientation preserving.

The boundary of moduli spaces are oriented so that the outward normal vector sits at the first place in the tangent space.

\section{$\mathcal{S}$-complexs and Fr{\o}yshov type invariants\label{S-complex}}
In this section, we extend the construction of $\mathcal{S}$-complexes $\tilde{C}_{*}(Y, K)$ for $(Y, K)$ in \cite{DS1} to general holonomy parameters.
We also introduce $\mathbb{Z}\times \mathbb{R}$-bigrading of $\mathcal{S}$-complex with rational holonomy parameters for the specific choice of coefficient and its filtered subcomplex based on \cite{nozaki2019filtered}.
\subsection{ A Review on $\mathcal{S}$-complexs and Fr{\o}yshov invariants}
\label{4-1}Firstly, we review the $\mathcal{S}$-complex and the Fro{\o}yshov type invariant introduced by \cite{DS1} and \cite{DS2}. They are  defined by using purely algebraic objects. 

\begin{dfn}\label{Scpx}
Let $R$ be an integral domain, and $\tilde{C}_{*}$ be a finitely generated and graded free $R$-module. The triple $(\tilde{C}_{*}, \tilde{d},\chi)$ is called the $\mathcal{S}$-complex if 
\begin{enumerate}
\item $\tilde{d}:\tilde{C}_{*}\rightarrow \tilde{C}_{*}$ is a degree $-1$ homomorphism.
\item $\chi: \tilde{C}_{*}\rightarrow \tilde{C}_{*}$ is a degree $+1$ homomorphism.
\item $\tilde{d}$ and $\chi$ satisfy 
\begin{itemize}
\item $\tilde{d}^{2}=0$, $\chi^{2}=0$, and $\tilde{d}\chi+\chi\tilde{d}=0$.
\item ${\rm Ker}(\chi)/{\rm Im}(\chi)\cong R_{(0)}$, where $R_{(0)}$ is a copy of $R$ in $\tilde{C}_{0}$
\end{itemize}
\end{enumerate}
\end{dfn}
If $(C_{*}, d)$ is a given chain complex with coefficient ring $R$, we can form  $\mathcal{S}$-complex as follows
$$\tilde{C}_{*}=C_{*}\oplus C_{*-1} \oplus R,$$
\begin{equation}
\tilde{d}=\left[\begin{array}{ccc}
d& 0& 0\\
v&-d&\delta_{2}\\
\delta_{1}&0&0
\end{array}\right],  \ 
\chi=\left[\begin{array}{ccc}
0&0&0\\
1&0&0\\
0&0&0
\end{array}\right],\label{mtx1}
\end{equation}

where $\delta_{1}:C_{*}\rightarrow R$ , $\delta_{2}:R\rightarrow C_{*-1}$ and $v: C_{*}\rightarrow C_{*-2}$. Since there are conditions on $\tilde{d}$ and $ \chi$ in (\ref{Scpx}), the components in $\tilde{d}$ and  $\chi$ have to satisfy the following relations
\begin{eqnarray}
\delta_{1}d=0,\ d\delta_{2}=0,\ dv-vd-\delta_{2}\delta_{1}=0.\label{rel d chi}
\end{eqnarray}
Conversely, if the $\mathcal{S}$-complex $(\tilde{C}, \tilde{d}, \chi)$ is given then there is a decomposition of the form $\tilde{C}_{*}=C_{*}\oplus C_{*-1}\oplus R$. 
\\

There is also a notion of $\mathcal{S}$-morphism, which is a morphism of $\mathcal{S}$-complexes. 
\begin{dfn}
Let $(\tilde{C}_{*}, \tilde{d}, \chi)$ and $(\tilde{C}'_{*}, \tilde{d}', \chi')$ be $\mathcal{S}$-complexes. Fix decompositions  $\tilde{C}_{*}=C_{*}\oplus C_{*-1}\oplus R$ and $\tilde{C}'=C'_{*}\oplus C'_{*-1}\oplus R$.  A chain map $\tilde{m}:\tilde{C}_{*}\rightarrow \tilde{C}'_{*}$ is called $\mathcal{S}$-morphism if it has the form
\begin{equation}
\tilde{m}=\left[\begin{array}{ccc}
m&0&0\\
\mu&m&\Delta_{2}\\
\Delta_{1}&0&\eta
\end{array}
\right]\label{mtx2}
\end{equation}
where $\eta\neq 0\in R$.
\end{dfn}
The condition $\tilde{m}$ is a chain map is equivalent to the following relations.
\begin{eqnarray*}
m d-d m =0\\
\Delta_{1}d+\eta\delta_{1}-\delta'_{1}m=0\\
d'\Delta_{2}-\delta_{2}'\eta+m\delta_{2}=0\\
\mu d+m v-\Delta_{2}\delta_{1}-v'm+d'\mu-\delta_{2}'\Delta_{1}=0
\end{eqnarray*}
\begin{dfn}
Let $\tilde{m}, \tilde{m}': \tilde{C}_{*}\rightarrow \tilde{C}'_{*}$ be two $\mathcal{S}$-morphisms. An $\mathcal{S}$-chain homotopy of $\tilde{m}$ and $\tilde{m'}$ is a degree 1 map $\tilde{h}:\tilde{C}_{*}\rightarrow \tilde{C}'_{*}$ such that
$$\tilde{d}'\tilde{h}+\tilde{h}\tilde{d}=\tilde{m}-\tilde{m}', \ \chi' \tilde{h}+\tilde{h}\chi=0.$$
Two $\mathcal{S}$-complex $\tilde{C}_{*}$ and $\tilde{C}'_{*}$ are called $\mathcal{S}$-chain homotopy  equivariant if there are $\mathcal{S}$-morphisms $\tilde{m}:\tilde{C}_{*}\rightarrow \tilde{C}'_{*}$ and $\tilde{m}':\tilde{C}'_{*}\rightarrow \tilde{C}_{*}$ such that $\tilde{m}\tilde{m}'$
and $\tilde{m}'\tilde{m}$ are $\mathcal{S}$-chain homotopic to the identity. 
\end{dfn}
\begin{rmk}\label{S chain unit}
{\rm Consider  $\mathcal{S}$-morphisms $\tilde{m}:\tilde{C}_{*}\rightarrow \tilde{C}'_{*}$ and $\tilde{m}': \tilde{C}'_{*}\rightarrow \tilde{C}_{*}$. If there are unit elements $c$ and $c'$ in the coefficient ring $R$, and two $\mathcal{S}$-chain homotopies 
\[\tilde{m}'\tilde{m}\sim c{\rm id}_{\tilde{C}_{*}},\ \tilde{m}\tilde{m}'\sim c'{\rm id}_{\tilde{C}'_{*}}\]
then two $\mathcal{S}$-complexes $\tilde{C}_{*}$ and $\tilde{C}'_{*}$ are $\mathcal{S}$-chain homotopy equivalent since both $c^{-1}\tilde{m}$ and $c'^{-1}\tilde{m}$ are   $\mathcal{S}$-chain homotopic to the composition $c^{-1}c'^{-1}\tilde{m}\tilde{m}'\tilde{m}$.
}
\end{rmk}
The Fr{\o}yshov type invariant is defined from $\mathcal{S}$-complex. It assigns an integer $h(\tilde{C}_{*})$ to each $\mathcal{S}$-complex $\tilde{C}_{*}$.
\begin{dfn}{\rm ( \cite[Proposition 4.15]{DS1})}
\begin{itemize}
\item $h(\tilde{C}_{*})>0 $ if only of there is an element $\beta \in C_{*}$ such that $d\beta=0$ and $\delta_{1}\beta\neq 0$.
\item If $h(\tilde{C}_{*})=k>0$ then $k$ is the largest integer such that there exists $\beta \in C_{*}$ satisfying the following properties; $$d\beta=0,\ \delta_{1}v^{k-1}(\beta)\neq 0, \delta_{1}v^{i}\beta=0\ {\rm for }\ i\leq k-2.$$
\item If $h(\tilde{C}_{*})=k\leq0$  then there are elements $a_{0}, \cdots, a_{-k}\in R$ and $\beta \in C_{*}$ such that $$d\beta=\sum_{i=0}^{-k}v^{i}\delta_{2}(a_{i}).$$
\end{itemize}

\end{dfn}
The followings are basic properties of the Fr{\o}yshov type invariant.
\begin{prp}{\rm (\cite[Corollary 4.14 ]{DS1})}\label{Froyshov ineq}

If there is an $\mathcal{S}$-morphism $\tilde{m}:\tilde{C}_{*}\rightarrow \tilde{C}'_{*}$ then $h(\tilde{C}_{*})\leq h(\tilde{C}'_{*})$.
\end{prp}
For given two $\mathcal{S}$-complexes $(\tilde{C}_{*}, \tilde{d}, \chi)$ and $(\tilde{C}'_{*}, \tilde{d}', \chi')$, the product of $\mathcal{S}$-complex $(\tilde{C}^{\otimes}_{*}, \tilde{d}^{\otimes}, \chi^{\otimes})$ is defined as follows,
\begin{eqnarray*}
\tilde{C}^{\otimes}_{*}&=&\tilde{C}_{*}\otimes \tilde{C}'_{*}\\
\tilde{d}^{\otimes}&=&\tilde{d}\otimes 1+\epsilon \otimes \tilde{d}'\\
{\chi}^{\otimes}&=&\chi\otimes 1+\epsilon\otimes \chi'
\end{eqnarray*}
where $\epsilon:\tilde{C}'_{*}\rightarrow \tilde{C}'_{*}$ is given by $\epsilon(\beta')=(-1)^{{\rm deg}(\beta')}\beta'$ on elements of  homogeneous degree. 
Let $d^{\otimes}, v^{\otimes}, \delta_{1}^{\otimes }$ and $\delta^{\otimes}_{2}$ be components of $\tilde{d}^{\otimes}$ with respect to the splitting $\tilde{C}^{\otimes}=C^{\otimes}_{* }\oplus C^{\otimes}_{*-1}\oplus R$. Using the decomposition ${C}^{\otimes}_{*}=(C\otimes C')_{*}\oplus (C\otimes C')_{*-1}\oplus C_{*}\oplus C'$, these maps are represented by as follows.
\[
d^{\otimes}=\left[\begin{array}{cccc}
d\otimes 1+\epsilon \otimes d'&0&0&0\\
-\epsilon v\otimes 1+\epsilon \otimes v'&d\otimes 1-\epsilon \otimes d'&\epsilon \otimes \delta'_{2}&-\delta'_{2}\otimes 1\\
\epsilon\otimes \delta'_{1}&0&d&0\\
\delta'_{1}\otimes 1&0&0&d'
\end{array}\right],
\]
\[v^{\otimes}=\left[\begin{array}{cccc}
v\otimes 1&0&0&\delta_{2}\otimes 1\\
0&v\otimes 1&0&0\\
0&0&v&0\\
0&\delta_{1}\otimes 1&0&v'
\end{array}\right],\]
\[\delta^{\otimes}_{1}=[0,0,\delta_{1}, \delta'_{1}], \ \delta^{\otimes}_{2}=[0,0, \delta_{2}, \delta'_{2}]^{\rm T}.
\]
The  Fr{\o}yshov type invariant behaves additively for the product of $\mathcal{S}$-complexes as follows.
\begin{prp}\label{h sum}{\rm ( \cite[Corollary 4.28 ]{DS1})}
\[h(\tilde{C}^{\otimes}_{*})=h(\tilde{C}_{*})+h(\tilde{C}'_{*}).\]
\end{prp}
\subsection{Floer homology groups with local coefficients}\label{Floer homology}
In this subsection, we construct summands $C_{*}$ in $\mathcal{S}$-complex as a Floer chain group with local coefficient. 
Let $(Y, K)$ be a oriented knot in an integral homology 3-sphere.
We fix  a holonomy parameter $\alpha$ so that $\Delta_{(Y, K)}(e^{4\pi i \alpha})\neq 0$ to isolate unique flat reducible connection $\theta_{\alpha}$.
We assign an abelian group $\Delta_{[B]}$, for each elements $[B]$ in the configuration space $\mathcal{B}(Y, K, \alpha)$ and an isomorphism $\Delta_{z}:\Delta_{[B_{0}]}\rightarrow \Delta_{[B_{1}]}$ for each homotopy class $z\in \pi_{1}(\mathcal{B}(Y, K,\alpha),[B_{0}], [B_{1}])$.
If such assignment is functorial, a Floer chain complex with local coefficient $\Delta$ is defined as follows:
\[C^{\alpha}_{*}(Y, K,\Delta)=\bigoplus_{\beta\in \mathfrak{C}_{\pi}^{*}(Y, K,{\alpha})}\Delta_{\beta}\mathscr{O}[\beta],\]
\[\langle d(\beta_{0}),\beta_{1}\rangle=\sum_{z:\beta_{0}\rightarrow \beta_{1}}\sum_{[\breve{A}]\in \breve{M}_{z}(\beta_{0}, \beta_{1})}\epsilon([\breve{A}])\otimes \Delta_{z}.\]
The $\mathbb{Z}/4$-grading of $C^{\alpha}_{*}(Y, K,\Delta)$ is defined by mod $4$ grading for critical points.
Consider a subring $\mathscr{R}_{\alpha}$ in the Novikov ring $\Lambda^{\mathbb{Z}[T^{-1}, T]]}$ which is introduced in Subsection \ref{summary} .

Lemma \ref{kappa and nu} enables us to define a local coefficient system $\Delta=\Delta_{\mathscr{R}_{\alpha}}$ as follows:

\begin{eqnarray*}
\Delta_{\mathscr{R}_{\alpha},[B]}&:=&\mathscr{R}_{\alpha}\cdot\lambda^{{CS}(B)} T^{{\rm hol}_{K}(B)},\\
\Delta_{\mathscr{R}_{\alpha}, z}&:=&\#\breve{M}_{z}(\beta, \beta_{1})_{0}\lambda^{-{\kappa}(z)}T^{\nu(z)}.
\end{eqnarray*}

Note that this definition is independent of  choices of representatives of $[B]$ and $\theta_{\alpha}$. 
 We denote $C^{\alpha}_{*}(Y, K;\Delta_{\mathscr{R}_{\alpha}})$ for a chain complex with the coefficient system over $\mathscr{R}_{\alpha}$. 
 For any algebra $\mathscr{S}$ over $\mathscr{R}_{\alpha}$, we can extend the above construction to the coefficient $\mathscr{S}$.
 \begin{dfn}
Let $(Y, K)$ be a oriented knot in a integral homology 3-sphere and $\mathscr{S}$ be an algebra over $\mathscr{R}_{\alpha}$. 
Fix $\alpha\in (0, \frac{1}{2})\cap\mathbb{Q}$ so that $\Delta_{(Y,K)}(e^{4\pi i \alpha})\neq 0$.
The homology group of the  $\mathbb{Z}/4$-graded chain complex $(C^{\alpha}_{*}(Y,K;\Delta_{\mathscr{S}}),d)$ is denoted by $I^{\alpha}_{*}(Y, K;\Delta_{\mathscr{S}})$.
We call $I^{\alpha}_{*}(Y, K;\Delta_{\mathscr{S}})$ the irreducible instanton knot homology for a holonomy parameter $\alpha$ with a local coefficient system $\Delta_{\mathscr{S}}$. 
\end{dfn}

Let $(W,S): (Y, K)\rightarrow (Y', K')$ be a negative definite cobordism over $\mathscr{S}$.
We define an induced morphism  $m={m}_{(W, S)}; C^{\alpha}_{*}(Y, K;\Delta_{\mathscr{S}})\rightarrow C^{\alpha}_{*}(Y', K';\Delta_{\mathscr{S}})$ by
\[m(\beta)=\sum_{\beta'\in \mathfrak{C}^{*}(Y, K,\alpha)}\sum_{z:\beta\rightarrow \beta'}\#M_{z}(W, S;\beta, \beta')_{0}\lambda^{\kappa_{0}-\kappa(z)}T^{\nu(z)-\nu_{0}} \beta'.\]
Counting boundary of one-dimensional moduli space $M^{+}_{z}(W, S; \beta, \beta')_{1}$ for each homotopy class $z$, we obtain a relation
\[dm-md'=0.\]

We remark that 
\begin{itemize}
\item For a composition of negative definite cobordism $(W, S):=(W_{0}, S_{0})\circ (W_{1}, S_{1})$, there is a map $\phi$ such that\[d\phi-\phi d=m_{(W_{1}, S_{1})}\circ m_{(W_{0}, S_{0})}-m_{(W, S)}.\] where metrics and perturbation data on $(W, S)$ is given by the composition of that of $(W_{0}, S_{0})$ and $(W_{1}, S_{1})$.
\item If $m_{(W, S)}$ and $m'_{(W, S)}$ are defined by different perturbation and metric date on the interior domain of $(W, S)$, they are chain homotopic.
\item If $(W, S)=[0, 1]\times (Y, K)$ then $m_{(W, S)}$ is chain homotopic to the identity map.  
\end{itemize}
Thus $C^{\alpha}_{*}(Y, K;\Delta_{\mathscr{S}})$ is an invariant of $(Y, K)$ up to chain homotopy.
We denote $I^{\alpha}_{*}(Y, K;\Delta_{\mathscr{S}})$ for its homology group and this is an invariant for $(Y, K)$.

Next, we introduce the filtered construction for the Floer chain complex based on Nozaki-Sato-Taniguchi \cite{nozaki2019filtered}.
For the filtered construction, we have to introduce the lift of critical points.
Let $\mathcal{G}_{(0, 0)}$ be a normal subgroup of the gauge group $\mathcal{G}(Y, K)$ which is given by
\[\mathcal{G}_{(0, 0)}:=\{g |k(g)=l(g)=0\}.\]Consider the quotient of the space of singular connections
\[\widetilde{\mathcal{B}}(Y, K,\alpha):=\mathcal{A}(Y, K,\alpha)/\mathcal{G}_{(0, 0)}.\]The Chern-Simons functional descends on $\tilde{\mathcal{B}}(Y, K,\alpha)$ as an $\mathbb{R}$-valued function, and we still use the same notation.
The normal subgroup $\mathcal{G}_{(0, 0)}$ is a connected component of the full gauge group $\mathcal{G}(Y, K)$ which corresponds to $(0, 0)\in \mathbb{Z}\oplus \mathbb{Z}\cong \pi_{0}(\mathcal{G}(Y, K))$.
Thus there is an action of $\pi_{0}(\mathcal{G}(Y, K))\cong \mathbb{Z}\oplus \mathbb{Z}$ on $\tilde{\mathcal{B}}(Y, K,\alpha)$ as a covering  transformation, and hence $\widetilde{\mathcal{B}}(Y,K,\alpha)$ is a covering space over $\mathcal{B}(Y, K, \alpha)$ with a fiber $\mathbb{Z}\oplus \mathbb{Z}$.
\begin{dfn}
A lift of  $[B]\in \mathcal{B}(Y,K,\alpha)$ to the covering space $\widetilde{\mathcal{B}}(Y, K,\alpha)$ is called a lift of $[B]$, and denoted by $\widetilde{[B]}$. 
\end{dfn}
For a fixed lift $\widetilde{[B]}$ of $[B]\in \mathcal{B}(Y, K, \alpha)$, the fiber of the projection $\widetilde{\mathcal{B}}(Y,  K, \alpha)\rightarrow \mathcal{B}(Y, K,\alpha)$ over a point $[B]$ can be described as \[\mathscr{L}_{[B]}:=\{g^{*}(\widetilde{[B]})\in \widetilde{\mathcal{B}}(Y,  K, \alpha)|g\in\pi_{0}(\mathcal{G}(Y, K))\}.\]
The fiber $\mathscr{L}_{[B]}$ can be seen a set of lifts of an element  $[B]$. 
There is another description of lifts.
Let $\tilde{\theta}_{\alpha}$ be a lift of reducible flat connections $\theta_{\alpha}$.
Then  a lift $\tilde{\beta}$ of $\beta\in \mathfrak{C}^{*}_{\pi}$ is fixed by choosing a path $z:\beta\rightarrow \theta_{\alpha}$ of connections over a cylinder whose endpoint is $\tilde{\theta}_{\alpha}$.

We again choose the coefficient ring $\mathscr{R}_{\alpha}$ and fix a lift $\tilde{\beta}$ for each critical points $\beta\in \mathfrak{C}_{\pi}(Y, K,\alpha)$.

Then we modify the   local coefficient system $\Delta_{\mathscr{R}_{\alpha}}$ so that

\[\Delta_{\mathscr{R}_{\alpha},{\beta}}=\mathscr{R}_{\alpha}\cdot\lambda^{CS(\tilde{\beta})}T^{{\rm hol}_{K}(\tilde{\beta})}.\]with the same map $\Delta_{\mathscr{R}_{\alpha}, z}$.
Once we fix an orientation of $\beta$,  each summand  $\Delta_{{\beta}}\mathscr{O}[{\beta}]$ in the chain complex are generated (over $\mathbb{Z}$) by the elements of the form $\lambda^{k}\xi_{\alpha}^{l}\tilde{\beta}=\lambda^{k+2\alpha l}T^{2l}\tilde{\beta}$ where $(k, l)\in \mathbb{Z}\oplus \mathbb{Z}$.   
The action of $\lambda^{k}\xi^{l}_{\alpha}$ corresponds to the action of the gauge transformation with $d(g)=(k, l)$.
Such elements can be identified with the set of lifts $\mathscr{L}_{\beta}$ of a critical point $\beta$.
Hence,  the chain complex $C^{\alpha}_{*}(Y, K;\Delta_{\mathscr{R}_{\alpha}})$ can be seen a $\mathbb{Z}$-module generated by the all lift of $\mathfrak{C}_{\pi}(Y, K, \alpha)$ under the modification as above. 

Once we fix lifts of generators, the chain  complex $C^{\alpha}_{*}(Y, K;\Delta_{\mathscr{R}_{\alpha}})$ admits $\mathbb{Z}\times \mathbb{R}$-bigrading as in \cite{DS2}, that is we can associate a pair of values which is defined as follows:

For a lift $\tilde{\beta}$ of a critical point $\beta\in \mathfrak{C}_{\pi}$, we define ${\rm deg}_{\mathbb{Z}}(\tilde{\beta}):={\rm gr}_{z}(\beta)$ where $z$ is a path corresponding to the lift $\tilde{\beta}$. Then we extend ${\rm deg}_{\mathbb{Z}}$ as 
\[{\rm deg}_{\mathbb{Z}}(\lambda^{i}\xi^{j}_{\alpha}\tilde{\beta})=8i+4j+{\rm deg}_{\mathbb{Z}}(\tilde{\beta}).\]
Next we define ${\rm deg}_{\mathbb{R}}$.
For a lift $\tilde{\beta}$ of a critical point $\beta\in \mathfrak{C}_{\pi}$, we define ${\rm deg}_{\mathbb{R}}(\tilde{\beta}):={CS}(\tilde{\beta})$.
This extends to elements of the form $\lambda^{i}\xi^{j}_{\alpha}\tilde{\beta}$ as 
\begin{equation}{\rm deg}_{\mathbb{R}}(\lambda^{i}\xi^{j}_{\alpha}\tilde{\beta})= i+2\alpha j+{\rm deg}_{\mathbb{R}}(\tilde{\beta}).\label{R-grading}
\end{equation}
In general,  an element $\gamma\in C^{\alpha}_{*}(Y, K, \alpha)$ has the form $\gamma=\sum_{i}a_{i}\gamma_{i}$ where $\gamma_{i}\in \bigcup_{\beta\in \mathfrak{C}_{\pi}}\mathscr{L}_{\beta}$. 
This is possibly an infinite sum. 
We define
\[{\rm deg}_{\mathbb{R}}(\gamma)={\rm max}\{{\rm deg}_{\mathbb{R}}(\gamma_{i})|a_{i}\neq 0\}\]
for $\gamma\neq 0$ and ${\rm deg}_{\mathbb{R}}(0)=-\infty$.

In summary, we have the following proposition.
\begin{prp}
Once we fix a lifts of critical points of the Chern-Simons functional, a chain complex $(C^{\alpha}_{*}(Y, K, \Delta_{\mathscr{R}_{\alpha}}), d)$ admits the $\mathbb{Z}\times {\mathbb{R}}$-bigrading.
\end{prp}
We write $C^{\alpha}_{*}(Y, K;\Delta_{\mathscr{R}_{\alpha}})^{[-\infty,\infty]}$ for the chain complex $C^{\alpha}_{*}(Y, K;\Delta_{\mathscr{R}_{\alpha}})$ with $\mathbb{Z}\times \mathbb{R}$-bigrading.

Let $\mathscr{C}^{*}\subset\mathbb{R}$ be a subset defined by $\mathscr{C}^{*}:=CS({\rm Crit}^{*})$.
For $R\in \mathbb{R}\setminus \mathscr{C}^{*}$, we define a subset
\[C^{\alpha}_{*}(Y, K;\Delta_{\mathscr{R}_{\alpha}})^{[-\infty,R]}:=\{\gamma\in C^{\alpha}_{*}(Y,K;\Delta_{\mathscr{R}_{\alpha}})^{[-\infty, \infty]}|{\rm deg}_{\mathbb{R}}(\gamma)<R\}.\]
This defines a subcomplex of $C^{\alpha}_{*}(Y, K;\Delta_{\mathscr{R}_{\alpha}})^{[-\infty, \infty]}$.
For two numbers $R_{0}, R_{1}\in (\mathbb{R}\setminus\mathscr{C}^{*})\cup\{\pm \infty\} $ such that $R_{0}\leq R_{1}$, we define a quotient complex as follows,
\begin{eqnarray*}C^{\alpha}_{*}(Y, K;\Delta_{\mathscr{R}_{\alpha}})^{[R_{0}, R_{1}]}&:=&C^{\alpha}_{*}(Y, K;\Delta_{\mathscr{R}_{\alpha}})^{[-\infty,R_{1}]}/C^{\alpha}_{*}(Y, K;\Delta_{\mathscr{R}_{\alpha}})^{[-\infty,R_{0}]}
\end{eqnarray*}
\begin{dfn}
For $R_{0}, R_{1}\in\mathbb{R}\cup\{\pm \infty\}$ such that $R_{0}\leq R_{1}$ and $R_{0}, R_{1}\notin \mathscr{C}^{*}\cup\mathscr{C'}^{*}$,
we call $C^{\alpha}_{*}(Y, K;\Delta_{\mathscr{R}_{\alpha}})$ a $[R_{0}, R_{1}] $-filtered chain complex. 
\end{dfn} 
Consider a negative definite cobordism $(W, S):(Y, K)\rightarrow (Y', K')$ with $\kappa_{0}=0$. 
A cobordism map $m_{(W, S)}$ on $C^{\alpha}_{*}(Y, K;\Delta_{\mathscr{R}_{\alpha}})$ induces a map  \[C^{\alpha}_{*}(Y, K;\Delta_{\mathscr{R}_{\alpha}})^{[-\infty,R]}\rightarrow C^{\alpha}_{*}(Y', K';\Delta_{\mathscr{R}_{\alpha}} )^{[-\infty,R]}\] by the restriction, and hence this induces a map \[m_{(W, S)}^{[R_{0}, R_{1}]}:C^{\alpha}_{*}(Y, K;\Delta_{\mathscr{R}_{\alpha}})^{[R_{0},R_{1}]}\rightarrow C^{\alpha}_{*}(Y', K'; \Delta_{\mathscr{R}_{\alpha}})^{[R_{0}, R_{1}]}.\]

As we described before,  the covering transformation on $\tilde{\mathcal{B}}(Y, K,\alpha)$ is generated by multiplications of elements $\lambda^{\pm 1}$and $\xi_{\alpha}^{\pm 1}$.
We also introduce other generators which fit to $\mathbb{Z}\times {\mathbb{R}}$-bigrading on $C^{\alpha}_{*}(Y, K;\Delta_{\mathscr{R}_{\alpha}})$.
\begin{figure}
    \centering
    \includegraphics[scale=0.75]{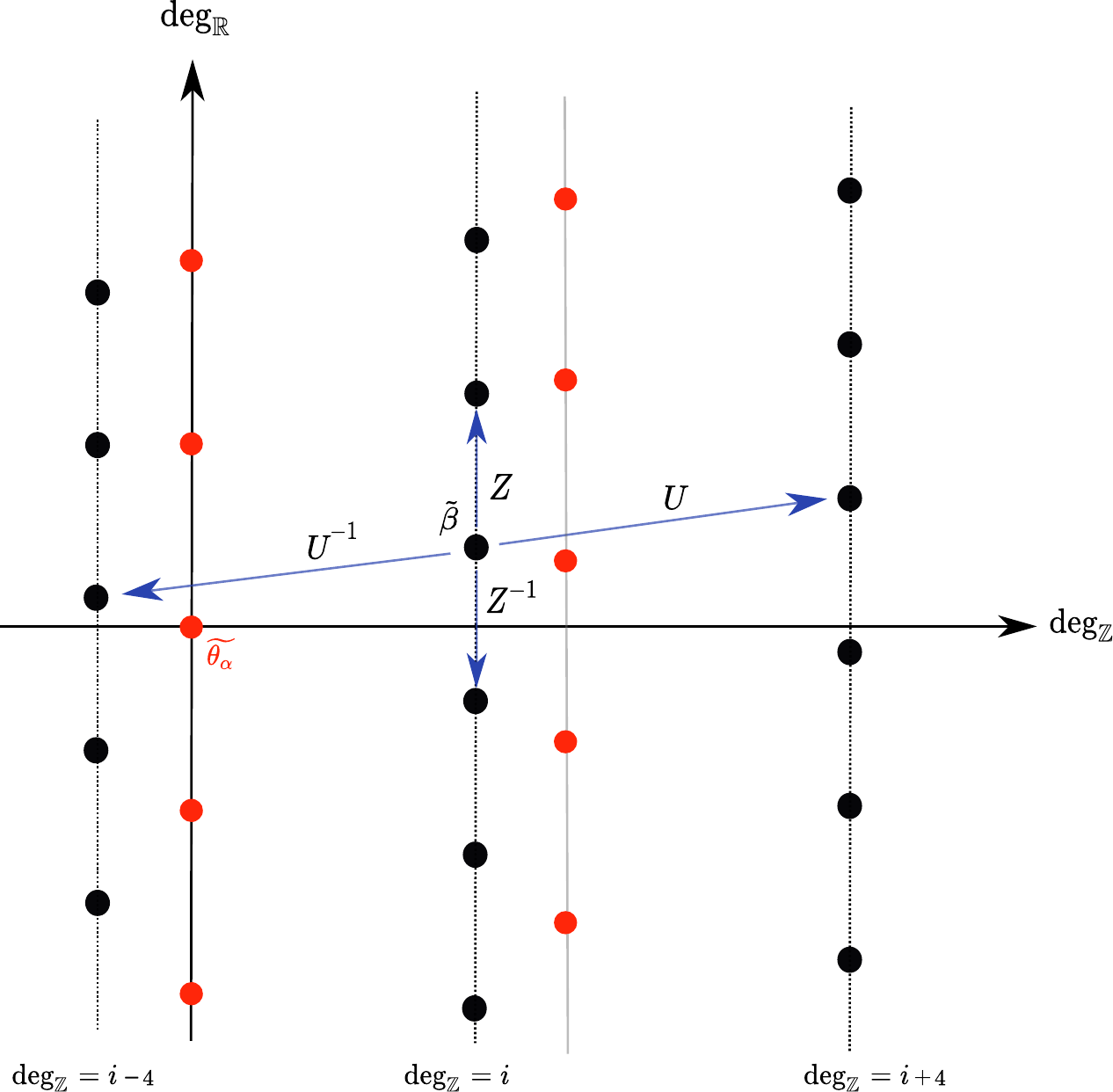}
    \caption{Black dots represent lifts of the irreducible flat connection $\beta$ and red dots represent lift of the reducible flat connection $\theta_{\alpha}$. }
    \label{lifts}
\end{figure}
Let us introduce two operators on $C^{\alpha}_{*}(Y, K;\Delta_{\mathscr{R}_{\alpha}})^{[-\infty, \infty]}$,
\[Z^{\pm 1}:=(\lambda^{1-4\alpha}T^{-4})^{\pm 1}, U^{\pm 1}:=(\lambda^{2\alpha}T^{2})^{\pm 1}.\]
These operators changes $\mathbb{Z}\times {\mathbb{R}}$-bigrading as follows,
\begin{itemize}
    \item{Operator $Z$:}
\begin{eqnarray}{\rm deg}_{\mathbb{Z}}(Z^{i}\tilde{\beta})&=&{\rm deg}_{\mathbb{Z}}(\tilde{\beta})\\{\rm deg}_{\mathbb{R}}(Z^{i}\tilde{\beta})&=&{\rm deg}_{\mathbb{R}}(\tilde{\beta})+(1-4\alpha)i.\end{eqnarray}

\item{Operator $U$:}
\begin{eqnarray}
{\rm deg}_{\mathbb{Z}}(U^{i}\tilde{\beta})&=&{\rm deg}_{\mathbb{Z}}(\tilde{\beta})+ 4i,\\
{\rm deg}_{\mathbb{R}}(U^{i}\tilde{\beta})&=&{\rm deg}_{\mathbb{R}}(\tilde{\beta})+ 2\alpha i.
\end{eqnarray}
\end{itemize}
See Figure \ref{lifts} for the case $\alpha<\frac{1}{4}$. 
Since $\lambda=ZU^{2}$, actions of two operations $Z$ and $U$ (and their inverses) on lifted critical points generate $C^{\alpha}_{*}(Y, K;\Delta_{\mathscr{R}_{\alpha}})^{[-\infty, \infty]}$.

\subsection{Maps $\delta_{1}$, $\delta_{2}$, $\Delta_{1}$, and $\Delta_{2}$}\label{delta}

We introduce operators which are defined by counting instantons on cylinder or  cobordism with a reducible limit. 
We remark that the sign convention of counting moduli spaces in this subsection is the same as that of \cite{DS1}.
Let $(Y, K)$ and $(Y', K')$ be two knots in integral homology $3$-spheres.
Let $\mathscr{S}$ be an integral domain over $\mathscr{R}_{\alpha}$.
In this subsection, we assume that the holonomy parameter $\alpha$ is chosen so that it $\Delta_{(Y, K)}(e^{4\pi i \alpha})\neq 0$ and $\Delta_{(Y', K')}(e^{4\pi i \alpha})\neq 0$.
\begin{dfn}
We define chain maps $\delta_{1}:C^{\alpha}_{*}(Y, K;\Delta_{\mathscr{S}})\rightarrow \mathscr{S}$ and $\delta_{2}:\mathscr{S}\rightarrow C^{\alpha}_{-2}(Y, K;\Delta_{\mathscr{S}})$ as follows.
\[\delta_{1}(\beta):=\sum_{z:\beta\rightarrow \theta_{\alpha}}\#\breve{M}_{z}(\beta, \theta_{\alpha})_{0}\lambda^{-\kappa(z)}T^{\nu(z)},
\]\[\delta_{2}(1):=\sum_{\substack{\beta\in \mathfrak{C}_{\pi}^{*}(Y, K,\alpha)\\ {\rm gr}(\beta)\equiv 2}}\sum_{z:\theta_{\alpha}\rightarrow \beta}\#\breve{M}_{z}(\theta_{\alpha},\beta)_{0}\lambda^{-\kappa(z)}T^{\nu(z)}\beta.\]

\end{dfn}
Since the compactified one dimensional moduli space $\breve{M}_{z}^{+}(\beta, \theta_{\alpha})_{1}$ has oriented boundaries
\[\bigcup_{\substack{\gamma\in \mathfrak{C}_{\pi}^{*}\\ {\rm gr}(\gamma)\equiv 1}}\bigcup_{\substack{z_{1}, z_{2}\\
z_{1}\circ z_{2}=z}}\breve{M}_{z_{1}}(\beta, \gamma)_{0}\times \breve{M}_{z_{2}}(\gamma, \theta_{\alpha})_{0},\]it is straightforward to check that $d\circ \delta_{1}=0$. Similarly, $\delta_{2}\circ d=0$ holds. 

Next, we define $\Delta_{1}:C^{\alpha}_{*}(Y, K;\Delta_{\mathscr{S}})\rightarrow \mathscr{S}$ and $\Delta_{2}:\mathscr{S}\rightarrow C^{\alpha}_{*}(Y', K';\Delta_{\mathscr{S}})$ for a cobordism of pairs $(W, S):(Y, K)\rightarrow (Y', K')$.
\begin{dfn}
\[\Delta_{1}(\beta):=\sum_{z}\#M_{z}(W,S;\beta, \theta'_{\alpha})_{0}\lambda^{\kappa_{0}-\kappa(z)}T^{\nu(z)-\nu_{0}},\]
\[\Delta_{2}(1):=\sum_{\beta'\in\mathfrak{C}_{\pi}^{*}(Y', K',\alpha)}\sum_{z}\#M_{z}(W,S;\theta_{\alpha}, \beta')_{0}\lambda^{\kappa_{0}-\kappa(z)}T^{\nu(z)-\nu_{0}}\beta'.\]
\end{dfn}

\begin{prp}\label{delt d}
Let $m=m_{(W, S)}:C_{*}^{\alpha}(Y, K;\Delta_{\mathscr{S}})\rightarrow C_{*}^{\alpha}(Y', K';\Delta_{\mathscr{S}})$ be a  cobordism map induced from a negative definite pair $(W,S): (Y, K)\rightarrow (Y', K')$.
Then the following relations:
\begin{enumerate}
\item[\rm (i)] $\Delta_{1}\circ d+\eta\delta_{1}-\delta_{1}'\circ m=0$,
\item[\rm (ii)]$d'\circ \Delta_{2}-\delta'_{2}\eta+m\circ \delta_{2}=0$.
\end{enumerate}
hold, where $\eta=\eta^{\alpha}(W, S)$ and $'$ denotes corresponding maps for the pair $(Y', K')$.
\end{prp}
\begin{proof}
The relation $\rm (i)$ is given by counting ends of each component of the one-dimensional moduli space $M_{z}(W, S; \beta, \theta'_{\alpha})_{1}$ as in \cite[Proposition 3.10 ]{DS1}.  Boundary components of $M_{z}(W, S; \beta, \theta'_{\alpha})_{1}$ with the induced orientation are given by 
\begin{itemize}
\item[(a)]\[-\bigcup_{\beta_{1}\in \mathfrak{C}^{*}_{\pi}}\bigcup_{\substack{z_{1}, z_{2}\\z_{1}\circ z_{2}=z}}\breve{M}_{z_{1}}(\beta, \beta_{1})_{0}\times M_{z_{2}}(W, S; \beta_{2}, \theta'_{\alpha})_{0}, 
\]
\item[(b)] \[\bigcup_{\gamma'\in \mathfrak{C'}^{*}_{\pi}}\bigcup_{\substack{ z_{1}, z_{2}\\ z_{1}\circ z_{2}=z}} M_{z_{1}}(W, S;\beta, \beta')_{0}\times \breve{M}_{z_{2}}(\beta', \theta_{\alpha})_{0},\]
\item[(c)]\[-\bigcup_{\substack{ z_{1}, z_{2}\\z_{1}\circ z_{2}=z}}\breve{M}_{z_{1}}(\beta, \theta_{\alpha})_{0}\times  M_{z_{2}}(W, S; \theta_{\alpha}, \theta'_{\alpha})_{0}\]
\end{itemize}
Note that product orientations of $\breve{M}_{z_{1}}(\beta, \gamma)_{0}\times M_{z_{2}}(W, S; \gamma, \theta'_{\alpha})_{0}$ and $\breve{M}_{z'}(\beta, \theta_{\alpha})_{0}\times  M(W, S; \theta_{\alpha}, \theta'_{\alpha})_{0}$ are opposite to orientations induced as boundaries of $M_{z}(W, S;\beta, \theta_{\alpha})_{1}$.
The signed counting of boundary components of type (a) and type (b) contribute to $-\Delta_{1}\circ  d(\beta)$  and $\delta'_{1}\circ m(\beta)$ respectively. 
Since $M(W, S; \theta_{\alpha}, \theta'_{\alpha})_{0}$ consists of minimal reducible elements, the counting of (c) gives $-\eta\delta_{1}(\beta)$.  This proves (i).
The relation (ii) can be similarly proved considering end of the one-dimensional moduli space $M_{z}(W,S;\theta_{\alpha}, \beta')_{1}$.
 \end{proof}
 
 \subsection{Maps $v$ and $\mu$}\label{vmu}
In this subsection, we introduce maps induced from cobordism of pairs $(W,S)$ with an embedded curve $\gamma\subset S$. 
Our assumptions for the choice of holonomy  parameter $\alpha$ and the coefficient $\mathscr{S}$ are same as the previous subsection.
We remark that the sign convention of moduli spaces in this subsection is also same as that of \cite{DS1}. 
In particular, if $f:M\rightarrow N$ be a smooth map between oriented manifolds then $f^{-1}(y)$ for a regular value $y\in N$ is oriented so that 
\[T_{x}M=N_{x}f^{-1}(y)\oplus T_{x}f^{-1}(y)\]
is orientation preserving, where $N_{x}f^{-1}(y)$ is a fiber of normal bundle for $f^{-1}(y)$ and its orientation is induced from that of $N$.
The mapping degree ${\rm deg}(f)$ is defined by using this orientation.

Assume that $\gamma:[0, 1]\rightarrow S$ is a smoothly embedded loop.
Fix a regular neighborhood $N_{\gamma}(\epsilon)$ of $\gamma$ in $W$ with radius $\epsilon>0$ and fix a base point $x_{0}\in\partial N_{\gamma}(\epsilon)$. We take a Seifert framing $\tilde{\gamma}_{\epsilon}\subset \partial N_{\gamma}(\epsilon) $ of $\gamma$ so that it pass through a base point $x_{0}$. The bundle decomposition $E=L\oplus L^{*}$ over $S\subset W$ extends to $N_{\gamma}(\epsilon)$, and the holonomy of adjoint connection of $[A]\in \mathcal{B}(W, S;\beta, \beta')$ yields ${\rm Hol}_{\tilde{\gamma}_{\epsilon}}(A^{\rm ad})\in S^{1}$.
Put
\[h^{\gamma}_{\beta \beta'}(A):=\lim_{\epsilon\rightarrow 0}{\rm Hol}_{\tilde{\gamma}_{\epsilon}}(A^{\rm ad}).\]
 The construction above gives a map
\begin{equation}
h^{\gamma}_{\beta \beta'}:\mathcal{B}(W, S, \alpha;\beta,\beta')\rightarrow S^{1}.\label{holmap}
\end{equation}
Note that this map itself depends on the choice of the Seifert framing of $\gamma$ and orientations of $K$ and $S$. 
However, such dependence on auxiliary data can be ignored to define the following map.  
\begin{dfn}
Let $\beta$ and $\beta'$ be irreducible critical points of the (perturbed) Chern-Simons functional on $(Y, K)$ and $(Y', K')$ respectively.
We define a map $\mu=\mu_{(W,S, \gamma)}:C^{\alpha}_{*}(Y, K;\Delta_{\mathscr{S}})\rightarrow C^{\alpha}_{*}(Y', K';\Delta_{\mathscr{S}})$ by
\[\mu(\beta)=\sum_{\beta'\in \mathfrak{C}_{\pi}^{*}(Y, K, \alpha)}\sum_{z:\beta\rightarrow \beta'}{\rm deg}\left(h^{\gamma}_{\beta\beta'}|_{M_{z}(W,S;\beta, \beta')_{1}}\right)\lambda^{\kappa_{0}-\kappa(z)}T^{\nu(z)-\nu_{0}}\beta'\]
for each $\beta\in \mathfrak{C}_{\pi}^{*}(Y, K, \alpha)$.

\end{dfn}
The map $\mu$ satisfies the following relation.
\begin{prp}
\[d'\circ \mu-\mu\circ d=0.\]
\end{prp}
\begin{proof}
Consider the compactified 2-dimensional moduli space $M_{z}^{+}(W,S;\beta, \beta')_{2}$ which has the oriented boundary of the following types
\[-\bigcup_{\beta_{1}\in\mathfrak{C}^{*}_{\pi}(Y, K,\alpha)}\bigcup_{z'\circ z''=z}\breve{M}_{z'}^{+}(\beta, \beta_{1})_{i-1}\times M_{z''}^{+}(W,S;\beta_{1},\beta')_{2-i} \]
\[\bigcup_{\beta'_{1}\in \mathfrak{C}^{*}_{\pi}(Y', K',\alpha)}\bigcup_{z'\circ z''=z}M^{+}_{z'}(W, S;\beta, \beta'_{1})_{2-i}\times \breve{M}^{+}_{z''}(\beta'_{1},\beta')_{i-1}\]where $i=1$ or $2$. We count the boundary of the 1-dimensional submanifold $(h^{\gamma}_{\beta\beta'})^{-1}(s)\subset M^{+}(W, S; \beta, \beta')$ for a regular value $s\in S^{1}$. Since the closed loop $\gamma$ is supported on a compact subset of $S$, $(h^{\gamma}_{\beta\beta'})^{-1}(s)$ intersects faces of the boundary of $M^{+}_{z}(W, S;\beta,\beta')$ with $i=1$. Thus we have
\[\#((h^{\gamma}_{\beta\beta'})^{-1}(s)\cap \partial M^{+}_{z}(W, S; \beta,\beta')_{2})=d'\circ \mu-\mu\circ d=0.\]

\end{proof}

We consider the case when $(W, S)=\mathbb{R}\times (Y,K)$ and $\gamma\subset S$ is a curve $\mathbb{R}\times\{y_{0}\}$ where $y_{0}$ is a fixed base point in $K$. Taking holonomy along $\gamma$, we obtain a map
\[h_{\beta_{1}\beta_{2}}:\mathcal{B}(Y, K,\alpha;\beta_{1},\beta_{2})\rightarrow S^{1}\]
by the similar way as (\ref{holmap}), where $\beta_{i}\ (i=1, 2)$ are irreducible critical points of the Chern-Simons functional.

The holonomy map $h_{\beta_{1}\beta_{2}}$  be modified to extend broken trajectories as in \cite{Don}.
Such modification of  $h_{\beta_{1}\beta_{2}}$ near the broken trajectories gives the maps
 \[H_{\beta_{1}\beta_{2}}:\breve{M}(\beta_{1},\beta_{2})_{d}\rightarrow S^{1}\]
 with the following properties. 
 \begin{itemize}
 \item[(i)] $H_{\beta_{1}\beta_{2}}=h_{\beta_{1}\beta_{2}}$ on the complement of a small neighborhood of $\partial \breve{M}^{+}(\beta_{1}, \beta_{2})_{d}$.
 \item[(ii)]$H_{\beta_{1}\beta_{3}}([A_{1}], [A_{2}])=H_{\beta_{1}\beta_{2}}([A_{1}])\cdot H_{\beta_{2}\beta_{3}}([A_{2}])$ on an unparametrized broken trajectory $\breve{M}^{+}(\beta_{1}, \beta_{2})_{i-1}\times \breve{M}^{+}({\beta_{2}, \beta_{3}})_{d-i}$,
 \item[(iii)] $H_{\beta_{1}\beta_{2}}=1$ if ${\rm dim}\breve{M}(\beta_{1}, \beta_{2})=0$. 
 \end{itemize}
 \begin{dfn}
 We define the $v$-map $v:C^{\alpha}_{*}(Y, K;\Delta_{\mathscr{S}})\rightarrow C^{\alpha}_{*}(Y, K;\Delta_{\mathscr{S}})$ by
 \[v(\beta_{1})=\sum_{\beta_{2}\in \mathfrak{C}^{*}_{\pi}}\sum_{z:\beta_{1}\rightarrow \beta_{2}}{\rm deg}\left(H_{\beta_{1}\beta_{2}}|_{\breve{M}_{z}(\beta_{1}, \beta_{2})_{1}}\right)\lambda^{-\kappa(z)}T^{\nu(z)}\beta_{2}.\]
\end{dfn}
$v$-map does not commute with the differential of the chain complex. However, the following relation holds. 
\begin{prp}\label{DonFur}
\[dv-vd-\delta_{2}\delta_{1}=0.\]
\end{prp}
\begin{proof}
We consider the 1-dimensional moduli space \[\breve{M}_{\gamma, z}(\beta_{1}, \beta_{2})_{1}:=\breve{M}_{z}(\beta_{1}, \beta_{2})_{2}\cap (H_{\beta_{1}\beta_{2}})^{-1}(s)\]
for a generic $s\in S^{1}\setminus \{1\}$.
As the argument in the proof of Proposition 3.16 in \cite{DS1}, the boundary of  $\breve{M}_{z}(\beta_{1}, \beta_{2})$ consists of unparametrized broken trajectory of the form,
\[\mathfrak{a}=([A_{1}],[A_{2}])\]
and there are following cases.

\item[(I)] $\mathfrak{a}\in \bigcup_{z'\circ z''=z}\breve{M}_{z'}(\beta_{1}, \beta_{3})_{0}\times\breve {M}_{z''}(\beta_{3},\beta_{2})_{1}$ where $\beta_{3}\in \mathfrak{C}^{*}_{\pi }(Y, K, \alpha)$,
\item[(II)] $\mathfrak{a}\in \bigcup_{z'\circ z''=z}\breve{M}_{z'}(\beta_{1}, \beta_{3})_{1}\times\breve{M}_{z''}(\beta_{3},\beta_{2})_{0}$ where $\beta_{3}\in \mathfrak{C}^{*}_{\pi}(Y, K,\alpha)$,
\item[(III)]$[A]\in\breve{M}_{z}(\beta_{1}, \beta_{2})$ factors through the reducible critical point $\theta_{\alpha}$.

 For the case (I), the corresponding oriented boundary components of $(H_{\beta_{1}\beta_{2}})^{-1}(s)\cap\breve{M}^{+}_{z}(\beta_{1}, \beta_{2})_{2}$  are 
 \begin{eqnarray*} (H_{\beta_{1}\beta_{2}})^{-1}(s)\cap-\left(\bigcup_{\beta_{3}\in \mathfrak{C}^{*}_{\pi}(Y, K, \alpha)}\bigcup_{z'\circ z''=z}\breve{M}_{z'}(\beta_{1}, \beta_{3})_{0}\times \breve{M}_{z''}(\beta_{3}, \beta_{2})_{1}\right)\\
 =-\bigcup_{\beta_{3}\in \mathfrak{C}^{*}_{\pi}(Y, K, \alpha)}\bigcup_{z'\circ z''=z}\breve{M}_{z'}(\beta_{1}, \beta_{3})_{0}\times (H_{\beta_{3}\beta_{2}})^{-1}(s)\cap\breve{M}_{z''}(\beta_{3}, \beta_{2})_{1}.
 \end{eqnarray*}since $H_{\beta_{1}\beta_{3}}=1$. This contributes the term $-\langle vd(\beta_{1}), \beta_{2}\rangle$. 
 For the case (II), a similar argument shows that this contributes the term $\langle dv(\beta_{1}),\beta_{2}\rangle$.
The case (III) requires gluing theory at reducible. Let $U$ be an open subset of $\breve{M}(\beta_{1}, \beta_{2})$ which is given by
\[U=\{[A]\in \breve{M}(\beta_{1}, \beta_{2})|\|A-\pi^{*}\theta_{\alpha}\|_{L^{2}_{1}((-1, 1)\times (Y\setminus K))}<\epsilon\}.\]
$U_{z}$ denotes the restriction of $U$ to $M_{z}(\beta_{1}, \beta_{2})$.
There is a 'ungluing' map 
\[\breve{M}_{z}(\beta_{1}, \beta_{2})\supset U_{z}\xrightarrow{\psi} (0, \infty)\times \bigcup_{z'\circ z''=z} \breve{M}_{z'}(\beta_{1}, \theta_{\alpha})_{0}\times S^{1}\times \breve{M}_{z''}(\theta_{\alpha}, \beta_{3})_{0}.\]
For $T>0$ large enough, consider a subset  $U_{z,T}=\psi^{-1}(\{(t, [A_{1}],s,[A_{2}])\in U_{z}|t>T\})$ of $U_{z}$. Then \[\psi(M_{z}(\beta_{1}, \beta_{2})\cap U_{z,T})=(T,\infty)\times \bigcup_{z'\circ z''=z} \breve{M}_{z'}(\beta_{1}, \theta_{\alpha})\times \{s\}\times \breve{M}_{z''}(\theta_{\alpha}, \beta_{2}).\]
Thus corresponding boundaries of one-manifold $H_{\beta_{1}\beta_{2}}^{-1}(s)\cap\breve{M}^{+}(\beta_{1}, \beta_{2})_{2}$ with induced orientations are given by
\[ -\bigcup_{z'\circ z''=z} \breve{M}_{z'}(\beta_{1}, \theta_{\alpha})_{0}\times \breve{M}_{z''}(\theta_{\alpha}, \beta_{2})_{0}\]The sign counting of this contributes to the term $-\langle \delta_{2}\delta_{1}(\beta_{1}), \beta_{2}\rangle$. 
Finally, we obtain the relation $\langle(dv-vd-\delta_{2}\delta_{1})(\beta_{1}),\beta_{2}\rangle=0$.
\end{proof}
 
 Next, we consider a negative definite pair 
  $(W, S):(Y,K)\rightarrow (Y',K')$ with an embedded curve $\gamma:[0,1]\rightarrow S$ such that $\gamma(0)=p\in K$  and $\gamma(1)=p'\in K'$. We identify $\gamma$ with its image. We define \[\gamma^{+}=(-\infty, 0]\times \{p\}\cup\gamma \cup  [0, \infty)\times\{p'\}\subset S^{+} .\]
Assume that $\beta\in \mathfrak{C}^{*}_{\pi}(Y, K, \alpha)$ and $\beta' \in \mathfrak{C}^{*}_{\pi}(Y', K', \alpha)$.  For each $A\in \mathcal{A}(W, S;\beta, \beta')$, taking the holonomy of $A^{\rm ad}$ along a path $\gamma^{+}$ and we obtain a map
 \[h^{\gamma}_{\beta\beta'}:\mathcal{B}(W,S; \beta, \beta')\rightarrow S^{1}\]
 and its modification \[H^{\gamma}_{\beta\beta'}:M^{+}(W,S;\beta, \beta')_{d}\rightarrow S^{1}\]
so that $H^{\gamma}_{\beta\beta'}=1$ on 0-dimensional unparametrized broken trajectories.
\begin{dfn}
We define a map $\mu=\mu_{(W, S, \gamma)}:C^{\alpha}_{*}(Y,K;\Delta_{\mathscr{S}})\rightarrow C^{\alpha}_{*}(Y',K';\Delta_{\mathscr{S}})$ by
\[\mu(\beta)=\sum_{\beta' \in \mathfrak{C}_{\pi}^{*}(Y', K',\alpha)}\sum_{z:\beta\rightarrow \beta'}{\rm deg}\left(H^{\gamma}_{\beta\beta'}|_{M_{z}(W,S,\beta,\beta')_{1}}\right)\lambda^{\kappa_{0}-\kappa(z)}T^{\nu(z)-\nu_{0}}\beta'.\]
\end{dfn}
\begin{prp}\label{rel mu}
Let $(W, S):(Y, K)\rightarrow (Y', K')$ be a negative definite pair.
$m$ and $ \mu$ denote its corresponding maps as above.
Then
\[d'\mu+\mu d+\Delta_{2}\delta_{1}-\delta'_{2}\Delta_{1}-v'm+m v=0,\]
where $'$ denotes corresponding maps for the pair $(Y', K')$.
\end{prp}
\begin{proof}
Consider a 2-dimensional moduli space $M_{z}^{+}(W, S;\beta, \beta')_{2}$ and its codimension 1 faces.
Firstly, there are two types of end of $M_{z}(W, S;\beta, \beta')_{2}$ in which $[A]\in M(W, S;\beta, \beta')_{2}$ broken at irreducible critical points. 
\item[(I)]
\[\breve{M}^{+}_{z'}(\beta, \beta_{1})_{i-1}\times M^{+}_{z''}(W, S,\beta_{1}, \beta')_{2-i},\]\item[(II)] \[M^{+}_{z'}(W, S, \beta, \beta')_{2-i}\times \breve{M}^{+}_{z''}(\beta_{1}, \beta')_{1-i},\]
where $i=1, 2$. Since
\begin{eqnarray*}&&(H^{\gamma}_{\beta\beta'})^{-1}(s)\cap \bigcup_{\beta_{1}}\bigcup_{z'\circ z''=z}\breve{M}_{z}(\beta, \beta_{1})_{0}\times M^{+}_{z''}(W, S,\beta_{1}, \beta')_{1}\\
&=&\bigcup_{\beta_{1}}\bigcup_{z'\circ z''}\breve{M}_{z'}(\beta, \beta_{1})_{0}\times (H^{\gamma}_{\beta_{1}\beta'})^{-1}(s)\cap M_{ z''}(W, S; \beta_{1}, \beta')_{1},
\end{eqnarray*}the signed counting of the points in $\partial ((H^{\gamma}_{\beta\beta'})^{-1}(s))\cap M^{+}(W,S, \beta, \beta')_{2})$  which are contained in codimension 1 faces of type $\rm (I)$ with $i=1$ contributes to the term $-\langle\mu d(\beta),\beta'\rangle$.
Next, we consider the case that type $\rm (I)$ with $i=2$. Since 
\begin{eqnarray*}
&&(H^{\gamma}_{\beta\beta'})^{-1}(s)\cap \bigcup_{\beta_{1}}\bigcup_{z'\circ z''=z}\breve{M}_{z'}(\beta, \beta_{1})_{1}\times M_{z}(W, S:\beta_{1}, \beta')_{0}\\
&=&\bigcup_{\beta_{1}}\bigcup_{z'\circ z''=z}(H_{\beta\beta_{1}})^{-1}(s)\cap \breve{M}_{z'}(\beta, \beta_{1})_{1}\times M(W, S; \beta_{1}, \beta')_{0},
\end{eqnarray*}
the signed counting of the points in $\partial ((H^{\gamma}_{\beta\beta'})^{-1}(s)\cap M^{+}(W,S;\beta, \beta')_{2})$ which are contained in codimension $1$ faces of type (I) with $i=2$ contributes to  the term $-\langle m v (\beta), \beta'\rangle$.
Similarly, a collection of codimension 1 faces of type (II) contributes to the term $-\langle d'\mu(\beta), \beta'\rangle$ if $i=1$ and $\langle v' m (\beta), \beta'\rangle$ if $i=2$.
Finally, we consider ends of $M_{z}(W, S; \beta, \beta')_{2}$ which  break at reducibles. Such ends are described as in the poof of Proposition \ref{DonFur} and contributes  to the term $-\langle (\Delta_{2}\delta_{1}-\delta_{2}\Delta_{1}(\beta), \beta'\rangle$.

\end{proof}
\begin{cor}\label{def of S-cpx}
$(\tilde{C}^{\alpha}_{*}(Y, K;\Delta_{\mathscr{S}}), \tilde{d}, \chi)$ where \[\tilde{C}^{\alpha}_{*}(Y, K;\Delta_{\mathscr{S}})=C^{\alpha}_{*}(Y, K;\Delta_{\mathscr{S}})\oplus C^{\alpha}_{*-1}(Y, K;\Delta_{\mathscr{S}})\oplus \mathscr{S},\]
\[\tilde{d}=\left[\begin{array}{ccc}
d&0&0\\
v&-d&\delta_{1}\\
\delta_{2}&0&0
\end{array}\right],\  \chi=\left[\begin{array}{ccc}
0&0&0\\
1&0&0\\
0&0&0\end{array}\right]\]forms an $\mathcal{S}$-complex. Moreover, if $(W, S):(Y, K)\rightarrow (Y', K')$ is a given negative definite cobordism and $\alpha$  satisfies $\Delta_{(Y, K)}(e^{4\pi i \alpha})\cdot \Delta_{(Y', K')}(e^{4\pi i \alpha})\neq 0$ then
\[\tilde{m}_{(W, S)}=\left[\begin{array}{ccc}
m&0&0\\
\mu& m&\Delta_{2}\\
\Delta_{1}&0&\eta
\end{array}
\right]\]defines an $\mathcal{S}$-morphism $\tilde{m}_{(W, S)}:\tilde{C}^{\alpha}_{*}(Y, K;\Delta_{\mathscr{S}})\rightarrow \tilde{C}_{*}^{\alpha}(Y',K';\Delta_{\mathscr{S}})$.
\end{cor}
\begin{proof}
The argument in Subsection \ref{delta} and Proposition \ref{DonFur} show that $(\tilde{C}^{\alpha}_{*}(Y, K;\Delta_{\mathscr{S}}),\tilde{d}, \chi)$ is an $\mathcal{S}$-complex.
For a generic perturbation, moduli spaces over the negative definite pair $(W, S)$ are regular at reducible points by Proposition  \ref{perturbation red}, and hence the counting of reducibles $\eta=\eta^{\alpha}(W, S)$ is well defined.
The argument in Subsection \ref{Floer homology}, Proposition \ref{delt d}, and \ref{rel mu} show that $\tilde{m}_{(W, S)}$ is an $\mathcal{S}$-morphism.
\end{proof}

The $\mathcal{S}$-complex  $\tilde{C}^{\alpha}_{*}(Y, K;\Delta_{\mathscr{S}})$ itself depends on the choices of metric and perturbations. 
However, the standard argument (see  \cite[Theorem 3.33]{DS1}) shows that its $\mathcal{S}$-chain homotopy class  is a topological invariant of pairs $(Y, K, p)$ with $\Delta_{(Y, K)}(e^{4\pi i \alpha})\neq 0$, where $K\subset Y$ is an oriented knot in an integer homology $3$-sphere and $p\in K$ is a base point.
The $\mathcal{S}$-chain homotopy type of $\mathcal{S}$-complex itself depends on the choice of a base point, however,
there is a canonical isomorphism between two homology groups of $\mathcal{S}$-complexes which are defined by different choices of base points.

\begin{dfn}
We call 
\[h_{\mathscr{S}}^{\alpha}(Y, K):=h(\tilde{C}^{\alpha}_{*}(Y, K;\Delta_{\mathscr{S}}))\]the Fr{\o}yshov invariant for $(Y, K)$ over $\mathscr{S}$ with a holonomy parameter $\alpha$.
\end{dfn}

The $\mathcal{S}$-complex $\tilde{C}^{\alpha}_{*}(Y,K;\Delta_{\mathscr{S}})$ admits the following connected sum theorem.
\begin{thm}\label{connected sum}
Let $(Y, K)$ and $(Y', K')$ be two oriented knots in integral homology spheres and $\alpha$ be a holonomy parameter such that $\Delta_{(Y, K)}(e^{4\pi i \alpha})\cdot\Delta_{(Y', K')}(e^{4\pi i \alpha})\neq 0$. 
Then
\[\tilde{C}^{\alpha}_{*}(Y\#Y',K\#K';\Delta_{\mathscr{S}})\simeq \tilde{C}^{\alpha}_{*}(Y,K;\Delta_{\mathscr{S}})\otimes_{\mathscr{S}} \tilde{C}^{\alpha}_{*}(Y',K';\Delta_{\mathscr{S}}),\]
where $\simeq$ denotes an $\mathcal{S}$-chain homotopy equivalence.
\end{thm} 
The strategy of proof is essentially same as \cite[Section 6 ]{DS1}. 
We will give the proof of Theorem \ref{connected sum} in the appendix.  

The following corollary gives the proof of Theorem\ref{Froyshov prp}.
\begin{cor}\label{connected sum h}
Let $(Y, K)$ and $(Y', K')$ be knots in integral homology 3-spheres and $\alpha$ be a holonomy parameter such that $\Delta_{(Y, K)}(e^{4\pi i \alpha})\cdot\Delta_{(Y', K')}(e^{4\pi i \alpha})\neq 0$.
Then
\[h_{\mathscr{S}}^{\alpha}(Y\#Y', K\#K')=h_{\mathscr{S}}^{\alpha}(Y, K)+h_{\mathscr{S}}^{\alpha}(Y', K').\]Moreover, if there are two negative definite cobordism $(W, S):(Y, K)\rightarrow (Y', K')$ and $(W', S'):(Y', K')\rightarrow (Y, K)$, then
\[h^{\alpha}_{\mathscr{S}}(Y, K)=h^{\alpha}_{\mathscr{S}}(Y', K').\]
\end{cor}
\begin{proof}
The first statement follows from Theorem \ref{connected sum} and Proposition \ref{h sum}.
The later half statement follows from Corollary \ref{def of S-cpx} and Proposition \ref{Froyshov ineq}.
\end{proof}

The filtered construction can be applied to $\mathcal{S}$-complex for the coefficient $\mathscr{R}_{\alpha}$.
A fixed lift $\tilde{\theta}_{\alpha}$ of reducible flat connection can be identified with $1\in \mathscr{R}_{\alpha}$, and $\mathscr{R}_{\alpha}$ itself can be identified with all set of lifts of $\theta_{\alpha}$.
We extend the $\mathbb{R}$-grading  to $\tilde{C}^{\alpha}_{*}(Y, K;\Delta_{\mathscr{R}_{\alpha}})$.
First, we define 
\[{\rm deg}_{\mathbb{R}}(\delta)=\begin{cases}\max\{r\vert a_{r}\neq 0\}& \text{if}\ \delta\neq 0\\
-\infty &\text{if}\ \delta=0
\end{cases}\]
for  $\delta=\sum_{r} a_{r}\lambda^{r}\in \mathscr{R}_{\alpha}$, $a_{r}\in \mathbb{Z}[T^{-1}, T]\!]$ .
Then for $(\beta,\gamma, \delta)\in C^{\alpha}_{*}(Y,K;\Delta_{\mathscr{R}_{\alpha}})\oplus C^{\alpha}_{*-1}(Y, K;\Delta_{\mathscr{R}_{\alpha}})\oplus \mathscr{R}_{\alpha}$, we define 
\[\widetilde{\rm deg}_{\mathbb{R}}(\beta, \gamma, \delta):=\max\{{\rm deg}_{\mathbb{R}}(\beta),{\rm deg}_{\mathbb{R}}(\gamma), {\rm deg}_{\mathbb{R}}(\delta) \}.\]
Obviously, we have the following proposition.
\begin{prp}
If we fix lifts of each critical points  of the Chern-Simons functional, then the $\mathcal{S}$-complex $\tilde{C}^{\alpha}_{*}(Y, K;\Delta_{\mathscr{R}_{\alpha}})$  admits $\mathbb{Z}\times \mathbb{R}$-bigrading.
\end{prp}
Note that the $\mathbb{R}$-grading of $\mathcal{S}$-complex extends to tensor products of $\mathcal{S}$-complexes in natural way.

The filtered $\mathcal{S}$-complex $\tilde{C}^{\alpha}_{*}(Y, K;\Delta_{\mathscr{R}_{\alpha}})^{[R_{0}, R_{1}]}$ for  $R_{0}, R_{1}\in(\mathbb{R}\cup\{\pm \infty\})\setminus\mathscr{C}^{*}$ with $R_{0}<R_{1}$ can be defined as follows.
Put $\tilde{C}^{\alpha}_{*}(Y, K;\Delta_{\mathscr{R}_{\alpha}})^{[-\infty, R]}:=\{(\beta, \gamma, \delta)\in \tilde{C}^{\alpha}_{*}(Y, K;\Delta_{\mathscr{R}_{\alpha}}) \vert\widetilde{{\rm deg}}_{\mathbb{R}}(\beta, \gamma, \delta)<R\}$ and
\[\tilde{C}^{\alpha}_{*}(Y, K;\Delta_{\mathscr{R}_{\alpha}})^{[R_{0}, R_{1}]}:=\tilde{C}^{\alpha}_{*}(Y, K;\Delta_{\mathscr{R}_{\alpha}})^{[-\infty, R_{1}]}/\tilde{C}^{\alpha}_{*}(Y, K;\Delta_{\mathscr{R}_{\alpha}})^{[-\infty, R_{0}]}\]
for $R_{0}<R_{1}$.

\subsection{ Cobordism maps for immersed surfaces}
Let $(W, S): (Y, K)\rightarrow (Y', S')$ be a cobordism of pairs where $S$ is possibly immersed.
Blowing up all double points of $S$, we obtain a cobordism of pairs $(\bar{W}, \bar{S})$ where $\bar{S}$ is a embedded surface.
\begin{dfn}
We say $(W, S)$ is negative definite if its blow up $(\bar{W}, \bar{S})$ is a negative definite.
We define a cobordism map for a negative definite cobordism $(W,S)$ where $S$ is possibly  immersed surface as follows,
\[\tilde{m}_{(W, S)}:=\tilde{m}_{(\bar{W},\bar{S})}.\]
\end{dfn}
Let us describe the relation between operations on immersed surface cobordisms and  induced $\mathcal{S}$-morphisms.

\begin{prp}\label{immersed map}
Let $\mathscr{S}$ be an integral domain over the ring $\mathscr{R}_{\alpha}$.
Assume that $(W,S)$ is a negative definite pair over $\mathscr
{S}$  where $S$ is a possibly immersed surface. 
Let $S^{*}$ be a surface obtained from $S$ by a positive, negative twist move or a finger move.
Then $\tilde{m}_{(W, S^{*})}$ is $\mathcal{S}$-chain homotopic to $\tilde{m}_{(W, S)}$ up to the multiplication of a unit element in $\mathscr{S}$.
\end{prp}
The definition of positive (negative) twist move and finger move can be found in \cite{freedman2014topology}.
\begin{proof}
Since monotonicity condition cannot be assumed in our setting, we have to modify the argument in \cite{Kr97}.

\textit{(i) positive twist move.}
Consider a blow-up at a positive self-intersection point  $(\bar{W},\bar{S}^{*})=(W, S)\# (\overline{\mathbb{CP}^{2}}, S_{2})$ where $S_{2}$ is an embedded sphere whose homology class is $-2e\in H_{2}(\overline{\mathbb{CP}}^{2};\mathbb{Z})$.
Note that $\mathcal{R}_{\alpha}(S^{3}\setminus S^{1},SU(2))=\{\theta_{\alpha}\}$ for $\alpha\in (0, \frac{1}{2})$. 
Assume that $(W,S^{*})$ has a metric $g_{T}$ such that $(S^{3}, S^{1})$ has a neighborhood which is isometric to $[-T,T]\times(S^{3}, S^{1})$, where $T>0$ is large enough.
Let $A_{T}$ be an instanton on $\{(W, S^{*}),g_{T}\}$ which is contained in the $0$-dimensional moduli space. $A_{\infty}$ denotes the limiting instanton of $A_{T}$ with respect to $T\rightarrow \infty$, and $A_{1}$, $A_{2}$ denote its restriction to components obtained by attaching cylindrical ends on $(W,S)$ and $(\overline{\mathbb{CP}^{2}}, S_{2})$ respectively. 
Then we have \[{\rm ind}D_{A_{1}}+1+{\rm ind}D_{A_{2}}={\rm ind}D_{A_{\infty}}\leq 0.\]
The last inequality is essentially follows from  \cite[Corollary 8.4]{KM93} and our assumption.
The index formula for a closed pair $(\overline{\mathbb{CP}^{2}}, S_{2})$ shows that ${\rm ind}D_{A_{2}}\underset{(4)}{\equiv}-1$ and we have ${\rm ind}D_{A_{2}}=-1$.
By the perturbation, the instanton $A_{2}$ on $\overline{\mathbb{CP}^{2}}$ satisfies $H^{1}_{A_{2}}=H^{2}_{A_{2}}=0$, and the gluing along $\mathcal{R}_{\alpha}(S^{3}\setminus S^{1})=\{\theta_{\alpha}\}$ is unobstructed. 
The moduli space $M(W, S^{*};\beta, \beta')_{0}$ is diffeomorphic to
\[M(W, S;\beta, \beta')_{0}\times M^{\alpha}(\overline{\mathbb{CP}^{2}}, S_{2})_{0}. \]
Note that there is a diffeomorphism $M^{\alpha}(\overline{\mathbb{CP}^{2}}\setminus D^4 , S_{2}\setminus D^{2};\theta_{\alpha})_{0}\cong M^{\alpha}(\overline{\mathbb{CP}^{2}}, S_{2})_{0}$ by the removable singularity theorem.
Since ${\rm ind}D_{A_{2}}=-1$, $A_{2}$ is a minimal reducible. Moreover, minimal reducibles on $(\overline{\mathbb{CP}^{2}}, S_{2})$ define elements in $M^{\alpha}(\overline{\mathbb{CP}^{2}}, S_{2})_{0}$. 
Counting elements in moduli space $M(W, S;\beta, \beta')_{0}$ defined by limiting metric $\lim_{T\rightarrow \infty}g_{T}$
contributes the relation \[\langle\tilde{m}^{\infty}_{(\bar{W},\bar{S}^{*})}\beta, \beta'\rangle=\eta^{\alpha}(\overline{\mathbb{CP}}^{2}, S_{2})\langle\tilde{m}_{(W, S)}\beta, \beta'\rangle.\] Since
\[
\eta^{\alpha}(\overline{\mathbb{CP}^{2}}, S_{2})=
\begin{cases}
1-\lambda^{4\alpha-1}T^{4}, \  \alpha\leq {1/4}\\ 
\lambda^{1-4\alpha}T^{-4}-1, \  \alpha> {1/4},
\end{cases}
\]
 $\eta^{\alpha}(\overline{\mathbb{CP}}^{2}, S_{2})$ is a unit in $\mathscr{S}$.
 The $\mathcal{S}$-chain homotopy between $\tilde{m}^{\infty}_{(\bar{W}, \bar{S}^{*})}$ and $\tilde{m}_{(\bar{W},\bar{S}^{*})}$ is given by considering one parameter family of moduli spaces.
 In particular, $\eta^{\alpha}(\overline{\mathbb{CP}}^{2}, S_{2})$ has the top term $1$ and hence the statement follows. 
\textit{(ii) negative twist move.}
In this case, we changes $S_{2}$ in the above argument to an embedded sphere $S_{0}$ whose homology class is trivial.  Thus we obtain $\eta^{\alpha}(\overline{\mathbb{CP}^{2}}, S_{0})=1$.

\textit{(iii) finger move.}
Consider a decomposition $(W, S)=(W_{1}, S_{1})\cup (W_{2}, S_{2})$ where $W_{2}=D^{4}$ and $S_{2}=D^{2}\sqcup D^{2}$.
Let $(\bar{W},\bar{S^{*}})=(W_{1}, S_{1})\cup (W'_{2}, S'_{2})$ be a twice blow-up of $(W, S^{*})$. 
In this case, $W_{2}$ is a 4-manifold obtained by removing a disk from $\overline{\mathbb{CP}^{2}}\#\overline{\mathbb{CP}^{2}}$ and $S_{2}$ is two disjoint disks.
Note that $\mathcal{R}_{\alpha}:=\mathcal{R}_{\alpha}(S^{3}\setminus (S^{1}\sqcup S^{1}), SU(2))\cong [0, \pi]$ for fixed $\alpha\in (0, \frac{1}{2})$, the interior of $\mathcal{R}_{\alpha}$ consists of irreducible flat connections and two end points are reducible.
Moreover, the end point map $r_{1}:M(W_{1},S_{1};\beta, \beta')_{0}\rightarrow \mathcal{R}_{\alpha}$ has its image in the irreducible part of $\mathcal{R}_{\alpha}$. 
See \cite[Lemma 3.2]{Kr97} for details.

We claim that the counting of two moduli spaces $M(W_{1}, S_{1};\beta,\beta')_{0}$ and $M(W, S;\beta, \beta')_{0}$ can be identified up to the multiplication by a unit element  in $\mathscr{S}$.
Firstly, we define a $\mathcal{S}$-morphism $\tilde{m}_{(W_{1}, S_{1})}$ as
\[\langle {m}_{(W_{1}, S_{1})}\beta,\beta' \rangle=\sum_{z}\#M_{z}(W_{1}, S_{1};\beta, \beta')\lambda^{\kappa_{0}-\kappa(z)}T^{\nu(z)-\nu_{0}}\beta'\] and similarly for other components in $\tilde{m}_{(W_{1}, S_{1})}$. 
Here $\kappa(z)$, $\kappa_{0}$, $\nu(z)$, and $\nu_{0}$ are similarly defined as in Subsection \ref{cobordism}.
We have to modify the argument in \cite{Kr97} which is related to that the unobstructed gluing along  the pair $(S^{3}, S^{1}\sqcup S^{1})$.
For $\rho \in \mathcal{R}_{\alpha}$ which is in the image of $r_{1}$, we take its extension  $A_{\rho}$ to $(D^{4}, D^{2}\sqcup D^{2})$. 
Consider a double $(S^{4},S^{2}\sqcup S^{2})=(D^{4}, D^{2}\sqcup D^{2})\cup_{(S^{3}, S^{1})}(D^{4},D^{2}\sqcup D^{2})$. 
Then ${\rm ind}D_{A_{\rho}\#A_{\rho}}=2{\rm ind}D_{A_{\rho}}+1$ by the  gluing formula.  
Consider the pair of connected sum $(S^{4}, S^{2}\sqcup S^{2})=(S^{4}, S^{2})\#_{(S^{3}, \emptyset)}(S^{4}, S^{2})$. 
Then ${\rm ind}D_{A_{\rho}\#A_{\rho}}=2{\rm ind}D_{A_{\rho}\#A_{\rho}}|_{(S^{4}, S^{2})}+3$ and the left hand side is equal to $1$. 
Hence we have ${\rm ind}D_{A_{\rho}}=0$.
Thus the relation 
\[{\rm ind}D_{A_{\rho}}+{\rm dim}H^{1}_{\rho}=-{\rm dim}H^{0}_{A_{\rho}}+{\rm dim}H^{1}_{A_{\rho}}-{\rm dim}H^{2}_{A_{\rho}}\]
tells us that $H^{2}_{A_{\rho}}=0$ since ${\rm dim}H^{0}_{A_{\rho}}=0$ and ${\rm dim}H^{1}_{A_{\rho}}=1$. 
Here $H^{1}_{\rho}$ is the cohomology with local coefficient system associated to the flat connection $\rho$.
Thus the Morse-Bott gluing of instantons over $(W_{1},S_{1})$ and $(W_{2}, S_{2})$ is unobstructed.
For a metric on $(W, S)$ with a long neck along the cylinder $[0, 1]\times (S^{3}, S^{1}\sqcup S^{1})$, we have a diffeomorphism,\[M(W, S;\beta,\beta')_{0}\cong M(W_{1}, S_{1};\beta,\beta')_{0}\ {}_{r}\times_{r'}M^{\alpha}(D^{4}, D^{2}\sqcup D^{2})_{1} \]where \[r: M(W_{1}, S_{1};\beta, \beta')_{0}\rightarrow \mathcal{R}_{\alpha}\] and \[r':M^{\alpha}(D^{4}, D^{2}\sqcup D^{2})_{1}\rightarrow \mathcal{R}_{\alpha} \] are restriction maps.
For simplicity, we consider the case $\alpha\leq \frac{1}{4}$.
Since flat connections on  $(S^{3},S^{1}\sqcup S^{1})$ uniquely extends to $(D^{4}, D^{2}\sqcup D^{2})$, the induced cobordism map has the form;
\[\tilde{m}_{(W, S)}=\left(1+\sum_{k>0}c_{k}Z^{-k}\right)\cdot \tilde{m}_{(W_{1}, S_{1})} \]
where $c_{k}\in \mathbb{Z}$ and $Z=\lambda^{1-4\alpha}T^{-4}$. 
Thus, $\tilde{m}_{(W, S)}$ and $\tilde{m}_{(W_{1}, S_{1})}$ are differed by the multiplication of a unit element in $\mathscr{S}$.

Assume that a cobordism of pairs $(\bar{W}, \bar{S}^{*})$ is equipped with a  metric such that $(\bar{W},\bar{S^{*}})$ has a long neck along $(S^{3}, S^{1}\sqcup S^{1})$.
Then the moduli space $M(\bar{W}, \bar{S}^{*}; \beta, \beta')_{0}$ decomposes into a union of  fiber products
\[M(W_{1}, S_{1}, \beta, \beta')_{d}\ {}_{r_{1}}{\times}_{r_{2}}M^{\alpha}(W'_{2}; S'_{2})_{d'}\]with $d+d'=1$, where
\[r_{1}:M(W_{1}, S_{1}, \beta, \beta')_{d}\rightarrow \mathcal{R}_{\alpha}\]
\[r_{2}:M^{\alpha}(W'_{2}, S'_{2})_{d'}\rightarrow \mathcal{R}_{\alpha} \]
are restriction maps.
Since $d'\underset{(4)}\equiv1 $ by the index formula, we have $d=0$ and $d'=1$. 
Thus there is the coefficient $c\in \mathscr{S}$ such that $\tilde{m}_{(\bar{W}, \bar{S^{*}})}=c\tilde{m}_{(W, S)}$.
Consider a special case of finger move which is a composition of one positive twist move and one negative twist move.
In this case, the coefficient $c$  is turned out to be $1-\lambda^{4\alpha-1}T^{4}$ for $\alpha\leq \frac{1}{4}$ and $\lambda^{1-4\alpha}T^{-4}-1$ for $\alpha> 1/4$ by the argument above.  
Finally, we conclude that there is a unit element $c\in \mathscr{S}$ such that $\tilde{m}_{(W, S^{*})}$ and $c\tilde{m}_{(W, S)}$ are $\mathcal{S}$-chain homotopic.
\end{proof}

\section{Generators of Floer chains\label{gen}}
Let us describe the outline of this subsection.
We will discuss conjugacy classes of representation
$$\rho: \pi_{1}(Y\setminus K) \rightarrow SU(2),$$
with the condition 
\[
\rho(\mu_{K})\sim\left[\begin{array}{cc}
e^{2\pi i \alpha }&0\\
0&e^{-2\pi i \alpha}

\end{array}\right].
\]

In this section, we write $[\rho]$ for its conjugacy class to distinct elements in ${\rm Hom}(\pi_{1}(Y\setminus K), SU(2))$ and $\mathcal{R}(Y\setminus K, SU(2))$.
Firstly, we introduce the method of taking cyclic branched covering. Considering a knot in an integral homology 3-sphere $(Y, K)$, we can take a cyclic branched covering $\tilde{Y}_{r}(K)$ over $Y$ branched along $K$. Let $N(K)$ be a tubular neighborhood of $K\subset Y$, and $V=Y\setminus {\rm int}(N(K))$ be its exterior. $\tilde{V}$ denotes $r$-fold unbranched covering over $V$ with $\pi_{1}(\tilde{V})$ is a kernel of $\pi_{1}(Y\setminus K)\rightarrow H_{1}(Y\setminus K, \mathbb{Z})\rightarrow \mathbb{Z}/r\mathbb{Z}$. $N(K)$ and $\tilde{V}$ have a torus boundary, and let $h: \partial N(K)\rightarrow \partial \tilde{V}$ be a gluing map which sends $\mu_{K}$ meridian of $K$ to its lift $\tilde{\mu_{K}}$. Then the $r$-fold cyclic branched covering over $Y$ is defined by
$$\tilde{Y}_{r}(K)=N(K)\cup_{h} \tilde{V}. $$
Let $\tau:\tilde{Y}_{r}\rightarrow \tilde{Y}_{r}$ be a covering transformation. $\tau$ induces action $\tau_{*}$ on $\pi_{1}(\tilde{Y}_{r})$ and on ${\rm Hom}^{*}(\pi_{1}(\tilde{Y}_{r}), SO(3))$ 
by $\tau^{*}(\rho)=\rho\circ \tau_{*}$. This also defines an action on $\mathcal{R}^{*}(\tilde{Y}_{r}, SO(3))$ by $\tau^{*}[\rho]=[\tau^{*}(\rho)]$. We define following subset of $\mathcal{R}^{*}(\tilde{Y}_{r}, SO(3))$.
$$\mathcal{R}^{*,\tau}(\tilde{Y}_{r}, SO(3))=\{[\rho]\in \mathcal{R}^{*}(\tilde{Y}_{r}, SO(3))|\tau^{*}[\rho]=[\rho]\}$$
$$\mathcal{R}^{*, \tau}(\tilde{Y}, SU(2))=\{[\rho] \in \mathcal{R}^{*}(\tilde{Y}, SU(2))|\ {\rm Ad}[\rho] \in \mathcal{R}^{*,\tau}(\tilde{Y}, SO(3))\}.$$
The aim of subsection \ref{1-1} is giving the construction of the lift map  $$\Pi: \bigsqcup_{1\leq l \leq r-1}\mathcal{R}^{*}_{\frac{l}{2r}}(Y\setminus K, SU(2))\rightarrow \mathcal{R}^{*, \tau}(\tilde{Y}_{r}, SU(2))$$which sends singular flat connections to non-singular flat connections on the cyclic branched covering  of knot $K\subset Y$. We will see that the lift map $\Pi$ satisfies the following proposition.
\begin{prp}\label{lift}
Assume that $r$-fold cyclic branched covering $\tilde{Y}_{r}$ of a knot $K$ in an integral homology 3-sphere $Y$  is integral homology 3-sphere. Then the lifting map $\Pi$ gives a two-to-one correspondence
$$\Pi:\bigsqcup_{1\leq l \leq r-1}\mathcal{R}^{*}_{\frac{l}{2r}}(Y\setminus K, SU(2))\rightarrow \mathcal{R}^{*, \tau}(\tilde{Y}_{r}, SU(2)).$$
\end{prp}
This is a generalization of the argument in \cite{CSa}.

Let $X(K)$ be a complement of the tubular neighborhood of a knot $K\subset S^{3}$. Its boundary $\partial X(K)$ is a torus.
In Subsection \ref{non-deg res}, we will show that the restriction map $r:\mathcal{R}^{*}(S^{3}\setminus T_{p, q}, SU(2))\rightarrow \mathcal{R}(\partial X(T_{p, q}), SU(2))$ is a smooth immersion of one-manifold without any perturbation of flat connections, using the setting of gauge theory by Herald \cite{her94} and computations of group cohomology of $\pi_{1}$.
In Subsection \ref{TLsig}, we will give a proof of Theorem \ref{absolute counting} using the results in  The construction of the lifting map \label{1-1}

\subsection{ The construction of the lifting map\label{1-1}}

We assign  the second Stiefel-Whitney class $w\in H^{2}(Y,\mathbb{Z}_{2})$ to $[\rho] \in \mathcal{R}(Y, SO(3))$. We can construct a flat bundle $E=\tilde{Y}\times_{\rho} \mathbb{R}^{3}$ from $SO(3)$-representation $\rho$ and define $w([\rho]):=w_{2}(E)\in H^{2}(Y,\mathbb{Z}_{2})$, where $w_{2}(E)$ is a second Stiefel-Whitney class of $E$.  If $w(\rho)=0$ then $SO(3)$-bundle $E$ lifts to $SU(2)$-bundle $F$. Let $P$ and $Q$ be corresponding principal bundle of $E$ and $F$ respectively. The natural map $p:Q\rightarrow P$ is fiberwise double covering map. Let $\theta_{\rho}$ be a connection form on $P$ which corresponds to  flat connection $\rho$. Then $p^{*}\theta_{\rho}$ defines a flat connection on $Q$. Thus each element of $\mathcal{R}(Y, SO(3))$ lifts to $\mathcal{R}(Y, SU(2))$ if its second Stiefel-Whitney class vanishes.

\begin{prp}\label{action}
Let $X$ be $Y$ or $Y\setminus K$. Then there is an action of $H^{1}(X, \mathbb{Z}_{2})$ on $\mathcal{R}(X, SU(2))$ and the map ${\rm Ad}:\mathcal{R}(X, SU(2))\rightarrow \mathcal{R}^{0}(X, SO(3))$ induces a bijectiton $$\mathcal{R}(X, SU(2))/H^{1}(X,\mathbb{Z}_{2})\cong\mathcal{R}^{0}(X, SO(3)).$$ 
Here $\mathcal{R}^{0}(X, SO(3))$ denotes the set of conjugacy classes of $SO(3)$-representations whose second Stiefel-Whitney class vanishes. 
\end{prp}
\begin{proof}
Let $\rho:\pi_{1}(X)\rightarrow SO(3)$ be a representation whose second Stiefel-Whitney class vanishes and $\tilde{\rho}:\pi_{1}(X)\rightarrow SU(2)$
be its $SU(2)$-lift. 
Consider another lift $\tilde{\rho}':\pi_{1}(X)\rightarrow SU(2)$. Then there is a map $\chi: \pi_{1}(X)\rightarrow \{\pm1\}$ such that $\tilde{\rho}'(g)=\chi(g)\tilde{\rho}(g)$ for any $g\in \pi_{1}(X)$. We can directly check that $\chi$
 is a homomorphism and determine an element $\chi \in {\rm Hom}(\pi_{1}(X), \mathbb{Z}_{2})=H^{1}(X, \mathbb{Z}_{2})$. Conversely, two $SU(2)$-representation $\sigma_{1}, \sigma_{2}: \pi(X)\rightarrow SU(2)$ such that there exist $\chi \in {\rm Hom}(\pi_{1}(X), \mathbb{Z}_{2}) $ and satisfying $\sigma_{1}(g)=\chi(g)\sigma_{2}(g)$ for any $g \in \pi_{1}(X)$ induces same $SO(3)$-representation. We define an action of $H^{1}(X, \mathbb{Z}_{2})$ on ${\rm Hom}(\pi_{1}(X), SU(2))$ by $\sigma \mapsto \chi \cdot\sigma$, where $(\chi\cdot \sigma)(g)=\chi(g)\sigma(g)$ for $g\in\pi_{1}(X)$. The action of $\chi$ commutes with the conjugacy action and descends to $\mathcal{R}(X, SU(2))$.  Thus we get the statement.
\end{proof}
Note that the action of $H^{1}(Y\setminus K,\mathbb{Z}_{2})$ coincides with the flip symmetry.
From Proposition \ref{action}, we get the following corollary.
\begin{cor}
For an integral homology $3$-sphere $Y$, all elements in $\mathcal{R}(Y,SO(3))$ have a unique lift in $\mathcal{R}(Y, SU(2))$.
\end{cor}
\begin{proof}
Since $H^{2}(Y,\mathbb{Z}_{2})=0$, the second Stiefel-Whitney class of $[\rho] \in \mathcal{R}(Y, SO(3))$ vanishes, and $\rho$ lifts to an $SU(2)$-representation. By Proposition \ref{action}, this lift is unique since $H^{1}(Y,\mathbb{Z}_{2})=0$.  
\end{proof}
If $[\rho] \in \mathcal{R}(Y\setminus K,SU(2))$ satisfies 
\[\rho(\mu_{K})\sim 
\left[\begin{array}{cc}
e^{2\pi i \alpha}&0\\
0&e^{-2\pi i \alpha}
\end{array}\right]
\]
then induced $SO(3)$-representation satisfies
\begin{equation}
{\rm Ad \rho}(\mu_{K})\sim
\left[\begin{array}{ccc}
1&0&0\\
0&{\rm cos}(4\pi \alpha)& -{\rm sin}(4\pi \alpha)\\
0&{\rm sin}(4\pi \alpha)&{\rm cos}(4\pi \alpha)
\end{array}
\right].\label{so3}
\end{equation}

Let $\mathcal{R}_{\alpha}(Y\setminus K, SO(3))$ be a subset of $\mathcal{R}(Y\setminus K, SO(3))$ whose elements are represented by $SO(3)$-representations of $\pi_{1} (Y\setminus K)$ such that their images of $\mu_{K}$ are conjugate to the right  hand side of (\ref{so3}).

Before proceeding the argument, we introduce orbifold fundamental group of $Y\setminus K$. (It appears in \cite{CSa}, \cite{collin1999instanton}, for example.)
\begin{dfn}
$\pi^{V}_{1}(Y, K; r):=\pi_{1}(Y\setminus K)/\langle \mu_{K}^{r} \rangle$ is called orbifold fundamental group of $Y\setminus K$.
\end{dfn}
\begin{prp}\label{orbseq}
$\pi^{V}_{1}(Y, K; r)$ admits the following split short exact sequence.
$$1\rightarrow \pi_{1}(\tilde{Y}_{r})\rightarrow \pi_{1}^{V}(Y, K; r)\rightarrow \mathbb{Z}/r\rightarrow 1.$$
\end{prp}
\begin{proof}
Let $\tilde{K}\subset\tilde{Y}_{r}$ be a branched locus. Since $\tilde{Y}_{r}\setminus \tilde{K}\rightarrow Y\setminus K$ is a regular covering, there is a exact sequence
$$1\rightarrow \pi_{1}(\tilde{Y}_{r}\setminus \tilde{K})\rightarrow \pi_{1}(Y\setminus K)\rightarrow \mathbb{Z}/r\rightarrow 1.$$
Applying Van-Kampen theorem to the decomposition $\tilde{Y}_{r}\setminus {\rm int}N(\tilde{K})\cup N(\tilde{K})$, we have $\pi_{1}(\tilde{Y}_{r})=\pi_{1}(\tilde{Y}_{r}\setminus\tilde{K})/\langle\mu_{\tilde{K}}\rangle$. Since $\pi_{1}(\tilde{Y}_{r}\setminus \tilde{K})\rightarrow \pi_{1}(Y\setminus K)$  maps $\mu_{\tilde{K}}$ to $\mu_{K}^{r}$, this induces
$1\rightarrow \pi_{1}(\tilde{Y}_{r})\rightarrow \pi^{V}_{1}(Y, K;r)$. Since $\pi_{1}(Y\setminus K)\rightarrow \mathbb{Z}/r$ maps $\mu_{K}^{r}\mapsto 1$, this induces $\pi_{1}^{V}(Y, K; r)\rightarrow \mathbb{Z}/r$ which sends $\mu_{K}$ to a generator of $\mathbb{Z}/r$. The spitting $\mathbb{Z}/r\rightarrow \pi^{V}_{1}(Y, K;r)$ sends a generator of $\mathbb{Z}/r$ to $\mu_{K}$.
\end{proof}
\begin{lem}\label{natcrr}
There is a natural one to one correspondence
$$\mathcal{R}^{*}(\pi^{V}_{1}(Y, K;r), SO(3))\cong \bigsqcup_{l=1}^{r-1}\mathcal{R}_{\frac{l}{2r}}^{*}(Y\setminus K,SO(3)).$$
\end{lem}
\begin{proof}
Let $\rho:\pi_{1}(Y\setminus K)\rightarrow SO(3)$ be a representation with $[\rho]\in \mathcal{R}_{\frac{l}{2r}}^{*}(Y\setminus K,SO(3))$. Then it factors through $\pi^{V}_{1}(Y, K; r)$. Conversely, any representation $\sigma:\pi^{V}_{1}(Y, K; r)\rightarrow SO(3)$ satisfies 
\[
\sigma(\mu_{K})\sim \left[\begin{array}{ccc}
1&0&0\\
0&{\rm cos}(4\pi \alpha)& -{\rm sin}(4\pi \alpha)\\
0&{\rm sin}(4\pi \alpha)&{\rm cos}(4\pi \alpha)
\end{array}
\right]
\]
where $\alpha=\frac{l}{2r}$  for some $0<l<r$. 
Thus $\sigma$ defines a representation of $\pi_{1}(Y\setminus K)$ with above condition. 
\end{proof}

\begin{prp}\label{bij}
There is a bijection
$$\bigsqcup_{l=1}^{r-1}\mathcal{R}^{*}_{\frac{l}{2r}}(Y\setminus K, SO(3))\cong \mathcal{R}^{*,\tau}(\tilde{Y}_{r}, SO(3)).$$
\end{prp}

\begin{proof}
Since there is a natural one to one correspondence in Lemma \ref{natcrr}, we only have to construct
$$\mathcal{R}^{*}(\pi^{V}_{1}(Y, K; r), SO(3))\xrightarrow{\cong}\mathcal{R}^{*, \tau}(\tilde{Y}_{r}, SO(3)).$$
This is induced from $\pi_{1}(\tilde{Y}_{r})\xrightarrow{i} \pi^{V}_{1}(Y, K ; r)$ in the short exact sequence in Proposition \ref{orbseq}. We claim that if $\rho:\pi^{V}_{1}(Y, K; r)\rightarrow SO(3)$ is irreducible then $\rho\circ i$ is also irreducible. Since $\tilde{Y}_{r}$ is an integral homology sphere, any reducible $SO(3)$-representation of $\pi_{1}(\tilde{Y}_{r})$ is the trivial representation.  If $\rho\circ i$ is trivial, then $\rho$ factors through $\pi^{V}_{1}(Y, K; r)/i(\pi_{1}(\tilde{Y}_{r}))\cong\mathbb{Z}/r$  and hence reducible. This is a contradiction.

We will construct the inverse correspondence of the above. 
Let $\sigma$ be a $SO(3)$-representation of $\pi_{1}(\tilde{Y}_{r})$ which represents an element in $\mathcal{R}^{*, \tau}(\tilde{Y}_{r}, SO(3))$. Since the conjugacy class of $\sigma$ is fixed by the induced action of $\tau$, there is a matrix $A\in SO(3)$ such that
$$\tau^{*}\sigma(u)=A\sigma (u) A^{-1}$$for any $u\in \pi_{1}(\tilde{Y}_{r})$. $A$ is uniquely determined since $\sigma$ is irreducible and has the trivial stabilizer $\{1\}$ in $SO(3)$, and $A$ is conjugate to the matrix of the form,
\begin{equation}
\left[\begin{array}{ccc}
1&0&0\\
0&{\rm cos}( \frac{2\pi l}{r})& -{\rm sin}(\frac{2\pi l}{r})\\
0&{\rm sin}(\frac{2\pi l}{r})&{\rm cos}(\frac{2\pi l}{r})
\end{array}
\right].
\end{equation}
Since $\tau$ has order $r$, we get a relation
$$\sigma(u)=A^{r}\sigma(u)A^{-r}$$ for any $u\in \pi_{1}(\tilde{Y}_{r})$, and we get $A^{r}=1$ using the irreducibility of $\sigma$. Thus we  can assign unique order $r$ element  $A_{\sigma} \in SO(3)$ for each $\sigma$. Finally, we assign a representation 
\begin{equation}
\begin{array}{ccccc}
\bar{\sigma}:\pi_{1}^{V}(Y, K;r)&\cong&\pi_{1}(\tilde{Y}_{r})\rtimes \mathbb{Z}/r&\rightarrow & SO(3)\\
&&&&\\
&&(u, t^{k})&\mapsto & \sigma(u)\cdot A_{\sigma}^{k}

\end{array}\label{ind-rep}
\end{equation}
to a given representation $\sigma$, where $t\in \mathbb{Z}/r$ is a generator.   This satisfies $\bar{\sigma}(\mu_{K})=A_{\sigma}$. The above construction gives the inverse of
$\mathcal{R}^{*}(\pi^{V}_{1}(Y, K; r), SO(3))\ni [\rho]\mapsto [\rho\circ i]\in \mathcal{R}^{*, \tau}(\tilde{Y}_{r}, SO(3))$.
\end{proof}
We write 
$$\Pi':\bigsqcup_{l=1}^{r-1}\mathcal{R}^{*}_{\frac{l}{2r}}(Y\setminus K, SO(3))\xrightarrow{\cong}\mathcal{R}^{*,\tau}(\tilde{Y}_{r}, SO(3))$$
for the bijection constructed above.

Now we can describe the definition of $\Pi$. 
Let $K$ be a knot in an integral homology $3$-sphere $Y$, and assume that $\tilde{Y}_{r}$ is also an integral homology 3-sphere.
\begin{dfn}

$\Pi:\bigsqcup_{l=1}^{r-1}\mathcal{R}_{\frac{l}{2r}}^{*}(Y\setminus K, SU(2))\rightarrow \mathcal{R}^{*, \tau}(\tilde{Y}_{r}, SU(2))$  is given by the following compositions.
\begin{eqnarray*}
\bigsqcup_{l=1}^{r-1}\mathcal{R}_{\frac{l}{2r}}^{*}(Y\setminus K, SU(2))&\xrightarrow{{\rm Ad}}&\bigsqcup_{l=1}^{r-1}\mathcal{R}^{*}_{\frac{l}{2r}}(Y\setminus K,SO(3))\\
&\xrightarrow{\Pi'}&\mathcal{R}^{*, \tau}(\tilde{Y}_{r}, SO(3))\xrightarrow{{\rm Ad}^{-1}}\mathcal{R}^{*, \tau}(\tilde{Y}_{r}, SU(2)).
\end{eqnarray*}
We call $\Pi([\rho])$ a lift of $[\rho]$.
\end{dfn}
An $SU(2)$-representation which factors through $Pin(2)$ subgroups is called binary dihedral representation.
An $SO(3)$-representation which factors through $O(2)$ subgroups is called dihedral representation.
Note that $O(2)$ is embedded in $SO(3)$ as
\[
\left[
\begin{array}{cc}
A&0\\
0&{\rm det}A
\end{array}\right]\in SO(3),
\]
where $A\in O(2)$.
The adjoint representation of a binary dihedral representation is dihedral representation.
In the proof of Proposition \ref{lift}, which is an important property of the lift $\Pi$, we use the following Lemma.
\begin{lem}\label{fix}(\cite{SP})
The fixed point set of $H^{1}(Y\setminus K, \mathbb{Z}_{2})$-action on $\mathcal{R}(Y\setminus K, SU(2))$ consists of conjugacy classes of binary dihedral representations. 
\end{lem}
\begin{proof}
Let $[\rho]\in \mathcal{R}(Y\setminus K, SU(2))$ be a fixed point of the action of $H^{1}(Y\setminus K, \mathbb{Z}_{2})$. We regard this as a representation $\rho:\pi_{1}(Y\setminus K)\rightarrow SU(2)$ such that there exists $A\in SU(2)$ and $(\chi\cdot \rho )(u)=A\rho(u)A^{-1}$ for any $u\in\pi_{1}(Y\setminus K)$. Here $\chi\in H^{1}(Y\setminus K)$ is a generator. Since $\chi$ has order $2$, $\rho(u)=A^{2}\rho(u)A^{-2}$. If $\rho$ is reducible then its image is contained in a circle in $SU(2)$ and is binary dihedral representation. Assume that $\rho$ is irreducible, and there are two cases $A^{2}=1$ and $A^{2}=-1$.  We regard $SU(2)$ as the unit sphere in quaternions. 
Then $Pin(2)=S^{1}\cup jS^{1}$.
If $A^{2}=1$ then $A=\pm 1$ and $-\rho(u)=\rho(u)$ for some $u\in \pi_{1}(Y\setminus K)$. This cannot happen in $SU(2)$. If $A^{2}=-1$ then we can assume that $A=i$ taking after a conjugation and then $\rho(u)=\pm i\rho(u)i^{-1} $  for any $u\in \pi_{1}(Y\setminus K)$. If $\rho(u)=i\rho(u)i^{-1}$ then $\rho(u)\in S^{1}=\{a+bi\}$. If $\rho(u)=-i\rho(u)i^{-1}$ then $\rho(u)\in jS^{1}=\{cj+dk\}$. Thus the image of $\rho$ is contained in $S^{1}\cup jS^{1}$.
\end{proof}
\begin{lem}\label{dih}
Let $r\in 2\mathbb{Z}$. If $\rho:\pi^{V}_{1}(Y, K; r)\rightarrow SO(3)$ be a dihedral representation, then  its pull-back $\pi_{1}(\tilde{Y}_{r})\rightarrow SO(3)$ by the orbifold exact sequence (in Proposition \ref{orbseq}) is a reducible representation. 
\end{lem}
\begin{proof}
Since $\rho$ factors through $O(2)$, we have a representation $\rho':\pi^{V}_{1}(Y, K; r)\rightarrow O(2)$.  Composing with ${\rm det}: O(2)\rightarrow \mathbb{Z}/2$, we have a representation ${\rm det}\circ \rho':\pi_{1}^{V}(Y, K; r)\rightarrow \mathbb{Z}/2$. Since ${\rm det}\circ\rho'$ factors through abelianization $\pi_{1}^{V}(Y, K; r)=\pi_{1}(Y\setminus K)/\langle{\mu_{K}}^{r}\rangle \xrightarrow {{\rm Ab}}\mathbb{Z}[\mu_{K}]/\langle{\mu_{K}}^{r}\rangle $, we have the following diagram.
\[
\xymatrix{
\pi_{1}(\tilde{Y}_{r})\ar[r]^{i}&\pi^{V}_{1}(Y,K;r)\ar[r]^{\rho'}\ar[rd]^{\rm Ab}&O(2)\ar[r]^{{\rm det}}&\mathbb{Z}/2\\
&&\mathbb{Z}/r\ar[ru]&\\
}
\]
where $\pi_{1}(\tilde{Y}_{r})\xrightarrow{i} \pi^{V}_{1}(Y, K;r)$ is the inclusion map in the orbifold exact sequence. By the construction,  ${\rm Ab}$ coincides with the map $\pi^{V}_{1}(Y, K;r)\rightarrow \mathbb{Z}/r$ in the orbifold exact sequence. Thus ${\rm Ab}\circ i$ is the trivial representation and hence ${\rm det}\circ \rho'\circ i$ is also the trivial representation. This implies that the image of $\rho'\circ i$ is contained in $SO(2)$.
Thus $\rho\circ i:\pi_{1}(\tilde{Y}_{r})\rightarrow SO(3)$ factors through $SO(2)$, and this means that $\rho\circ i$ is reducible. 
\end{proof}

The following proposition gives the proof of Proposition \ref{lift}.
\begin{prp}
Let $K\subset Y$ be a knot in an integral homology $3$-sphere whose $r$-fold cyclic branched covering $\tilde{Y}_{r}$ is also an integral homology $3$-sphere.
For each $[\rho] \in \mathcal{R}^{*, \tau}(\tilde{Y}, SU(2))$,  $\Pi^{-1}([\rho])$ consists of two elements which correspond to  each other by the flip symmetry.
\end{prp}
\begin{proof}
Applying Proposition \ref{action} to the 3-manifold $Y\setminus K$, we have a bijection,
\begin{equation}\mathcal{R}(Y\setminus K, SU(2))/H^{1}(Y\setminus K, \mathbb{Z}_{2})\cong\mathcal{R}(Y\setminus K, SO(3)).\label{isom}
\end{equation}
Note that $\mathcal{R}(Y\setminus K, SO(3))=\mathcal{R}^{0}(Y\setminus K, SO(3))$ since $H^{2}(Y\setminus K, \mathbb{Z}_{2})=0$.  
We restrict this correspondence to elements with holonomy parameter $\alpha=\frac{l}{2r}\ \ ( l=1,\cdots, r-1)$.
Note that  $H^{1}(Y\setminus K, \mathbb{Z}_{2})$ acts on $\mathcal{R}^{*}_{\alpha}(Y\setminus K, SU(2))\cup \mathcal{R}^{*}_{\frac{1}{2}-\alpha}(Y\setminus K, SU(2))$  since the flip symmetry changes the  holonomy parameter as $\alpha \mapsto \frac{1}{2}-\alpha$, and the bijection (\ref{isom}) is restricted to 
\begin{equation}
\left[\bigsqcup_{l=1}^{r-1}\mathcal{R}^{*}_{\frac{l}{2r}}(Y\setminus K, SU(2))\right]/H^{1}(Y\setminus K, \mathbb{Z}_{2})
 \cong \bigsqcup_{l=1}^{r-1}\mathcal{R}^{*}_{\frac{l}{2r}}(Y\setminus K, SO(3)).
\end{equation}
Fix $[\rho] \in \mathcal{R}^{*,\tau}(\tilde{Y}_{r}, SU(2))$. The composition 
$$\bigsqcup_{0<l<r}\mathcal{R}^{*}_{\frac{l}{2r}}(Y\setminus K, SO(3))\xrightarrow{\Pi'} \mathcal{R}^{*,\tau}(\tilde{Y}_{r}, SO(3))\xrightarrow{{\rm Ad}^{-1}} \mathcal{R}^{*,\tau}(\tilde{Y}_{r}, SU(2))$$
is bijective. Thus we have a unique element $[\rho'] \in \bigsqcup_{0<l<r}\mathcal{R}^{*}_{\frac{l}{2r}}(Y\setminus K, SO(3))$ which corresponds to $[\rho]$. $\Pi^{-1}([\rho])$ is the inverse image of $[\rho']$ by the map$$ \bigsqcup_{l=1}^{r-1}\mathcal{R}^{*}_{\frac{l}{2r}}(Y\setminus K , SU(2))\xrightarrow{{\rm Ad}} \bigsqcup_{l=1}^{r-1}\mathcal{R}^{*}_{\frac{l}{2r}}(Y\setminus K, SO(3)).$$ Finally, we prove that ${\rm Ad}^{-1}([\rho'])$ consists of two elements. Let $[\sigma] \in {\rm Ad}^{-1}([\rho']) $ be an element contained in $\mathcal{R}^{*}_{\frac{l}{2r}}(Y\setminus K, SU(2))$.
If $r$ is odd then $\frac{l}{2r}\neq\frac{1}{2}-\frac{l}{2r}$, and thus $[\rho]\neq\chi [\rho]$ in $\bigsqcup_{0<l<r}\mathcal{R}^{*}_{\frac{l}{2r}}(Y\setminus K, SU(2))$ since $\rho$ and $\chi\rho$ have different holonomy parameters, where $\chi$ is a generator of $H^{1}(Y\setminus K,\mathbb{Z}_{2})\cong\mathbb{Z}_{2}$ . This means that $H^{1}(Y\setminus K,\mathbb{Z}_{2})$-action on  $\cup_{0<l<r}\mathcal{R}_{\frac{l}{2r}}(Y\setminus K, SU(2))$ is free and ${\rm Ad}^{-1}([\rho])$ consists of two elements. If $r$ is even and $2l\neq r$, then ${\rm Ad}^{-1}([\rho])$ consists of two elements by same reason. If $2l=r$ then $H^{1}(Y\setminus K,\mathbb{Z}_{2})$ acts on $\mathcal{R}^{*}_{\frac{1}{4}}(Y\setminus K, SU(2))$. The fixed points of $H^{1}(Y\setminus K,\mathbb{Z}_{2})$-action on $\mathcal{R}^{*}_{\frac{1}{4}}(Y\setminus K, SU(2))$ are binary dihedral representations by Lemma \ref{fix}. We show that $\Pi^{-1}([\rho])$ does not contain a binary dihedral representation. Let $\sigma' : \pi_{1}(Y\setminus K) \rightarrow SU(2)$ be a binary dihedral representation with a holonomy parameter $\alpha=\frac{1}{4}$ then ${\rm Ad}\sigma'$ defines a dihedral representation$\pi_{1}^{V}(Y, K;r)\rightarrow SO(3)$.  Then the induced representation  ${\rm Ad}\sigma'\circ i :\pi_{1}(\tilde{Y}_{r})\rightarrow SO(3)$ is reducible by Lemma \ref{dih} and its $SU(2)$-lift is also reducible. This means that $\Pi^{-1}([\rho])$ does not contain any binary dihedral representation. Thus $H^{1}(Y\setminus K, \mathbb{Z}_{2})$ acts free on $\Pi^{-1}([\rho])$ and hence $\Pi^{-1}([\rho])$ consists of two elements which are related by the flip symmetry.  
\end{proof}

\subsection{ Non-degeneracy results}\label{non-deg res}
The purpose of this subsection is to associate non-degeneracy property of critical point set $\mathfrak{C}$ of singular Chern-Simons functional and transversality of moduli space of irreducible flat connections $\mathcal{R}^{*}(Y\setminus K, SU(2))$.  
Let us recall the setting of the gauge theory used in \cite{her94} and \cite{her97} to deal with "pillowcase picture" of perturbed flat connections.
In this subsection, $Y$ denotes a (general) oriented closed 3-manifold and $K$ be a knot in $Y$.
Let $E$ be a $SU(2)$-bundle over $X=Y\setminus N(K)$. We fix a Riemannian metric on $X$. We introduce the space of $SU(2)$-connections over $X$ and $\partial X=T^{2}$ as follows:
\[\mathcal{A}_{X}=L^{2}_{2}(X, \mathfrak{su}(2)\otimes \Lambda^{1}), \  \mathcal{A}_{T^{2}}=L^{2}_{\frac{3}{2}}(T^{2}, \mathfrak{su}(2)\otimes \Lambda^{1}).\]
Here we fix a trivialization of the $SU(2)$-bundle over $X$ and $\partial X$ and identify the trivial connection to zero element in each functional spaces. We also introduce spaces of gauge transformations:
\[\mathcal{G}_{X}=\{g\in {\rm Aut}(E)| g\in L^{2}_{3}\},\ \mathcal{G}_{T^{2}}=\{g \in {\rm Aut}(E|_{T^{2}})| g \in L^{2}_{\frac{5}{2}} \}.\] 
The action of gauge transformations on connections and $\mathfrak{su}(2)$ valued $p$-forms are given by the obvious way.
$\mathcal{G}_{X}$ and $\mathcal{G}_{T^{2}}$ have Banach Lie group structures and act smoothly on $\mathcal{A}_{X}$ and $\mathcal{A}_{T^{2}}$ respectively. A connection whose stabilizer of gauge transformations is $\{\pm 1\}$ is called irreducible. $\mathcal{A}^{*}_{X}$ denotes subset of irreducible connections.

We introduce following spaces of $p$-forms with  boundary conditions:
$$\Omega^{p}_{\nu}(X,\mathfrak{su}(2))=\{\omega\in \Omega^{p}(X,\mathfrak{su}(2))|*\omega|_{\partial X}=0\},$$
$$\Omega^{p}_{\tau}(X, \mathfrak{su}(2))=\{\omega\in \Omega^{p}(X, \mathfrak{su}(2))|\ \omega|_{\partial X}=0\}.$$

We define the $L^{2}$-inner product on $\Omega^{p}(X, \mathfrak{su}(2))$ by the formula
$$\langle a, b\rangle=-\int_{X}{\rm tr}(a\wedge *b).$$For each $A\in \mathcal{A}_{X}$, the slice of the action of $\mathcal{G}_{X}$ on $\mathcal{A}_{X}$ is given by
$$X_{A}=A+{\rm Ker}d^{*}_{A}\cap L^{2}_{2}\Omega^{1}_{\nu}(X,\mathfrak{su}(2)).$$
For each flat connection $A\in \mathcal{A}$, the space of harmonic $p$-forms are given by
\begin{eqnarray*}
\mathcal{H}^{p}(X;{\rm ad}{A})&=&\{\omega\in \Omega^{p}_{\nu}(X, \mathfrak{su}(2))|d_{A}\omega=0, d^{*}_{A}\omega=0\},\\
\mathcal{H}^{p}(X, \partial X; {\rm ad}{A})&=&\{\omega\in \Omega^{p}_{\tau}(X,\mathfrak{su}(2))|d_{A}\omega=0, d^{*}_{A}\omega=0\}.
\end{eqnarray*}

The holonomy perturbation $h$ defines a compact perturbation term $V_{h}:\mathcal{A}_{X}\rightarrow \Omega^{1}(X, \mathfrak{su}(2))$ and perturbed flat connection can be defined as a solution of the equation:
\begin{equation}
*F_{A}+V_{h}=0.\label{perflat}
\end{equation}
$\mathcal{R}^{*, h}(X,SU(2))$ denotes gauge equivalence classes of irreducible solutions for (\ref{perflat}). 
Consider the restriction map $r:\mathcal{R}^{*, h}(X,SU(2))\rightarrow \mathcal{R}(T^{2}, SU(2))$. For a generic perturbation $h$, $\mathcal{R}^{*, h}(X,SU(2))$ is a smooth one-manifold. 
Moreover, the restriction map $r$ is a smooth immersion of $\mathcal{R}^{*, h}(X,SU(2))$ to the smooth part of the pillowcase. The detailed argument is contained in \cite{her94}. 
Put $S_{\alpha}:=\{\rho \in \mathcal{R}(T^{2}, SU(2))|{\rm tr}\rho(\mu_{K})=2{\rm cos}(2\pi i \alpha)\}$. This is a vertical slice in the pillowcase. 
Note that $\mathcal{R}_{\alpha}(X, SU(2))=r^{-1}(S_{\alpha})\cap \mathcal{R}(X, SU(2))$ and we define $\mathcal{R}_{\alpha}^{*, h}(X, SU(2)):=r^{-1}(S_{\alpha})\cap \mathcal{R}^{*, h}(X, SU(2))$.
Now we state the following proposition.
\begin{prp}\label{immers'}
Let $K\subset Y$ be a knot in a closed $3$-manifold, and $\alpha$ be an arbitrary holonomy parameter in $(0,\frac{1}{2})$.
Assume that $[\rho] \in \mathcal{R}^{*}_{\alpha}(Y\setminus K, SU(2))$ is a non-degenerate critical point.
We also assume that the image of $\mathcal{R}^{*}(Y\setminus K, SU(2))$ by the restriction map $r$ is contained in the smooth part of the pillowcase.
Then $\mathcal{R}^{*}(Y\setminus K, SU(2))$ is smooth near $[\rho]$.
Moreover, the restriction map $r:\mathcal{R}^{*}(Y\setminus K ,SU(2))\rightarrow \mathcal{R}(T^{2}, SU(2))$ is an immersion to the smooth part of the pillowcase at $[\rho]$. 
\end{prp}

For the proof of Proposition \ref{immers'}, we need setting of the gauge theory on 3-manifolds with boundary described as above.
\begin{lem}\label{diag}
Let $B_{0}$ be an abelian $SU(2)$-flat connection on a torus $T^{2}$.
Then $\mathfrak{su}(2)$-valued harmonic form $h\in \mathcal{H}^{1}(T^{2}; {\rm ad}{B_{0}})$ has a diagonal form
\[h=\left[
\begin{array}{cc}
ai&0\\
0&-ai
\end{array}
\right]\]where $a\in \Omega^{1}(T^{2})$.
\end{lem}
\begin{proof}
Since $B_{0}$ is an abelian flat connection,  it defines  a splitting of the $SU(2)$-bundle $E$ over $T^{2}$
into $E=L\oplus L^{*}$ where $L$ is a trivial line bundle.
Then the adjoint bundle of $E$ has a splitting
$\mathfrak{g}_{E}=\underline{\mathbb{R}}\oplus L^{\otimes 2}$. Let $d_{B_{0}}=d\oplus d_{C}$ be the covariant derivative induced on $\Omega^{p}(T^{2}, \mathfrak{g}_{E})=\Omega^{p}(T^{2}, \underline{\mathbb{R}})\oplus \Omega^{p}(T^{2}, L^{\otimes 2})$. Any section $\omega \in \Omega^{p}(T^{2}, {\mathfrak{g}_{E}})$ has the form
\[
\omega=\left[\begin{array}{cc}
ai& b\\
-\bar{b}&-ai
\end{array}
\right],
\]
where $a\in \Omega^{p}(T^{2})$ and $b\in \Omega^{p}(T^{2}, L^{\otimes 2})$. The space of harmonic forms $\mathcal{H}^{p}(T^{2}; {\rm ad}{B_{0}})$ splits into $\mathcal{H}^{p}(T^{2}; \underline{\mathbb{R}})\oplus \mathcal{H}^{p}(T^{2}; L^{\otimes 2})$ with respect to the decomposition of $(\Omega^{p}(T^{2},\mathfrak{g}_{E}), d_{B_{0}})$. Let us compute $\mathcal{H}^{1}(T^{2}; {\rm ad}{B_{0}})$ using $H^{1}(\pi_{1}(T^{2}); {\rm ad}{\rho})$ where $\rho$ is an abelian $SU(2)$-representation corresponding to $B_{0}$.
Let $\mu$ and $\lambda$ be canonical generators of $\pi_{1}(T^{2})$. Then the space of 1-cocycles consists of the element $\gamma:\pi_{1}(T^{2})\rightarrow \mathfrak{su}(2)\cong \mathbb{R}^{3}$ such that $$(1-{\rm Ad}_{\rho(\mu)})\gamma(\lambda)=(1-{\rm Ad}_{\rho(\lambda)})\gamma(\mu).$$ since $\mu$ and $\lambda$ commute. Let $F:\mathbb{R}^{3}\oplus \mathbb{R}^{3}\rightarrow \mathbb{R}^{2}$ be a linear map given by
$$F(x_{1}, x_{2})=(1-A_{\mu})x_{1}-(1-A_{\lambda})x_{2}$$
where $A_{\mu}:={\rm Ad}_{\rho(\mu)}$ and $A_{\lambda}:={\rm Ad}_{\rho(\lambda)}$. Since $A_{\mu}$ and $A_{\lambda}$ are $SO(3)$-linear transformation acting on $\mathbb{R}^{3}$, they have one dimensional axes of rotation $\mathbb{R}_{\mu}$ and $\mathbb{R}_{\lambda}$ respectively. 
Let $\mathbb{C}_{\mu}$ and $\mathbb{C}_{\lambda}$ be their orthogonal complement spaces. 
Then ${\rm Im}(1-A_{\mu})=\mathbb{C}_{\mu}$ and ${\rm Im}(1-A_{\lambda})=\mathbb{C}_{\lambda}$ and hence $F$ is surjective. Thus the space of 1-cocycles is isomorphic to $\mathbb{R}^{4}$. On the other hand, the space of 1-coboundaries is spanned by ${\rm Im}(1-{\rm Ad}_{\rho(g)})$   for all $g\in \pi_{1}(T^{2})$, and this is 2-dimensional since $\rho$  is reducible. Thus $\mathcal{H}^{1}(T^{2}; {\rm ad}{B_{0}})\cong H^{1}(\pi_{1}(T^{2}); {\rm ad}{\rho})\cong \mathbb{R}^{2}$. 
  Thus $ H^{1}(T^{2}; L^{\otimes 2})$ vanishes since $\mathcal{H}^{1}(T^{2}; \underline{\mathbb{R}})\cong H^{1}(T^{2}; \mathbb{R})\cong\mathbb{R}^{2}$.
This means that if $\omega \in \Omega^{1}(T^{2}, \mathfrak{g}_{E})$ is a harmonic form then $b=0$. Thus $h\in \mathcal{H}^{1}(T^{2}; {\rm ad}{B_{0}})$ has only diagonal components.
\end{proof}
Since $B_{0}$ is a reducible connection with $U(1)$-stabilizer, $\mathcal{H}^{0}(T^{2}; {\rm ad}{B_{0}})={\rm Ker}d_{B_{0}}\cong \mathbb{R}$. We fix a generator $\gamma_{0}\in \mathcal{H}^{0}(T^{2}; {\rm ad}{B_{0}})$. 
\begin{lem}\label{lemma1}
There is a $\mathcal{G}_{T^{2}}$-invariant neighborhood ${N}_{B_{0}}$ of $B_{0}\in \mathcal{A}_{T^{2}}$ and $\mathcal{G}_{T^{2}}$-invariant map $\eta:{N}_{B_{0}}\rightarrow \Omega^{0}(T^{2}, \mathfrak{su}(2))$ such that
\begin{enumerate}
\item $\eta(B_{0})=\gamma_{0}$,
\item $\int_{T^{2}}{\rm tr}(F_{B}\wedge \eta(B))=0$ for all $B\in {N}_{B_{0}}$.
\end{enumerate}
\end{lem}
\begin{proof}
Take a small neighborhood of $B_{0}$ in the slice of the action of $\mathcal{G}_{T^{2}}$ on $\mathcal{A}_{T^{2}}$ as  follows.
$$X_{B_{0}, \epsilon}=\{B_{0}+b|\ b\in L^{2}_{\frac{3}{2}}\Omega^{1}(T^{2},\mathfrak{su}(2)), d^{*}_{B_{0}}b=0, \|b\|_{L^{2}_{\frac{3}{2}}}<\epsilon\}$$
where $\epsilon>0$ is small enough. Firstly, we define $\Omega^{0}(T^{2}, \mathfrak{su}(2))$-valued map $\eta$ on the slice $X_{B_{0}, \epsilon}$ and then extend it to a gauge-invariant neighborhood. For $B=B_{0}+b\in X_{B_{0}, \epsilon}$, define $$\eta(B):=\gamma_{0}.$$ Then we have,
\begin{eqnarray*}
\int_{T^{2}}{\rm tr}(F_{B}\wedge \gamma_{0})&=&\int_{T^{2}}{\rm tr}((d_{B_{0}}b+b\wedge b)\wedge \gamma_{0})\\
&=&\int_{T^{2}}{\rm tr}(d_{B_{0}}b\wedge \gamma_{0})+\int_{T^{2}}{\rm tr }(b\wedge b\wedge \gamma_{0}).
\end{eqnarray*}Using the Stokes' theorem and the condition $d_{B_{0}}\gamma_{0}=0$,
$$\int_{T^{2}}{\rm tr}(d_{B_{0}}b\wedge \gamma_{0})=\int_{T^{2}}d{\rm tr}(b\wedge \gamma_{0})=0.$$Thus we have 
\begin{equation}
\int_{T^{2}}{\rm tr}(F_{B}\wedge \eta(B))=\int_{T^{2}}{\rm tr }(b\wedge b\wedge \gamma_{0}).\label{a}
\end{equation}

Since $d_{B_{0}}^{*}b=0$, we have $h\in \mathcal{H}^{1}(T^{2};{\rm ad}{B_{0}})$ and $\omega \in \Omega^{2}(T^{2}, \mathfrak{su}(2))$ such that
\begin{equation}
b=h+d_{B_{0}}^{*}\omega.\label{b}
\end{equation}

Using (\ref{a}) and (\ref{b}),
\begin{eqnarray*}
\int_{T^{2}}{\rm tr}(F_{B}\wedge \eta(B))&=&\int_{T^{2}}{\rm tr}\left[(h+d_{B_{0}}^{*}\omega)\wedge(h+d_{B_{0}}^{*}\omega)\wedge\gamma_{0}\right]\\
&=& \int_{T^{2}}{\rm tr}(h\wedge h\wedge \gamma_{0})+\int_{T^{2}}{\rm tr}(d_{B_{0}}^{*}\omega\wedge h\wedge \gamma_{0})\\
&&+\int_{T^{2}}{\rm tr}(h\wedge d_{B_{0}}^{*}\omega\wedge \gamma_{0})+\int_{T^{2}}{\rm tr}(d_{B_{0}}^{*}\omega\wedge d_{B_{0}}^{*}\omega\wedge \gamma_{0}).
\end{eqnarray*}
Note that,
\begin{eqnarray*}
\int_{T^{2}}{\rm tr}(*d_{B_{0}}*\omega \wedge h\wedge \gamma_{0})&=&\int_{T^{2}}{\rm tr}(d_{B_{0}}*\omega\wedge *h\wedge \gamma_{0})\\
&=&-\int_{T^{2}}{\rm tr}(*\omega\wedge d_{B_{0}}*h\wedge \gamma_{0})+\int_{T^{2}}{\rm tr}(*\omega\wedge *h\wedge d_{B_{0}}\gamma_{0})\\
&=&0.
\end{eqnarray*}Here we use the Stokes' theorem at the second equality.
Similarly we have,
\begin{eqnarray*}
\int_{T^{2}}{\rm tr}(h\wedge d_{B_{0}}^{*}\omega\wedge \gamma_{0})&=&-\int_{T^{2}}{\rm tr}(d_{B_{0}}*h\wedge *\omega \wedge \gamma_{0})-\int_{T^{2}}{\rm tr}(*h\wedge *\omega \wedge d_{B_{0}}\gamma_{0})=0,
\end{eqnarray*}
\begin{eqnarray*}
\int_{T^{2}}{\rm tr}(d_{B_{0}}^{*}\omega\wedge d_{B_{0}}^{*}\omega\wedge \gamma_{0})&=&-\int_{T^{2}}{\rm tr}(d_{B_{0}}^{2}*\omega\wedge *\omega \wedge \gamma_{0})-\int_{T^{2}}{\rm tr}(d_{B_{0}}*\omega\wedge *\omega \wedge d_{B_{0}}\gamma_{0})=0.
\end{eqnarray*}
Hence, $$\int_{T^{2}}{\rm tr}(F_{B}\wedge\eta(B) )=\int_{T^{2}}{\rm tr}(h\wedge h\wedge\gamma_{0} ).$$Since $\gamma_{0}\in \mathcal{H}^{0}(T^{2}; {\rm ad}{B_{0}})$ is an element of Lie algebra of the stabilizer of $B_{0}$,  ${\rm Stab}(B_{0})=U(1)$ and it has the point-wise form
\[
\gamma_{0}(x)=\left[\begin{array}{cc}
ri&0\\
0&-ri
\end{array}\right]\in \mathfrak{su}(2)
\]where $r\in \mathbb{R}$.
Similarly, $h\in \mathcal{H}^{1}(T^{2}; {\rm ad}{B_{0}})$ has the the form
\[
h(x)=\left[\begin{array}{cc}
ai&0\\
0&-ai\end{array}\right]
\]
by Lemma \ref{diag}.
 By the point-wise computation of ${\rm tr }(h\wedge h\wedge \gamma_{0})$, we obtain
\begin{eqnarray*}
{\rm tr}(h\wedge h\wedge \gamma_{0})(x)&=&{\rm tr}\left(\left[\begin{array}{cc}
ai&0\\
0&-ai\
\end{array}\right]\wedge\left[\begin{array}{cc}
ai&0\\
0&-ai\
\end{array}\right]\wedge\left[\begin{array}{cc}
ri&0\\
0&-ri
\end{array}\right]\right)\\
&=&{\rm tr}\left[ \begin{array}{cc}
-ra\wedge ai&0\\
0&ra\wedge ai
\end{array}
\right]\\
&=&0.
\end{eqnarray*}
Thus $\eta:X_{B_{0}, \epsilon}\rightarrow \Omega^{0}(T^{2}, \mathfrak{su}(2))$ satisfies $\int_{T^{2}}{\rm tr}(F_{B}\wedge \eta(B))=0$ for all $B\in X_{B_{0}, \epsilon}$.
Define $N_{B_{0}}:=\mathcal{G}_{T^{2}}\cdot X_{B_{0},\epsilon}$ and we extend $\eta$ to $N_{B_{0}}$ in a gauge-equivariant way (i.e. $\eta(g^{*}(B))=g^{-1}\eta(B)g$).
\end{proof}

Let  $A_{0}$ be a flat irreducible $SU(2)$-connection  on $X$. 
We can assume that $A_{0}|_{T^{2}}$ is non-central flat connection on $T^{2}$ by the assumption of Proposition \ref{immers'}, and we write $A_{0}|_{T^{2}}=B_{0}$.
Let $U_{A_{0}}$ be a gauge invariant neighborhood of $A_{0}$ in $\mathcal{A}^{*}_{X}$. We define $\tilde{\eta}:U_{A_{0}}\rightarrow \Omega^{0}(X, \mathfrak{su}(2))$ as a smooth extension of $\eta$ which satisfies
$$\tilde{\eta}(A)|_{\partial X}=\eta(A|_{\partial X}).$$
Here we assume that the extension $\tilde{\eta}$ satisfies $d_{A_{0}}\tilde{\eta}(A_{0})\in \mathcal{H}^{1}(X, \partial X; {\rm ad}{A_{0}}) $ and this is possible by the following lemma.
\begin{lem}\label{connecting}
For $\eta\in \mathcal{H}^{0}(T^{2}; {\rm ad}{B_{0}})$
there is an extension $\tilde{\eta}$ on $X$ such that $d_{A_{0}}\tilde{\eta}\in \mathcal{H}^{1}(X, \partial X;{\rm ad}{A_{0}})$.
\end{lem}
\begin{proof}
For  $\eta\in \mathcal{H}^{0}(T^{2}, {\rm ad}{B_{0}})$, we take arbitrary smooth extension $\tilde{\eta}$ to $X$. Then 
$$d_{A_{0}}\tilde{\eta}\in {\rm Ker}d_{A_{0}}|_{\Omega^{1}_{\tau}(X, \mathfrak{su}(2))}=d_{A_{0}}\Omega^{0}_{\tau}(X, \mathfrak{su}(2))\oplus \mathcal{H}^{1}(X, \partial X; {\rm ad}{A_{0}}).$$
Let $d_{A_{0}}\tilde{\xi}$ be the $d_{A_{0}}\Omega^{0}_{\tau}(X, \mathfrak{su}(2))$-component of $d_{A_{0}}\tilde{\eta}$, then 
$d_{A_{0}}(\tilde{\eta}-\tilde{\xi})\in \mathcal{H}^{1}(X, \partial X; {\rm ad}{A_{0}})$ with $(\tilde{\eta}-\tilde{\xi})|_{\partial X}=\eta$. Hence we can choose an extension $\tilde{\eta}$ of $\eta$ as $d_{A_{0}}\tilde{\eta}\in \mathcal{H}^{1}(X, \partial X; {\rm ad}{A_{0}})$. 
\end{proof}
We define a map
$$\Phi:U_{A_{0}}\times L^{2}_{2}\Omega^{0}_{\tau}(X, \mathfrak{su}(2))\times \mathbb{R}\rightarrow L^{2}_{1}(X, \mathfrak{su}(2)\otimes \Lambda^{1})$$by $\Phi(A, \zeta, t)=*F_{A}+d_{A}\zeta+td_{A}\tilde{\eta}(A)$.  
The linearized operator of $\Phi$ at $(A_{0}, 0, 0)$ has a form  $$D\Phi_{(A_{0}, 0, 0)}(a, \zeta, t)=*d_{A_{0}}a+d_{A_{0}}\zeta+td_{A_{0}}\tilde{\eta}(A_{0}).$$
${\rm Coker}D\Phi_{(A_{0}, 0,0)}$ is $\mathcal{H}^{1}(X, \partial X;{\rm ad}_{A_{0}})\cap (d_{A_{0}}\tilde{\eta}(A_{0}))^{\perp}$ by the Hodge decomposition.

\begin{lem}\label{flat eqn}
$\Phi(A,\zeta, t)=0$ if  only if $F_{A}=0$, $\zeta=0$ and $t=0$. 
\end{lem}
\begin{proof}
Assume that $\Phi(A, \zeta, t)=0$. Then
\begin{eqnarray*}\|F_{A}\|^{2}_{L^{2}}&=&-\int_{X}{\rm tr}(F_{A}\wedge*F_{A})\\
&=&\int_{X}{\rm tr }(F_{A}\wedge d_{A}\zeta)+t\int_{X}{\rm tr}(F_{A}\wedge d_{A}\tilde{\eta}(A)).
\end{eqnarray*}
Using the Stokes' theorem and the Bianchi identity, 
\begin{eqnarray*}\int_{X}{\rm tr}(F_{A}\wedge d_{A}\zeta)&=&\int_{X}d{\rm tr}(F_{A}\wedge \zeta)-\int_{X}{\rm tr}(d_{A}F_{A}\wedge \zeta)\\
&=&\int_{T^{2}}{\rm tr}(F_{A|_{T^{2}}}\wedge \zeta|_{T^{2}}).
\end{eqnarray*}
The last term vanishes by the boundary condition on $\zeta$.
Consider the remained term
\begin{equation}
\int_{X}{\rm tr}(F_{A}\wedge d_{A}\tilde{\eta}(A)).\label{inner prod}
\end{equation}
Using the Stokes' theorem and the Bianchi identity, this is equal to
$$\int_{T^{2}}{\rm tr}(F_{B}\wedge \eta(B))$$where $B=A|_{T^{2}}$. By Lemma \ref{lemma1}, this is equal to zero and we have $F_{A}=0$.
Thus $0=*F_{A}=-d_{A}(\zeta+t\tilde{\eta}(A))$ by our assumption. Since $A$ is a irreducible connection, $d_{A}$ has the trivial kernel and $\zeta=-t\tilde{\eta}(A)$. Restricting this to the boundary $T^{2}$, we have a relation $t{\eta}(B)=0$.
Since $\eta(B)=\gamma_{0}$ is a generator of $\mathcal{H}^{0}(T^{2};{\rm ad}{B_{0}})$, we have $\eta(B)\neq 0$. Hence $t=0$ and $\zeta=0$ follows.

Conversely, if we assume that $F_{A}=0$, $\zeta=0$, and $t=0$, then clearly $\Phi(A, \zeta, t)=0$.
\end{proof}
Lemma \ref{flat eqn} means that two equations $F_{A}=0$ and $\Phi(A, \zeta, t)=0$ have the same zero set near a irreducible flat connection $A_{0}$. Hence $\Phi=0$ defines space of flat connections near  $A_{0}$.
\begin{proof}[Proof of Proposition \ref{immers'}]
Natural embeddings 
$\mu_{K}\hookrightarrow \partial X \hookrightarrow X$  induce maps on cohomology groups with a local coefficient system,
\begin{equation}
H^{1}(X; {\rm ad}{\rho})\xrightarrow{j} H^{1}(\partial{X}; {\rm ad}{\rho})\rightarrow H^{1}(\mu_{K}; {\rm ad}{\rho}).\label{comp}
\end{equation}
The non-degeneracy condition on $[\rho]$  is equivalent to the condition that the composition  (\ref{comp}) is injective by Proposition \ref{grcoh non-deg}. This implies that $j$ is also injective. 
 Thus the restriction map $\mathcal{R}^{*}(X ,SU(2))\rightarrow \mathcal{R}(T^{2}, SU(2))$ to the pillowcase is an immersion at $[\rho]$ if we show that $\mathcal{R}^{*}(X, SU(2))$ has a smooth manifold structure near $[\rho]$. Next, we show that $\mathcal{R}^{*}(X, SU(2))$ is a smooth manifold near $[\rho]$. Consider the long exact sequence of cohomology with local coefficient associated to the pair $(X, \partial X)$,
$$\cdots \rightarrow H^{0}(\partial X;{\rm ad}{\rho})\xrightarrow{\partial} H^{1}(X, \partial X;{\rm ad}{\rho})\rightarrow H^{1}(X; {\rm ad}{\rho})\xrightarrow{j} H^{1}(\partial{X}; {\rm ad}{\rho})\rightarrow \cdots$$The cokernel of connecting homomorphism $\partial$ is zero since $j$ is injective.  Using harmonic representative of cohomology with local coefficient, the connecting homomorphism $\partial$ is given by $\gamma_{0}\mapsto d_{A_{0}}\tilde{\eta}(A_{0})$  where $A_{0}$ is a $SU(2)$-flat connection corresponding to $\rho$.  Thus ${\rm coker}\partial=\mathcal{H}^{1}(X, \partial X; {\rm ad}_{A_{0}})\cap (d_{A_{0}}\tilde{\eta}(A_{0}))^{\perp}=0$. This means that the equation $\Phi(A, \zeta, t)=0$ has a surjective linearization map at $(A_{0}, 0, 0)$. 
Thus there is a neighborhood $V_{A_{0}}$ of $A_{0}\in \Phi^{-1}(0)$ which has a smooth structure by the implicit function theorem. Since $A_{0}$ is irreducible, the quotient singularity by gauge transformations $\mathcal{G}_{X}$ does not occur. Thus $\mathcal{R}^{*}(X, SU(2))$ has a smooth manifold structure near $[\rho]$. 
\end{proof}
By Proposition \ref{grcoh non-deg}, the following proposition shows that singular Chern-Simons functional for $(p, q)$-torus knot has non-degenerate irreducible critical points without perturbation. 
\begin{prp}\label{non-deg'}
For any $[\rho] \in \mathcal{R}^{*}(S^{3}\setminus T_{p, q}, SU(2))$, the natural map 
\begin{equation}H^{1}(S^{3}\setminus T_{p, q};{\rm ad}{\rho})\rightarrow H^{1}(\mu_{T_{p, q}}; {\rm ad}{\rho})\label{coh}
\end{equation}
is injective.
\end{prp}
\begin{proof}
Firstly, we compute $H^{1}(S^{3}\setminus T_{p, q}; {\rm ad}{\rho})$ using the group cohomology of $\pi_{1}(Y\setminus K)$. Since the fundamental group $\pi_{1}(S^{3}\setminus T_{p, q})$ has a presentation $\langle x, y|x^{p}=y^{q}\rangle$, 1-cocycles $\gamma:\pi_{1}(S^{3}\setminus T_{p, q})\rightarrow \mathfrak{su}(2)$ satisfy the relation
$$(I+{A_x}+\cdots +A^{p-1}_x)\gamma(x)=(I+A_{y}+\cdots +A_{y}^{q-1})\gamma(y)$$
where $A_{x}:={\rm Ad}_{\rho(x)}$ and $A_{y}:={\rm Ad}_{\rho(y)}$. Since $A_{x}$ and $A_{y}$ are $SO(3)$ linear transformation acting on $\mathbb{R}^{3}$, they have one dimensional axis of rotation $\mathbb{R}_{x}$ and $\mathbb{R}_{y}$ respectively. Let $\mathbb{C}_{x}$ and $\mathbb{C}_{y}$ denote their orthogonal complement spaces. Then ${\rm Im}(I-A_{x})=\mathbb{C}_{x}$ and ${\rm Im}(I-A_{y})=\mathbb{C}_{y}$. Note that $\rho(x)$ and $\rho(y)$ are contained in different great circles in $SU(2)\cong S^{3}$ since $\rho$ is irreducible $SU(2)$-representation. Thus $\rho$ satisfies $\rho(x)^{p}=\rho(y)^{q}=\pm 1$ and hence $A_x^{p}=A_{y}^{q}=I$. Thus we have ${\rm Ker}(I+A_x+\cdots +A^{p-1}_x)=\mathbb{C}_{x}$ and ${\rm Ker}(I+A_y +\cdots +A_{y}^{q-1})=\mathbb{C}_{y}$. 
Since $\rho$ is irreducible, $\mathbb{R}_{x}$ and $\mathbb{R}_{y}$ are independent in $\mathbb{R}^{3}$. Consider the linear map $L: \mathbb{R}^{3}\oplus \mathbb{R}^{3}\rightarrow \mathbb{R}^{3}$ defined by
$$L(x_{1}, x_{2})=(I+A_{x}+\cdots +A_{x}^{p-1})x_{1}-(I+A_{y}+\cdots +A_{y}^{q-1})x_{2}.$$
This has rank $2$, and the space of $1$-cocycles has dimension 4. On the other hand, the space of $1$-coboundaries is a subspace of $\mathbb{R}^{3}$ spanned by ${\rm Im}(I-{\rm Ad}_{\rho(g)})$ for all $g\in \pi_{1}(S^{3}\setminus T_{p, q})$, and this coincides to $\mathbb{R}^{3}$ itself. Thus, we have $H^{1}(S^{3}\setminus T_{p, q}; {\rm ad}{\rho})\cong \mathbb{R}^{4}/\mathbb{R}^{3}\cong \mathbb{R}$.

Next, we compute $H^{1}(\mu_{T_{p, q}}; {\rm ad}{\rho})$. In this case, the space of $1$-cocycles is isomorphic to $\mathbb{R}^{3}$ since its elements are determined by choosing $\gamma(\mu)\in \mathfrak{su}(2)\cong \mathbb{R}^{3}$. The space of $1$-coboundaries is ${\rm Im}(I-{\rm Ad}_{\rho(\mu)})\cong \mathbb{C}$. Thus, we have $H^{1}(\mu; {\rm ad}{\rho})\cong \mathbb{R}^{3}/\mathbb{C}\cong \mathbb{R}$.

Finally, we prove that the map (\ref{coh}) is surjective.  If $\gamma:\pi_{1}(S^{3}\setminus T_{p, q})\rightarrow \mathfrak{su}(2)$ represents nonzero element in $H^{1}(S^{3}\setminus T_{p, q}; {\rm ad}{\rho})$ then $\gamma(g)\notin {\rm Im}(I-{\rm Ad_{\rho(g)}})$ for any $g\in \pi_{1}(S^{3}\setminus T_{p, q})$. Thus $\gamma(\mu)\notin {\rm Im}(I-{\rm Ad}_{\rho(\mu)})$ for meridian $\mu \in \pi_{1}(S^{3}\setminus T_{p, q})$, and this means that the image of $[\gamma]\in H^{1}(S^{3}\setminus T_{p, q}; {\rm ad}{\rho})$ in $H^{1}(\mu_{T_{p, q}}; {\rm ad}{\rho})$ is a nonzero element. This completes the proof.
\end{proof}
Consider a knot $K$ in $S^{3}$. 
Note that the image of restriction map $r:\mathcal{R}^{*}(S^{3}\setminus K,SU(2))\rightarrow \mathcal{R}(T^{2},SU(2))$ is contained in the smooth part of the pillowcase.
By Proposition \ref{non-deg'} and \ref{immers'}, we get the following statement. 
\begin{cor}\label{torus immers}
The natural restriction map $\mathcal{R}^{*}(S^{3}\setminus T_{p,q}, SU(2))\rightarrow \mathcal{R}(T^{2}, SU(2))$ to the smooth part of the pillowcase is an immersion of smooth one-manifold. 
\end{cor}
In fact, it is known that the irreducible representation variety $\mathcal{R}^{*}(S^{3}\setminus T_{p, q}, SU(2))$ is a disjoint union of $\frac{1}{2}(p-1)(q-1)$ segment. 
(See \cite{klassen1991representations}.)
\subsection{Tristram-Levine signature and representation variety\label{TLsig}}
The following statement relates size of the set of singular flat connections over $S^3\setminus T_{p, q}$ and the set of flat connections over the cyclic branched covering.
\begin{lem}\label{P2}
Let $p, q, r$ be relatively coprime  positive  integers and $\Sigma(p, q ,r)$ be a Brieskorn  homology sphere.  Then
$$2|\mathcal{R}^{*}(\Sigma(p, q, r), SU(2))|=\sum_{ l=1}^{r-1}|\mathcal{R}^{*}_{\frac{l}{2r}}(S^{3}\setminus T_{p, q}, SU(2))|.$$
\end{lem}
\begin{proof}
We apply Proposition \ref{lift} to the $r$-fold cyclic branched covering $\Sigma(p, q, r)$ of $T_{p, q}\subset S^{3}$.
Note that the covering transformation $\tau$ induces trivial action on $\mathcal{R}^{*}(\Sigma(p, q, r), SU(2))$ by \cite{CSa}  and hence there is a two to one correspondence 
\[\bigsqcup_{l}\mathcal{R}^{*}_{\frac{l}{2r}}(S^{3}\setminus T_{p, q},SU(2))\xrightarrow{\cong} \mathcal{R}^{*}(\Sigma(p, q, r), SU(2)).\]
 \end{proof}
The non-degeneracy condition at irreducible critical points can be interpreted in the pillowcase as follows.
 \begin{lem}\label{trans}
 Let $\alpha\in(0,\frac{1}{2})$ be a fixed holonomy parameter.
 Assume that  $[\rho]\in\mathcal{R}^{*}_{\alpha}(S^{3}\setminus K,SU(2))$ is non-degenerate. Then $S_{\alpha}$ and the image of $\mathcal{R}^{*}(S^{3}\setminus K,SU(2))$ by the restriction map $r:\mathcal{R}^{*}(S^{3}\setminus K,SU(2))\rightarrow \mathcal{R}(T^{2}, SU(2))$ intersect transversely at $r([\rho])$.
 \end{lem}
 \begin{proof}
 Consider the natural map $p:\mathcal{R}(T^{2}, SU(2))\rightarrow \mathcal{R}(\mu_{K}, SU(2))$ induced from the embedding $\mu_{K}\hookrightarrow T^{2}$. Let $[\sigma]\in \mathcal{R}(\mu_{K}, SU(2))$ be an element such that ${\rm tr}(\sigma(\mu_{K}))=2{\rm cos}(2\alpha \pi)$.
 Then $p^{-1}([\sigma])=S_{\alpha}\subset \mathcal{R}(T^{2}, SU(2))$ by the definition. Note that $S_{\alpha}$ is contained in the smooth part of the pillowcase, and $[\sigma]$ is also contained in the smooth part of $\mathcal{R}(\mu_{K}, SU(2))$ since the quotient singularity by the conjugacy action of $SU(2)$ does not happen when $\alpha\neq 0, \frac{1}{2}$.  Thus the kernel of the map
 $$dp_{[\sigma']}:T_{[\sigma']}\mathcal{R}(T^{2},SU(2))=H^{1}(T^{2};{\rm ad}{\rho})\rightarrow T_{[\sigma]}\mathcal{R}(\mu_{K}, SU(2))=H^{1}(\mu_{K}; {\rm ad}{\rho})$$
  induced on their  tangent spaces is  $T_{[\sigma]}S_{\alpha}$, where $[\sigma']=r([\rho])$. Note that $\mathcal{R}^{*}(S^{3}\setminus K, SU(2))$ is smooth near $[\rho]$ by Proposition \ref{immers'}.
  The composition of the natural maps
  \begin{equation}
T_{[\rho]}\mathcal{R}(S^{3}\setminus K, SU(2))=H^{1}(S^{3}\setminus K; {\rm ad}{\rho})\xrightarrow{j} H^{1}(T^{2}; {\rm ad}{\rho})\rightarrow H^{1}(\mu_{K}; {\rm ad}{\rho}).
\end{equation}
  is injective by our non-degeneracy assumption. Thus the image of $H^{1}(S^{3}\setminus K;{\rm ad}{\rho})$ in $ H^{1}(T^{2}; {\rm ad}{\rho})$ is independent of ${\rm Ker}(H^{1}(T^{2};{\rm ad}{\rho})\rightarrow H^{1}(\mu_{K}; {\rm ad}{\rho}) )$.
  This means that $r(\mathcal{R}^{*}(S^{3}\setminus K, SU(2)))$ and $S_{\alpha}$ intersect transversely at $r([\rho])$.
 \end{proof}
There is a relation between $\sigma_{\alpha}(K)$ and $\mathcal{R}^{*}_{\alpha}(S^{3}\setminus K, SU(2))$. We use the following inequality in the proof of Proposition \ref{T1}.

\begin{lem}\label{P1}
Let  $K$ be a knot in $S^{3}$. Assume that $\mathcal{R}^{*}_{\alpha}(S^{3}\setminus K, SU(2))$ is non-degenerate. Then
$$|\sigma_{\alpha}(K)|\leq2|\mathcal{R}^{*}_\alpha(S^{3}\setminus K, SU(2))|$$ for $\alpha \in [0, \frac{1}{2}]$ with $\Delta_{K}(e^{4 \pi i \alpha})\neq 0$.
\end{lem}
\begin{proof}
By Proposition \ref{immers'}, $\mathcal{R}^{*}(S^{3}\setminus K, SU(2))\rightarrow \mathcal{R}(T^{2}, SU(2))$ is an immersion to the smooth part of the pillowcase.
By Proposition \ref{non-deg'} and Lemma \ref{trans}, the immersed image of $\mathcal{R}^{*}(S^{3}\setminus K, SU(2))$ intersects transversely to $S_{\alpha}$.
After taking a small perturbation, the image of $\mathcal{R}^{*, h}(S^{3}\setminus K, SU(2))$ intersects to $S_{\alpha}$ transversely and the number of intersection points does not change,
\begin{equation}
|\mathcal{R}^{*}(S^{3}\setminus K, SU(2))\cap r^{-1}(S_{\alpha})|=|\mathcal{R}^{*, h}(S^{3}\setminus K,SU(2))\cap r^{-1}( S_{\alpha})|.\label{intersec}
\end{equation}

If $\Delta_{K}(e^{4\pi i \alpha})\neq 0$ and the perturbation $h$ is chosen so that it satisfies conditions in \cite[Lemma 5.1]{her97} then the signed counting $\#\mathcal{R}_{\alpha}^{*, h}(S^{3}\setminus K,SU(2)))$ can be defined and
\\
$$\# \mathcal{R}_{\alpha}^{*, h}(S^{3}\setminus K,SU(2)))=-\frac{1}{2}\sigma_{\alpha}(K)$$
holds by  \cite[Corollary 0.2]{her97}. On the other hand, the left hand side of (\ref{intersec}) is just the size of the set $\mathcal{R}^{*}_{\alpha}(S^{3}\setminus K, SU(2))$ by the definition. Hence the statement is proved.
\end{proof}
Since $K=T_{p, q}$ satisfies the assumption of Lemma \ref{P1}, the inequality $|\sigma_{\alpha}(T_{p, q})|\leq 2|\mathcal{R}^{*}_{\alpha}(S^{3}\setminus T_{p,q}, SU(2))|$ holds for $\alpha \in [0, \frac{1}{2}]$ with $\Delta_{T_{p, q}}(e^{4\pi i \alpha})\neq 0$.
\begin{prp}\label{T1}
Let $p, q ,r$ be positive and relatively coprime integers.
The following formula
$$-\frac{1}{2}\sigma_{\frac{l}{2r}}(T_{p, q})=|\mathcal{R}_{\frac{l}{2r}}^{*}(S^{3}\setminus T_{p, q}, SU(2))|$$
holds for $1\leq l\leq r-1 $ with $\Delta_{T_{p, q}}(e^{\frac{2\pi i  l}{r}})\neq 0$.
\end{prp}
For the proof, we use the similar argument in the proof of \cite[Theorem 3.4]{CSa}.
\begin{proof}
Consider a $4$-ball $B^{4}$ and a torus knot in its boundary $T_{p, q}\subset S^3 =\partial B^4$, and take a Seifert surface $S$ for $T_{p, q}$  as $S\subset B^4$ and $S\cap\partial B^4=\partial S $. The $r$-fold cyclic branched covering of $B^4$ branched along $S$ is the Milnor fiber
$$M(p, q, r)=\{(z_{1},z_{2},z_{3})|z^p_{1}+z^q_{2}+z^r_{3}=\epsilon\}\cap B^{6}\subset \mathbb{C}^3,$$
where $\epsilon>0$ is small enough. Furthermore, $\partial M(p, q, r)=\Sigma(p, q, r)$ is a $r$-fold cyclic branched covering of $\partial B^4=S^3$, blanched along $T_{p, q}$.
There is the following formula ( \cite[Corollary 2.9 ]{FS90}), 
$$-\frac{1}{4}\sigma(M(p, q, r))=|\mathcal{R}^{*}(\Sigma(p, q, r), SU(2)) |.$$
Using signature formula in \cite{Viro}, Lemma \ref{P2} and decomposition of $\sigma (M(p, q, r))$ into the equivariant signature $\sigma(M(p, q, r);\frac{i}{r})$, we have
$$-\frac{1}{2}\sum_{ l=1}^{r-1}\sigma_{\frac{l}{2r}}(T_{p, q})=\sum_{ l=1}^{r-1}|\mathcal{R}_{\frac{l}{2r}}^{*}(S^{3}\setminus T_{p, q}, SU(2))|.$$
Note that $\sigma_{\frac{l}{2r}}(T_{p, q})\leq 0$ since $T_{p, q}$ is a positive knot. If we assume that the inequality in Lemma \ref{P1} is strict for some $l$, then $-\frac{1}{4}\sigma(M(p, q, r ))<|\mathcal{R}^{*}(\Sigma(p, q, r))|$, and this is a contradiction. Thus we get the desired formula.
\end{proof}

We complete the proof of Theorem \ref{absolute counting}.
\begin{proof}[Proof of Theorem \ref{absolute counting}]
When $\alpha=0$ or $\frac{1}{2}$, $\sigma_{T_{p, q}}=0$ and $\mathcal{R}^{*}_{\alpha}(S^{3}\setminus T_{p, q},SU(2))$ is empty. So we consider the case  $\alpha \in (0, \frac{1}{2})$ with $\Delta_{T_{p, q}}(e^{4\pi i \alpha})\neq 0$. Since the image of $\mathcal{R}^{*}(S^{3}\setminus T_{p, q}, SU(2))$ in the pillowcase intersects to $S_{\alpha}$ transversely, there is small $\epsilon>0$  such that for any $\alpha' \in (\alpha-\epsilon, \alpha +\epsilon)$, $\Delta_{T_{p, q}}(e^{4\pi i \alpha})\neq 0$ and
$$|\mathcal{R}^{*}(S^{3}\setminus T_{p, q}, SU(2))\cap r^{-1}(S_{\alpha})|=|\mathcal{R}^{*}(S^{3}\setminus T_{p, q}, SU(2))\cap r^{-1}(S_{\alpha'})|$$
holds.
Thus $|\mathcal{R}^{*}_{\alpha}(S^{3}\setminus T_{p, q}, SU(2))|=|\mathcal{R}^{*}_{\alpha'}(S^{3}\setminus T_{p, q}, SU(2))|$.
The Tristram-Levine signature is piecewise constant and jumps at the root of the Alexander polynomial. Hence $\sigma_{T_{p, q}}(e^{4\pi i \alpha})=\sigma_{T_{p, q}}(e^{4 \pi i \alpha'})$ if $\epsilon>0$ is small enough. We can find a positive integer $r$ which is coprime to $p$ and $q$, and a positive integer  $l$ such that $\frac{l}{2r}\in (\alpha-\epsilon, \alpha+\epsilon)$. Then
$$-\frac{1}{2}\sigma_{\frac{l}{2r}}(T_{p,q})=|\mathcal{R}^{*}_{\frac{l}{2r}}(S^{3}\setminus T_{p, q}, SU(2))|$$
by Proposition \ref{T1}. Thus we have
$$-\frac{1}{2}\sigma_{\alpha}(T_{p, q})=|\mathcal{R}^{*}_{\alpha}(S^{3}\setminus T_{p, q}, SU(2))|.$$
\end{proof}

\section{Properties of instanton knot invariants and their applications}
In this section, we give the proof of Theorem \ref{main thm} which is the main theorem in this paper. The important consequence of Subsection \ref{eulerchar} is that the Floer chain $C^{\alpha}_{*}(T_{p, q};\Delta_{\mathscr{S}})$ is supported only on odd graded part. The key argument is that the Fr{\o}yshov  invariant of a knot $K\subset S^{3}$ for an appropriate choice of coefficient $\mathscr{S}$ reduces to the Tristram-Levine signature. This is a generalization of the corresponding result in \cite{DS2} and argument is parallel. Subsection \ref{self-con}  gives the proof of Theorem \ref{main thm} using this specific property of the Floer chain complex $\tilde{C}^{\alpha}_{*}(T_{p, q};\Delta_{\mathscr{S}})$ and  the Fr{\o}yshov knot invariant.
\subsection{The Fr{\o}yshov knot invariant and the structure  theorem}\label{eulerchar}
Let $W$ be a compact, oriented, smooth 4-manifold with $b^{1}(W)=b^{+}(W)=0$, whose boundary
$\partial W=Y$ is an integral homology 3-sphere. Let $K\subset Y$ be an oriented knot and $S\subset W$ be an embedded,  oriented surface with $\partial S=K$.
Throughout this subsection, we assume that $\mathscr{S}$ is an integral domain over $\mathscr{R}_{\alpha}$ . 
We define 
\[K(A):=\kappa(A)+\left(\alpha-\frac{1}{4}\right)\nu(A)+\alpha^{2}S\cdot S\]
and
$$d^{\alpha}(W, S):=4K(A_{\rm min})-g(S)-\frac{1}{2}\sigma_{\alpha}(Y, K)-1$$for each holonomy parameter $\alpha \in (0, \frac{1}{2})\cap \mathbb{Q}$. Here $A_{\rm min}$ is a minimal reducible and note that $K(A_{\rm min})$ is independent of the choice of minimal reducibles. Moreover $d^{\alpha}(W, S)$ is an integer by the index theorem.
The value of the Fr{\o}yshov knot invariant is evaluated by the following proposition.
\begin{prp}\label{ineq}
Let $(W, S)$ and $(Y, K)$ as above and $\alpha \in (0, \frac{1}{2})\cap\mathbb{Q}$ satisfies $\Delta_{(Y, K)}(e^{4\pi i \alpha})\neq 0$. 
If $d:=d^{\alpha}(W, S)\geq 0$ then there is a cycle $c^{\alpha}(W, S)\in C^{\alpha}_{2d+1}(Y, K; \Delta_{\mathscr{S}})$ satisfying
\[
\delta_{1}v^{j}(c^{\alpha}(W, S))=\begin{cases}
0&(0\leq j<d)\\
\eta^{\alpha}(W, S) & (j=d)
\end{cases}.
\]
\end{prp}
\begin{proof}
We define $c^{\alpha}(W, S)\in C^{\alpha}_{2d+1}(Y, K; \Delta_{\mathscr{S}})$ as follows
\[\langle c^{\alpha}(W, S), \beta\rangle=\sum_{[A]\in M(W, S;\beta)_{0}}\epsilon(A)\lambda^{\kappa_{0}-\kappa(A)}T^{\nu(A)-\nu_{0}},\]
where $M(W,S,\beta)_0$ is a zero dimensional moduli space. Since $dc^{\alpha}(W,S)$ corresponds to the counting of the boundary of one dimensional moduli space $M^{\alpha}(W, S; \beta')^{+}_{1}$, we have $dc^{\alpha}(W, S)=0$. 
Let $M(W,S, \theta_{\alpha})_{2d+1}$ be a moduli space of instantons $A$ over $(W, S)$ which are asymptotic to $\theta_{\alpha}$ and satisfy $\kappa(A)=\kappa(A_{\rm min})$.
If $d\geq 0$  then we can perturb ASD equation so that the neighborhood of a reducible connection in $M_{z}(W, S;\theta_{\alpha})_{2d+1}$ is homeomorphic to the cone of $\mathbb{CP}^{d}$.
Removing small $2d+1$ balls of each reducible points from $M_{z}(W,S;\theta_{\alpha})_{2d+1}$, then we get $2d+1$ manifold $M'_{z}$ whose boundary is $\sqcup(\pm \mathbb{CP}^{d})$, where signs $\pm$ are determined by the orientation of each reducible points.
Cutting down $M'_{z}$ by codimension two divisors $\{V_{i}\}_{1\leq i\leq d}$ associated to $d$ points in $S$, $M'_{z}\cap V_{1}\cap\cdots \cap V_{d}$ is a one manifold with boundary. 
Then we have $\sum_{z}\#\partial(  M_{z}'\cap V_{1}\cap \cdots \cap V_{d}))\lambda^{\kappa_{0}-\kappa(z)}T^{\nu(z)-\nu_{0}}=\eta^{\alpha}(W, S)$. On the other hand, $M'_{z}\cap V_{1}\cap\cdots \cap V_{d}$ has ends arising from sliding end of instantons. Define $\psi \in C_{1}(Y, K;\Delta_{\mathscr{S}})$ by
$$\langle \beta, \psi \rangle=\sum_{z}\sum_{[A]}\epsilon([A])\lambda^{\kappa_{0}-\kappa(A)}T^{\nu(A)-\nu_{0}}$$where  $[A]$ run through all elements in $\#(M_{z}(W, S;\beta)_{2d}\cap V_{1}\cap\cdots \cap V_{d})$ for each $z$.
Then $\psi$ and $v^{d}c^{\alpha}(W, S)$ are homologous. Since $\delta_{1}\psi=\eta^{\alpha}(W, S) $, we have $\delta_{1}v^{d}c^{\alpha}(W,S)=\eta^{\alpha}(W, S)$.
If $j<d$, $M_{z}(W,S; \theta_{\alpha})_{2j+1}$ does not contain reducible points and we have $\delta_{1}v^{j}c^{\alpha}(W, S)=0$.
\end{proof}

Before the proof of Theorem \ref{str}, we state Lemma \ref{tsig} and Proposition \ref{h=1} related to two-bridge torus knots.
\begin{lem}\label{tsig}
For any $\alpha\in (0, \frac{1}{2})$, there is an integer $k> 0$ such that $\sigma_{\alpha}(T_{2, 2k+1})=-2$ and $\Delta_{T_{2,2k+1}}(e^{4\pi i \alpha})\neq 0$.
\end{lem}
\begin{proof}
Consider the case $\alpha\leq\frac{1}{4}$. By \cite[Proposition 1]{Liz}, $\sigma_{\alpha}(T_{2, 2k+1})$ is given  by
\[\sigma_{\alpha}(T_{2, 2k+1})=n_{1}-n_{2}\]
where $n_{1}$ is a number of lattice points $\{(1, m)|(k+\frac{1}{2})(1+4\alpha)<m<2k+1\}$ and $n_{2}$ is a number of lattice points $\{(1, m)|0<m<(k+\frac{1}{2})(1+4\alpha)\}$. Thus 
$\sigma_{\alpha}(T_{2, 2k+1})=-2$ if only if $\frac{1}{8k+4}<\alpha<\frac{3}{8k+4}$. Moreover, note that the interval  $\left(\frac{1}{8k+4}, \frac{3}{8k+4}\right)$ does not contain any root of $\Delta_{T_{2, 2k+1}}(t)$. Thus, for any $\alpha\leq\frac{1}{4}$, we can find $k>0$ such that $\sigma_{\alpha}(T_{2, 2k+1})=-2$ and $\Delta_{T_{2,2k+1}}(e^{4\pi i \alpha})\neq 0$.   For the case $\alpha>\frac{1}{4}$, it follows that $\sigma_{\alpha}(T_{2, 2k+1})=-2$ if only if $\frac{1}{2}-\frac{3}{8k+4}<\alpha<\frac{1}{2}-\frac{1}{8k+4}$ by the flip symmetry. 
 \end{proof}

\begin{prp}\label{h=1}
For any $\alpha\in \mathbb{Q}\cap (0, \frac{1}{2})$, there is an integer $k>0$ such that $\Delta_{T_{2, 2k+1}}(e^{4\pi i\alpha})\neq 0$ and $h_{\mathscr{S}}^{\alpha}(T_{2, 2k+1})=1$.
\end{prp}
\begin{proof}
By Lemma \ref{tsig}, we can find an integer $k>0$ such that $\sigma_{\alpha}(T_{2, 2k+1})=-2$ and $\Delta_{T_{2,2k+1}}(e^{4\pi i \alpha})\neq 0$.
Consider a cobordism of pairs $(W_k, S_k)$ obtained by the compositions
\[(W_k, S_k): (S^3, U)\rightarrow (S^3, T_{2,3})\rightarrow \cdots \rightarrow (S^3, T_{2, 2k-1})\rightarrow (S^3, T_{2, 2k+1})\]
where $(S^3, T_{2,2i-1})\rightarrow (S^3,T_{2,2i+1})$ is obtained by crossing change of knots.
Put $(\overline{W_{k}}, \overline{S_k}):=(D^4, D^2)\cup_{(S^3, U)}(W_k, S_k)$.
Then it is easy to see that $b^1(\overline{W_k})=b^{+}(\overline{W_k})=0$ and $d^{\alpha}(\overline{W_k}, \overline{S_k})=0$ by the similar argument in Proposition \ref{neg def pair}.
Applying  Proposition \ref{ineq} to the pair $(\overline{W_k}, \overline{S_k})$, we obtain a cycle $c^{\alpha}(\overline{W_k}, \overline{S_k})\in C^{\alpha}_{1}(T_{2, 2k+1})$ such that
$\delta_{1}c^{\alpha}(\overline{W_k}, \overline{S_k})\neq 0$. This implies that $h_{\mathscr{S}}^{\alpha}(T_{2, 2k+1})\neq 0$. 
Since ${\rm rank}C^{\alpha}_{*}(T_{2, 2n+1})=1$, we have $h_{\mathscr{S}}^{\alpha}(T_{2, 2k+1})=1$.
\end{proof}
Now, we complete the proof of Theorem \ref{str}.
\begin{proof}[Proof of  Theorem \ref{str}]
Consider a knot $K\subset S^3$ and a holonomy parameter $\alpha\in \mathbb{Q}\cap(0, \frac{1}{2})$ with $\Delta_{K}(e^{4\pi i \alpha})\neq 0$.
Since $K\subset S^{3}$ is homotopic to any knot, it can be deformed into $lT_{2, 2n+1}$ by positive and negative crossings changes, where $l=-
\frac{1}{2}\sigma_{\alpha}(K)$.
This operation defines a cobordism of pairs $([0, 1]\times S^{3}, S): (S^{3}, K)\rightarrow (S^{3}, lT_{2, 2n+1})$ where $S$ is an immersed surface with normal self-intersection points.
Let $S':lT_{2, 2n+1}\rightarrow K$ be the inverse cobordism of $S$.
Since $\sigma_{\alpha}(K)=\sigma_{\alpha}(lT_{2,2n+1})$, two cobordisms $S$ and $S'$ induce negative definite cobordisms.
Let $\tilde{m}_{S}:\tilde{C}^{\alpha}_{*}(K;\Delta_{\mathscr{S}})\rightarrow \tilde{C}^{\alpha}_{*}(lT_{2,2n+1};\Delta_{\mathscr{S}})$ and $\tilde{m}_{S'}:\tilde{C}^{\alpha}_{*}(lT_{2,2n+1};\Delta_{\mathscr{S}})\rightarrow \tilde{C}^{\alpha}_{*}(K;\Delta_{\mathscr{S}})$
be induced cobordism maps on $\mathcal{S}$-complexes.
Since two immersed cobordisms $S'\circ S$ and $S\circ S'$ can be deformed in to product cobordisms by finitely many finger moves, its induced map $\tilde{m}_{S'\circ S}$ and $\tilde{m}_{S\circ S'}$ are $S$-chain homotopic to identities up to the multiplication of unit elements by Proposition \ref{immersed map}.
By the functoriality of $\mathcal{S}$-morphisms, $\tilde{m}_{S'}\circ \tilde{m}_{S}$ and $\tilde{m}_{S}\circ\tilde{m}_{S}$ are $\mathcal{S}$-chain homotopic to identities up to the multiplication of unit elements.
The proof is completed by Remark \ref{S chain unit}.
\end{proof}
The proof of Theorem \ref{Froyshov} immediately follows from Theorem \ref{str}.
\begin{proof}[Proof of Theorem \ref{Froyshov}]
Comparing Fr{\o}yshov invariants for both sides of
\[\widetilde{C}_{*}^{\alpha}(K; \Delta_{\mathscr{S}})\simeq C^{\alpha}_{*}(lT_{2,2k+1}; \Delta_{\mathscr{S}}),\]
we obtain $h_{\mathscr{S}}^{\alpha}(K)=lh_{\mathscr{S}}^{\alpha}(T_{2,2k+1})$ where $l=-\frac{1}{2}\sigma(K)$.
Since $h_{\mathscr{S}}^{\alpha}(T_{2, 2k+1})=1$ by Proposition \ref{h=1}, we obtain the desired formula. 
\end{proof}
\begin{rmk}
{\rm
Since $\mathcal{S}$-chain homotopy equivalence of two $\mathcal{S}$-complexes $\tilde{C}_{*}\simeq\tilde{C'}_{*}$
implies  chain homotopy equivalence between $C_{*}$ and $C'_{*}$.
Assume that $\sigma_{\alpha}(K)\leq 0$. 
Then Theorem \ref{connected sum S-cpx} and Theorem \ref{str} implies the $\mathcal{S}$-chain homotopy equivalence 
$\tilde{C}^{\alpha}_{*}(K;\Delta_{\mathscr{S}})\simeq \tilde{C}^{\alpha}_{*}(T_{2, 2n+1};\Delta_{\mathscr{S}})^{\otimes l}$ and hence we have the Euler characteristic formula, 
\begin{equation*}
\chi(C^{\alpha}_{*}(K;\Delta_{\mathscr{S}}))=l\chi(C^{\alpha}_{*}(T_{2, 2n+1};\Delta_{\mathscr{S}})).
\end{equation*}
If $\sigma_{\alpha}(K)>0$ then there is an $\mathcal{S}$-chain homotopy equivalence $\tilde{C}^{\alpha}_{*}(K;\Delta_{\mathscr{S}})\simeq \tilde{C}^{\alpha}_{*}(-T_{2, 2n+1};\Delta_{\mathscr{S}})^{\otimes -l}$
and we have
\[\chi(C^{\alpha}_{*}(K;\Delta_{\mathscr{S}}))=-l\chi(C^{\alpha}_{*}(-T_{2, 2n+1};\Delta_{\mathscr{S}})).\]By Proposition \ref{h=1}, we have $\chi(C^{\alpha}_{*}(T_{2, 2n+1};\Delta_{\mathscr{S}}))=-1$.
On the other hand, $\chi(C^{\alpha}_{*}(-T_{2, 2n+1};\Delta_{\mathscr{S}}))=1$ since if we reverse the orientation of $3$-manifold, the $\mathbb{Z}/4$-grading of chain complex changes so that ${\rm gr}_{-Y}(\beta)\equiv 3-{\rm gr}_{Y}(\beta)$, which follows from (\ref{adding formula for grading}).
In any case, we have 
\[\chi(C^{\alpha}_{*}(K; \Delta_{\mathscr{S}}))=\frac{1}{2}\sigma_{\alpha}(K).\]
Note that this formula for the Euler characteristic is independent of the choice of the coefficient $\mathscr{S}$.}
\end{rmk}

\begin{proof}[Proof of Theorem \ref{sub thm}]
Consider arbitrarily knot $K\subset S^{3}$. 
For any holonomy parameter $\alpha \in (0, \frac{1}{2})\cap \mathbb{Q}$ with $\Delta_{K}(e^{4\pi i \alpha})\neq 0$, the Floer chain complex $C^{\alpha}_{*}(K;\Delta_{\mathscr{S}})$ is defined and the relation
$h^{\alpha}_{\mathscr{S}}(K)=-\frac{1}{2}\sigma_{\alpha}(K)$ holds.
By the definition of the Fr{\o}yshov knot invariant, we have lower bounds of Floer homology groups,
\[{\rm rank}I^{\alpha}_{1}(K; \Delta_{\mathscr{S}})\geq \left\lceil -\frac{\sigma_{\alpha}(K)}{4}\right\rceil,\  {\rm rank}I^{\alpha}_{3}(K; \Delta_{\mathscr{S}})\geq \left\lfloor -\frac{\sigma_{\alpha}(K)}{4}\right\rfloor.\]
for any knot $K\subset S^{3}$ with $\sigma_{\alpha}(K)\leq 0$. 
In particular, $K=T_{p, q}$ satisfies this condition. Using the equality ${\rm rank} I_{*}(T_{p, q})=-\frac{\sigma_{\alpha}(T_{p, q})}{2}$, we obtain   
\[{\rm rank}I^{\alpha}_{1}(T_{p, q};\Delta_{\mathscr{S}} )=\left\lceil -\frac{\sigma_{\alpha}(T_{p, q})}{4}\right\rceil,\ {\rm rank}I^{\alpha}_{3}(T_{p, q};\Delta_{\mathscr{S}})=\left\lfloor-\frac{\sigma_{\alpha}(T_{p, q})}{4}\right\rfloor.\]
Since $I^{\alpha}_{*}(T_{p, q})$ is supported only on odd graded part, we obtain the statement.
 \end{proof}
\subsection{An application to knot concordance\label{self-con}}

In this subsection, we complete the proof of our main theorem (Theorem \ref{main thm}).

The operator $Z^{\pm 1}$ and $U^{\pm 1}$ extend to the $\mathcal{S}$-complex $\tilde{C}_{*}$ in the obvious way.
We also introduce the operator
\[\mathcal{W}_{i,j,k}:=\delta_{1}v^{i}U^{k}Z^{j}: \tilde{C}_{*}\rightarrow \tilde{C}_{*}.\]

If ${\rm deg}_{\mathbb{R}}(Z\gamma)>{\rm deg}_{\mathbb{R}}(\gamma)$, the operator $Z$ does not act on the filtered chain complex $C^{[-\infty, R]}_{*}$ and
$\mathcal{W}_{i, j}$ does not directly induces a map on $\tilde{C}^{[-\infty,R]}_{*}$. 
For this reason, we introduce the map $\mathcal{V}^{[-\infty, R]}_{i, j, k}$ on the filtered chain complex by the following composition:
\[\tilde{C}^{[-\infty,R]}_{*}\hookrightarrow \tilde{C}^{[-\infty, \infty]}_{*}\xrightarrow{\mathcal{W}_{i, j, k}}\tilde{C}^{[-\infty, \infty]}_{*}.\]
We also introduce the operator $\mathcal{W}^{[R', R]}_{i, j, k}$ on the quotient filtered $\mathcal{S}$-complex $\tilde{C}^{[R', R]}_{*}$ by the following compositions:
\[\tilde{C}^{[R',R]}_{*}\hookrightarrow \tilde{C}^{[-\infty, \infty]}_{*}\xrightarrow{\mathcal{W}_{i, j, k}}\tilde{C}^{[-\infty, \infty]}_{*}\twoheadrightarrow \tilde{C}^{[R', \infty]}_{*}.\]
Here, the last map is a natural quotient map.

\begin{prp}\label{selfconc}
Let $S:T_{p, q}\rightarrow T_{p,q}$ be a given self-concordance. Then there is a dense subset $\mathcal{I}\subset (0, \frac{1}{2})$ such that all elements in $\mathcal{R}_{\alpha}(S^{3}\setminus T_{p, q},SU(2))$ extend to elements in $\mathcal{R}_{\alpha}((S^{3}\times [0,1])\setminus S, SU(2))$ for any $\alpha\in \mathcal{I}$.
\end{prp}

\begin{proof}
We choose a dense subset $\mathcal{I}\subset (0, \frac{1}{2})$ such that Theorem \ref{sub thm} holds for $T_{p, q}$.
Since all irreducible critical points of the Chern-Simons functional of $T_{p, q}$ are non-degenerate by Proposition \ref{non-deg'}, we can choose a perturbation $\pi$ so that it is supported away from flat connections.
In particular, we can assume that the chain complex $C^{\alpha}_{*}(T_{p,q}; \Delta_{\mathscr{S}})$ is generated by $\mathcal{R}^{*}_{\alpha}(S^{3}\setminus T_{p, q},SU(2))$.
Since the assertion for $\alpha=\frac{1}{4}$ is proved in \cite{DS2}, we assume that $\alpha\neq \frac{1}{4}$. 
In particular, we consider the case $\alpha<\frac{1}{4}$ for a while.
Since the unique flat reducible $\theta_{\alpha}$ with holonomy parameter $\alpha$ on $S^{3}\setminus T_{p, q}$ always extends to concordance complement, it is enough to consider the extension problem for irreducibles.
We choose a field $\mathscr{S}:=\mathscr{R}_{\alpha}\otimes \mathbb{Q}$.
By Theorem \ref{Froyshov} and \ref{sub thm} , we have 
$$h_{\mathscr{S}}^{\alpha}(T_{p, q})=-\frac{1}{2}\sigma_{\alpha}(T_{p, q})=d$$
where $d:={\rm rank}C^{\alpha}_{*}(T_{p, q};\Delta_{\mathscr{S}})$. This implies that there is an cycle $\beta_{0}\in C^{\alpha}_{*}(T_{p, q};\Delta_{\mathscr{S}})$ such that $\delta_{1}v^{k}(\beta_{0})=0 $ if $k<d-1$ and $\delta_{1}v^{d-1}(\beta_{0})\neq 0$.
Put $\beta_{i}:=v^{i}(\beta_{0})$ for $0\leq i\leq d-1$.
The chain complex $C^{\alpha}_{*}(T_{p,q};\Delta_{\mathscr{S}_{\alpha}})$ admits a $\mathbb{Z}\times \mathbb{R}$-bigrading by fixing lifts $\tilde{\rho}_{1},\cdots ,\tilde{\rho}_{d}$ of singular flat connections $\rho_{1},\cdots,\rho_{d}\in \mathcal{R}^{*}_{\alpha}(S^{3}\setminus T_{p, q},SU(2))$.
In particular, we may assume that ${\rm deg}_{\mathbb{Z}}(\tilde{\rho}_{i})=1$ or $3$ by Theorem \ref{sub thm}. 
Using properties of elements $\beta_{0},\cdots,\beta_{d-1}$, we fix elements $\hat{\beta}_{0},\cdots ,\hat{\beta}_{d-1} \in C^{\alpha}_{*}(T_{p, q};\Delta_{\mathscr{S}_{\alpha}})^{[-\infty, \infty]}$ by the following way.
Firstly, there exists an cycle $\hat{\beta}_{d-1}$  such that ${\rm deg}_{\mathbb{Z}}(\hat{\beta}_{d-1})=1$ and satisfies
\[\delta_{1}(\hat{\beta}_{d-1})=\sum_{k\leq 0}c_{k}Z^{k}\tilde{\theta}_{\alpha}\]
with $c_{0}\neq 0$ since $\delta_{1}(\beta_{d-1})\neq 0$.
Next, we assume that an element $\hat{\beta}_{i}$ is chosen. Then $\hat{\beta}_{i-1}$ is defined as a cycle  satisfying 
\[v(\hat{\beta}_{i-1})=\hat{\beta}_{i}.\]
Finally, we obtain cycles $\hat{\beta}_{0}, \cdots, \hat{\beta}_{d-1}$ by induction.
\begin{figure}
    \centering
    \includegraphics[scale=0.59]{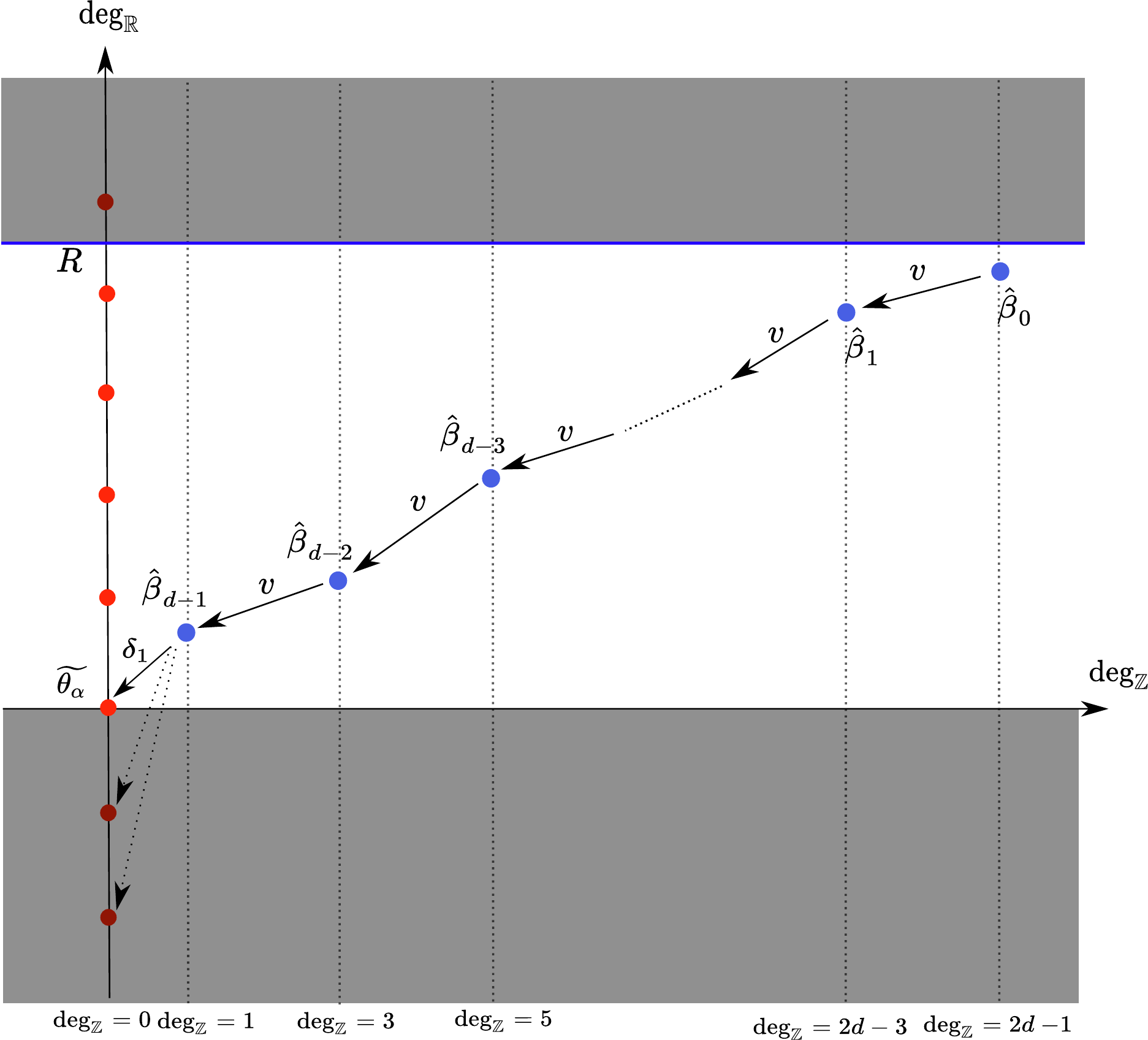}
    \caption{Elements $\hat{\beta}_{0},\cdots, \hat{\beta}_{d-1}$ and their $\mathbb{Z}\times \mathbb{R}$-gradings.}
    \label{vchain}
\end{figure}

Note that ${\rm deg}_{\mathbb{Z}}(\hat{\beta}_{i})=2(d-1)-2i$ and $0\leq{\rm deg}_{{\mathbb{R}}}(\hat{\beta}_{i-1})\leq\cdots \leq{\rm deg}_{{\mathbb{R}}}(\hat{\beta}_{0})$.
We fix $R>0$ so that it satisfies ${\rm deg}_{{\mathbb{R}}}(\hat{\beta}_{0})<R$ and $R\notin \mathscr{C}^{*}$.
See Figure \ref{vchain}.

Let $-\epsilon<0$ be a small negative number such that an interval $[-\epsilon, 0)$ does not contain any critical value of Chern-Simons functional $CS$.
Our aim is to show that the cobordism map $m_{S}$ on the quotient filtered chain $\tilde{C}^{\alpha}_{*}(T_{p, q};\Delta_{\mathscr{S}})^{[-\epsilon, R]}$ is an isomorphism.

Since the map $\tilde{m}_{S}$ preserves the $\mathbb{Z}$-grading, it is enough to show that $\tilde{m}_{S}$ is an isomorphism on $C^{\alpha}_{1}(T_{p,q};\Delta_{\mathscr{S}})^{[-\epsilon, R]}$ and $C^{\alpha}_{3}(T_{p,q};\Delta_{\mathscr{S}})^{[-\epsilon, R]}$.
We claim that $C^{\alpha}_{1}(T_{p,q};\Delta_{\mathscr{S}})^{[-\epsilon, R]}$ and $C^{\alpha}_{3}(T_{p,q};\Delta_{\mathscr{S}})^{[-\epsilon, R]}$ are generated by elements of the form \[\{Z^{j}U^{2k+1-(d-1)}\hat{\beta}_{2k+1}|k, m_{2k+1}\leq j\leq n_{2k+1}\}\] or \[\{Z^{j}U^{2k-(d-1)}\hat{\beta}_{2k}|k, m_{2k}\leq j\leq n_{2k}\}\] over $\mathbb{Q}$.
To see this, consider the following linear combination,
\begin{equation}
\sum_{0\leq i\leq d-1}\sum_{m_{j}\leq j\leq n_{i}}c_{i,j}Z^{j}U^{i-(d-1)}\hat{\beta_{i}}=0, \label{linear comb}
\end{equation}
where $c_{i,j}$ are rational coefficients.
Then we consider applying operators $\mathcal{W}^{[-\epsilon, R]}_{i, j, k}$ to the equation (\ref{linear comb}).
Firstly, we apply the operator for $(i, j,k)=(0, -n_{d-1}, 0)$.
Then we obtain $c_{0, n_{d-1}}\delta_{1}(\hat{\beta}_{d-1})=0$ with $\delta_{1}(\hat{\beta}_{d-1})$ is non-zero. Since $\mathscr{S}$ is an integral domain, we have $c_{d-1, n_{d-1}}=0$.
Next, we apply the operator $\mathcal{W}^{[-\epsilon, R]}_{(i, j, k)}$ for $(i,j,k)=(0, -n_{d-1}+1,0)$ to the equation (\ref{linear comb}). Then we obtain $c_{d-1, n_{d-1}-1}$ using $c_{d-1, n_{d-1}}=0$.
Inductively, we obtain 
\[c_{d-1,n_{d-1}}=\cdots =c_{d-1, m_{d-1}}=0\]
by applying operators $\mathcal{W}^{[-\epsilon, R]}_{0, -n_{d-1},0}, \cdots, \mathcal{W}^{[-\epsilon, R]}_{0, -m_{d-1},0}$.
We repeat similar arguments using operators $\{\mathcal{W}^{[-\epsilon, R]}_{1, j, 1}\}_{m_{d-2}\leq j\leq n_{d-2}}$, and obtain 
\[c_{d-2,n_{d-2}}=\cdots =c_{d-2, m_{d-2}}=0.\]
Inductively, we conclude that 
\[c_{i,n_{i}}=\cdots =c_{i, m_{i}}=0\]
for all $0\leq i\leq d-1$.
This means that $\{Z^{j}U^{i-(d-1)}\hat{\beta_{i}}\}_{\substack{0\leq i\leq d-1\\ m_{i}\leq j\leq n_{i}}}$ are linearly independent.

Put $\hat{\beta}'_{i}:=m_{S}(\hat{\beta}_{i})$.
Since the induced cobordism map $\tilde{m}_{S}$ on an $\mathcal{S}$-complex satisfies relations (\ref{delt d}),
elements $\hat{\beta}_{0},\cdots, \hat{\beta}_{d-1}$ have same properties and the same technique shows that $\{Z^{j}U^{i-(d-1)}\hat{\beta'_{i}}\}_{\substack{0\leq i\leq d-1\\ m_{i}\leq j\leq n_{i}}}$ are linearly independent. 
Moreover, ${\rm deg}_{\mathbb{R}}(\hat{\beta}_{i})={\rm deg}_{\mathbb{R}}(\hat{\beta}'_{i})$ by the construction of elements $\{\hat{\beta}_{i}\}$.
We conclude that the map ${m}_{S}$ is an isomorphism on $C^{\alpha}_{1}(T_{p, q};\Delta_{\mathscr{S}_{\alpha}})^{[-\epsilon, R]}$ and $C^{\alpha}_{3}(T_{p, q};\Delta_{\mathscr{S}_{\alpha}})^{[-\epsilon, R]}$.

Note that the chain complex $C^{\alpha}_{*}(T_{p, q};\Delta_{\mathscr{S}})$ is generated by those irreducible singular flat connections.
Then the degree one part  $C^{\alpha}_{1}(T_{p, q};\Delta_{\mathscr{S}})^{[-\epsilon, R]}$ of the quotient filtered chain complex is generated by elements
of the form $\{Z^{j}\tilde{\rho_{1}}\}_{m_{1}\leq j\leq n_{1}},\cdots, \{Z^{j}\tilde{\rho_{l}}\}_{m_{l}\leq j\leq n_{l}} $ over $\mathbb{Q}$.
We ordered these generators by values of Chern-Simons functional. 
Then the cobordism map $m_{S}$ can be represented by the following form
\begin{equation}
\left[\begin{array}{ccccc}L_{1}&O&\cdots &&\\ &L_{2}&O &\cdots&\\ & &\ddots&O&\cdots \\ &&&\ddots&O\\ &&&&L_{k}\end{array}\right]\label{matrix repre}
\end{equation}
where diagonal blocks $L_{i}$ are components which corresponds to basis with same value of the Chern-Simons  functional.
Note that components in $L_{i}$ are defined by counting (perturbed) flat connection over the concordance complement.
Since $m_{S}$ is an isomorphism on the degree one part, the matrix (\ref{matrix repre}) is invertible over $\mathbb{Q}$. 
Hence each diagonal blocks $L_{i}$ are also invertible. 
In particular, they do not contain any zero-column.
Since the regular condition on moduli space is open condition with respect to choices of perturbation, all flat connections $\rho_{1}, \cdots, \rho_{l}$ extend to flat connections over the concordance complement.
The similar argument works for $\rho_{l+1}, \cdots \rho_{d}$ and thus all elements in $\mathcal{R}_{\alpha}^{*}(S^{3}\setminus T_{p, q};\Delta_{\mathscr{S}})$ extend to the concordance complement.

Finally, we consider the case $\alpha>\frac{1}{4}$.  
In this case, we only change the above  argument at the following point:
We apply the operator $\mathcal{W}^{[-\epsilon, R]}_{0, -m_{d-1}}$ on the equation (\ref{linear comb}) at the first time.
Then we obtain $c_{d-1, m_{d-1}}=0$. 
Next we apply the operator $\mathcal{W}^{[-\epsilon, R]}_{0, -m_{d-1}+1,0}$ and obtain $c_{d-1, m_{d-1}-1 }=0$.
We inductively obtain $c_{d-1,n_{d-1}}=\cdots =c_{d-1, m_{d-1}}=0$.
The rest of the argument goes in the similar way and finally all coefficients in (\ref{linear comb}) vanish.
\end{proof}

\begin{proof}[Proof of Theorem \ref{main thm}]
Let $S:T_{p, q}\rightarrow K$ be a given concordance, then we can construct a concordance  $\bar{S}\circ S:T_{p, q}\rightarrow K\rightarrow T_{p, q}$ by the composition, where $\bar{S}$ is the opposite concordance of $S$. By Proposition \ref{selfconc} there exists a dense subset  $\mathcal{I}\subset [0,\frac{1}{2}]$ such that there is a extension $\mathcal{R}_{\alpha}(S^{3}\setminus T_{p, q}, SU(2))\rightarrow \mathcal{R}_{\alpha}((S^{3}\times [0, 1])\setminus \bar{S}\circ S, SU(2) )$ for any $\alpha \in \mathcal{I}$.  Let $\alpha\in [0, \frac{1}{2}]$ be any holonomy parameter and consider the representation $\rho:\pi_{1}(S^{3}\setminus T_{p, q})\rightarrow SU(2)$ with
\[\rho(\mu_{T_{p, q}})\sim
\left[\begin{array}{cc}
e^{2\pi i \alpha}&0\\
0&e^{-2\pi i \alpha}
\end{array}\right].
\]
Then we can choose a sequence $\{\alpha_{i}\}\subset \mathcal{I}$ such that $\lim_{i\rightarrow \infty}\alpha_{i}=\alpha$ and $SU(2)$ representations $\rho_{i}$ of $\pi_{1}(S^{3}\setminus T_{p, q})$ with
\[\rho_{i}(\mu_{T_{p, q}})\sim
\left[\begin{array}{cc}
e^{2\pi i \alpha_{i}}&0\\
0&e^{-2\pi i \alpha_{i}}
\end{array}\right].
\]
Since $\rho_{i}$ extends to $SU(2)$ representation $\Phi_{i}:\pi_{1}((S^{3}\times [0, 1])\setminus \bar{S}\circ S)\rightarrow SU(2)$, and we can choose a convergent subsequence of $\{\Phi_{i}\}$ which has the limiting representation $\Phi_{\infty}:\pi_{1}((S^{3}\times [0, 1])\setminus \bar{S}\circ S)\rightarrow SU(2)$. (Since $SU(2)$ is compact, we can choose a convergent subsequence $\{\Phi_{i}(x_{j})\}_{i}$ for each generator $x_{j}$ of $\pi_{1}((S^{3}\times [0, 1])\setminus \bar{S}\circ S)$, and $\lim_{i\rightarrow \infty}\Phi_{i}(x_{j})$ defines a limiting  representation $\Phi_{\infty}$.)  By restriction, we get a representation  $\pi_{1}((S^{3}\times [0, 1])\setminus  S)\rightarrow SU(2)$ which is the extension of $\rho$. 

\end{proof}

\section{Appendix: The connected sum theorem}

In this section, we give the proof of the connected sum theorem.
The connected sum theorem for non-singular settings was proved by \cite{Fuk96}, and singular setting with $\alpha=\frac{1}{4}$ was proved by \cite{DS1}.
We use the similar argument used in \cite{DS1} to prove our connected sum theorem (Theorem \ref{connected sum}).
Let us recall the settings which are introduced in Section 6 in \cite{DS1}.
Let $(Y, K)$ and $(Y', K')$ be two given knots in integral homology 3-spheres. Fixing base points $p\in K$ and $p'\in K'$and we take a pair connected  sum $(Y\# Y', K\#K')$ at these base points. We also fix a base point $p^{\#}\in K\# K'$. Consider a cobordism\[(W, S):(Y\sqcup Y', K\sqcup K')\rightarrow (Y\#Y', K\#K')\]
constructed by attaching a pair of $1$-handle $(D^{1}\times D^{3}, D^{1}\times D^{1})$ to the product cobordism $[0, 1]\times (Y\sqcup Y', K\sqcup K')$. Let \[(W', S'):(Y\#Y', K\#K')\rightarrow (Y\sqcup Y', K\sqcup K')\]be a cobordism of opposite direction. We define three oriented piece-wise smooth paths $\gamma, \gamma', \gamma^{\#}$on $S\subset W$. Assume that these paths intersect boundaries of the cobordism only at their edge points.  The path $\gamma$ starts from $p\in Y$ and ends at the base point $p^{\#}\in Y\#Y'$. 
Similarly, $\gamma'$ starts from $p'\in Y'$ and ends at $p^{\#}$. $\gamma^{\#}$ starts from $p\in Y$ and ends at $p'\in Y'$. Let us define the paths $\sigma, \sigma', \sigma^{\#}$ in $S'$ as mirrors of $\gamma, \gamma', \gamma^{\#}$ respectively.   
We use notations $\beta, \beta', \beta^{\#}$ (and their indexed versions) for critical points of the perturbed Chern-Simons functional on $(Y, K)$, $(Y', K')$ and $(Y\# Y',K\# K')$ respectively. $\theta_{\alpha}, \theta'_{\alpha}$ and $\theta^{\#}_{\alpha}$ denote unique flat reducibles on $(Y, K)$, $(Y', K')$ and $(Y\# Y',K\# K')$  respectively.
We use the reduced notations for $d$-dimensional moduli spaces as follows
\[M_{z}(\beta, \beta'; \beta^{\#})_{d}:=M_{z}(W,S;\beta, \beta', \beta^{\#})_{d }\]
\[M_{z}(\beta^{\#};\beta, \beta')_{d}:=M_{z}(W', S';\beta^{\#},\beta, \beta')_{d}\]
We drop $z$ from notations above if we consider all union of $z$.
We define maps 
\[H^{\gamma}:\mathcal{B}(W,S;\beta, \beta',\beta^{\#})\rightarrow S^{1}\]
\[H^{\gamma'}:\mathcal{B}(W,S;\beta, \beta',\beta^{\#})\rightarrow S^{1}\]
as in Subsection \ref{vmu}.  
The moduli space cut down by these maps are defined by as follows,
\[M_{\gamma,z}(\beta,\beta';\beta^{\#})_{d}:=\{[A]\in M_{z}(\beta, \beta';\beta^{\#})_{d+1}|H^{\gamma}([A])=s\}\]
\[M_{\gamma',z}(\beta,\beta';\beta^{\#})_{d}:=\{[A]\in M_{z}(\beta, \beta';\beta^{\#})_{d+1}|H^{\gamma'}([A])=s'\}\]
\[M_{\gamma\gamma',z}(\beta,\beta';\beta^{\#})_{d}:=\{[A]\in M_{z}(\beta, \beta';\beta^{\#})_{d+1}|H^{\gamma}([A])=s, H^{\gamma'}([A])=s'\}\]
where $s\in S^{1}$ is a generic point. 
The orientation of moduli space over $(W, S)$ is defined by the following way.
Let $o_{W}\in \mathscr{O}[W,S;\theta_{\alpha +}, \theta'_{\alpha +}, \theta^{\#}_{\alpha -}]$ be a canonical homology orientation of $(W,S)$, and $o_{\beta}\in\mathscr{O}[\beta]$, $o_{\beta'}\in\mathscr{O}[\beta']$and $o_{\beta^{\#}}\in \mathscr{O}[\beta^{\#}]$ be given orientations for generators. Then $o_{\beta, \beta': \beta^{\#}}\in \mathscr{O}[W, S:\beta, \beta'; \beta^{\#}]$ is fixed so that the relation
\[\Phi(o_{\beta}\otimes o_{\beta'}\otimes o_{W})=\Phi(o_{\beta, \beta';\beta^{\#}}\otimes o_{\beta^{\#}})\] holds.

The argument of the proof consists of the following steps.
\begin{itemize}
\item [I.]A pair of cobordism $(W,S):(Y\sqcup Y', K\sqcup K')\rightarrow (Y\#Y', K\#K')$ induces an $\mathcal{S}$-morphism \[\tilde{m}_{(W, S)}:\tilde{C}^{\alpha}_{*}(Y, K)\otimes \tilde{C}^{\alpha}_{*}(Y', K')\rightarrow \tilde{C}^{\alpha}_{*}(Y\#Y', K\# K').\]
\item [II.]A pair of cobordism $(W', S'): (Y\#Y', K\# K')\rightarrow (Y\sqcup Y', K\sqcup K')$ induces an $\mathcal{S}$-morphism\[\tilde{m}_{(W', S')}:\tilde{C}^{\alpha}_{*}(Y\# Y', K\#K')\rightarrow \tilde{C}^{\alpha}_{*}(Y, K)\otimes\tilde{C}^{\alpha}_{*}(Y', K').\]
\item [III.] Put $\tilde{C}^{\#}:=\tilde{C}^{\alpha}_{*}(Y\# Y', K\# K')$. The composition $\tilde{m}_{(W, S)}\circ \tilde{m}_{(W',S')}$ is $\mathcal{S}$-chain homotopic to ${\rm id}_{\tilde{C}^{\#}}$ up to the multiplication of a unit element in $\mathscr{S}$.
\item [IV.] The composition $\tilde{m}_{(W', S')}\circ \tilde{m}_{(W, S)}$ is $\mathcal{S}$-chain homotopic to ${\rm id}_{\tilde{C}^{\otimes}}$.
\end{itemize}

\subsection{Step I.}
We define a map $\tilde{m}_{(W, S)}$as follows.
Using the decomposition of the Floer chain group ${C}^{\otimes}=(C\otimes C')_{*}\oplus(C\otimes C')_{*-1}\oplus C_{*} \oplus C'_{*}$, we define four maps: 
\[m=[m_{1}, m_{2}, m_{3}, m_{4}]:(C\otimes C')_{*}\oplus(C\otimes C')_{*-1}\oplus C_{*}\oplus C'_{*}\rightarrow C^{\#}_{*},\]
\[\mu=[\mu_{1},\mu_{2},\mu_{3},\mu_{4}]:(C\otimes C')_{*}\oplus(C\otimes C')_{*-1}\oplus C_{*}\oplus C'_{*}\rightarrow C^{\#}_{*},\]
\[\Delta_{1}=[\Delta_{1,1},\Delta_{1,2},\Delta_{1,3},\Delta_{1,4}]:(C\otimes C')_{0}\oplus(C\otimes C')_{-1}\oplus C_{0}\oplus C'_{0}\rightarrow \mathscr{S},\]
\[\Delta_{2}:\mathscr{S}\rightarrow C^{\#}_{-1}.\]Each components of above maps are defined as follows:

\begin{eqnarray*}
\langle m_{1}(\beta\otimes \beta'),\beta^{\#} \rangle&=&\sum_{z}\#M_{\gamma^{\#}, z}(\beta, \beta';\beta^{\#})_{0}\lambda^{-\kappa(z)}T^{\nu(z)}\\
\langle m_{2}(\beta\otimes \beta'), \beta^{\#}\rangle&=&\sum_{z}\#M_{z}(\beta, \beta';\beta^{\#})_{0}\lambda^{-\kappa(z)}T^{\nu(z)}\\
\langle m_{3}(\beta),\beta^{\#} \rangle&=&\sum_{z}\#M_{z}(\beta, \theta'_{\alpha};\beta^{\#})_{0}\lambda^{-\kappa(z)}T^{\nu(z)}\\
\langle m_{4}(\beta'),\beta^{\#} \rangle&=&\sum_{z}\#M_{z}(\theta_{\alpha},\beta';\beta^{\#})_{0}\lambda^{-\kappa(z)}T^{\nu(z)}
\end{eqnarray*}

\begin{eqnarray*}
\langle\mu_{1}(\beta\otimes\beta'),\beta^{\#} \rangle&=&\sum_{z}\#M_{\gamma\gamma', z}(\beta, \beta';\beta^{\#})_{0}\lambda^{-\kappa(z)}T^{\nu(z)}\\
\langle\mu_{2}(\beta\otimes \beta'),\beta^{\#} \rangle&=&\sum_{z}\#M_{\gamma, z}(\beta,\beta';\beta^{\#})_{0}\lambda^{-\kappa(z)}T^{\nu(z)}\\
\langle\mu_{3}(\beta),\beta^{\#}\rangle&=&\sum_{z}\#M_{\gamma, z}(\beta, \theta'_{\alpha};\beta^{\#})_{0}\lambda^{-\kappa(z)}T^{\nu(z)}\\
\langle\mu_{4}(\beta'),\beta^{\#} \rangle&=&\sum_{z}\#M_{\gamma', z}(\theta_{\alpha},\beta';\beta^{\#})_{0}\lambda^{-\kappa(z)}T^{\nu(z)}
\end{eqnarray*}

\begin{eqnarray*}
\Delta_{1,1}(\beta\otimes \beta')&=&\sum_{z}\#M_{\gamma^{\#},z}(\beta, \beta';\theta_{\alpha}^{\#})_{0}\lambda^{-\kappa(z)}T^{\nu(z)}\\
\Delta_{1,2}(\beta\otimes \beta')&=&\sum_{z}\#M_{z}(\beta,\beta';\theta^{\#}_{\alpha})_{0}\lambda^{-\kappa(z)}T^{\nu(z)}\\
\Delta_{1,3}(\beta)&=&\sum_{z}\#M_{z}(\beta,\theta'_{\alpha};\theta^{\#}_{\alpha})_{0}\lambda^{-\kappa(z)}T^{\nu(z)}\\
\Delta_{1,4}(\beta')&=&\sum_{z}\#M_{z}(\theta'_{\alpha},\beta';\theta^{\#}_{\alpha})_{0}\lambda^{-\kappa(z)}T^{\nu(z)}
\end{eqnarray*}

\[\langle\Delta_{2}(1),\beta^{\#}\rangle=\sum_{z}\#M_{z}(\theta_{\alpha},\theta'_{\alpha};\beta^{\#})_{0}\lambda^{-\kappa(z)}T^{\nu(z)}\]
As described in Remark 6.10 and Remark 6.11 in \cite{DS1}, notice that
\begin{itemize}
\item $H^{-1}_{\beta\beta_{1}}(s)\cap H^{-1}_{\beta\beta'}(s')\cap M(\beta, \beta_{1})_{2}=\emptyset$ for distinct regular value $s, s'\in S^{1}$,
\item $\#M_{\gamma}(\beta, \beta';\beta^{\#})-\#M_{\gamma'}(\beta, \beta'; \beta^{\#})=\#M_{\gamma^{\#}}(\beta, \beta';\beta^{\#})$.
\end{itemize}

\begin{prp}\label{step1}
There are the following relations:
\begin{eqnarray}
d^{\#}\circ m&=&m\circ d^{\otimes}\label{7.1}\\
\delta_{1}^{\#}\circ m&=&\Delta_{1}\circ d^{\otimes}+\delta_{1}^{\otimes}\label{7.2}\\
m\circ \delta_{2}^{\otimes}&=&\delta_{2}^{\#}-d^{\#}\circ \Delta_{2},\label{7.3}\\
d^{\#}\circ \mu +\mu \circ d^{\otimes}&=&v^{\#}\circ m -m\circ v^{\otimes}+\delta_{2}^{\#}\circ \Delta_{1}-\Delta_{2}\circ \delta_{1}^{\otimes}\label{7.4}.
\end{eqnarray}
\end{prp}
\begin{proof}
Identity (\ref{7.1}) decomposes into the following four relations:
\begin{eqnarray}d^{\#}m_{1}&=&m_{1}(d\otimes 1)+m_{1}(\epsilon\otimes d')-m_{2}(\epsilon v\otimes 1)+m_{3}(\epsilon\otimes v')+m_{3}(\epsilon\otimes \delta'_{1})+m_{4}(\delta_{1}\otimes 1)\label{7.1.1}\\
d^{\#}m_{2}&=&m_{2}(d\otimes 1)-m_{2}(\epsilon\otimes d')\label{7.1.2}\\
d^{\#}m_{3}&=&m_{3}(\epsilon\otimes \delta'_{2})+m_{3}d\label{7.1.3}\\
d^{\#}m_{4}&=&-m_{2}(\delta_{2}\otimes 1)+m_{4}d'\label{7.1.4}
\end{eqnarray}The identity (\ref{7.1.1}) is obtained by counting boundary of the compactified moduli space $M^{+}_{\gamma^{\#}, z}(\beta, \beta';\beta^{\#})_{1}$ for each path $z$. 
In fact, the oriented boundary of $M^{+}_{\gamma^{\#}, z}(\beta, \beta';\beta^{\#})_{1}$ consists of the following types of codimension one faces:
\[
M_{\gamma^{\#},z'}(\beta, \beta'; \beta_{1}^{\#})_{0}\times \breve{M}_{z''}(\beta_{1}^{\#}, \beta^{\#})_{0},\ \breve{M}_{z'}(\beta, \beta_{1})_{0}\times M_{\gamma^{\#}, z''}(\beta_{1}, \beta'; \beta^{\#})_{0},\]
\[(-1)^{\rm gr(\beta)}\breve{M}_{z'}(\beta', \beta'_{1})_{0}\times M_{\gamma, z''}(\beta, \beta'_{1}; \beta^{\#})_{0},\ (-1)^{\rm gr(\beta)+1}(H^{-1}_{\beta\beta_{1}}(s)\cap M_{z'}(\beta, \beta_{1})_{1})\times M_{z''}(\beta_{1}, \beta';\beta^{\#})_{0},\]
\[(-1)^{\rm gr(\beta)}(H_{\beta'\beta'_{1}}^{-1}(s)\cap M_{z'}(\beta', \beta'_{1})_{1})\times M_{z''}(\beta, \beta'_{1}; \beta^{\#})_{0}, \ (-1)^{\rm gr(\beta)}\breve{M}_{z'}(\beta', \theta'_{\alpha})_{0}\times M_{z''}(\beta, \theta'_{\alpha};\beta^{\#})_{0},\]
\[\breve{M}_{z'}(\beta, \theta_{\alpha})_{0}\times M_{z''}(\theta_{\alpha}, \beta'; \beta^{\#})_{0}.\]
Here we explain how the orientation of each faces is determined:
\\

Identities (\ref{7.1.2}), (\ref{7.1.3}) and (\ref{7.1.4}) are obtained by counting the compactified moduli space $M^{+}_{z}(\beta, \beta'; \beta^{\#})_{1}$, $M^{+}_{z}(\beta, \theta'_{\alpha};\beta^{\#})_{1}$ and $M_{z}^{+}(\theta_{\alpha}, \beta';\beta^{\#})_{1}$ respectively. Here we list up codimension one faces of each moduli spaces:
\begin{itemize}
\item Codimension one faces of $\partial M_{z}^{+}(\beta, \beta'; \beta^{\#})_{1}$:\[M_{z'}(\beta,\beta'; \beta_{1}^{\#})_{0}\times \breve{M}_{z''}(\beta^{\#}_{1}, \beta^{\#})_{0}, \ \breve{M}_{z'}(\beta, \beta_{1})\times M_{z''}(\beta_{1}, \beta'; \beta^{\#})_{0},\  \]\[(-1)^{\rm gr(\beta)}\breve{M}_{z'}(\beta', \beta'_{1})_{0}\times M_{z''}(\beta, \beta'_{1};\beta^{\#})_{0}.\]
\item Codimension one faces of $\partial M^{+}_{z}(\beta, \theta'_{\alpha};\beta^{\#})_{1}$:\[{M}_{z'}(\beta, \theta'_{\alpha};\beta_{1}^{\#})_{0}\times \breve{M}_{z''}(\beta_{1}^{\#},\beta^{\#})_{0},\ (-1)^{\rm gr(\beta)}\breve{M}_{z'}(\theta'_{\alpha}, \beta')_{0}\times M_{z''}(\beta, \beta';\beta^{\#})_{0},\]\[\breve{M}_{z'}(\beta, \beta_{1})\times M_{z''}(\beta_{1}, \theta'_{\alpha};\beta^{\#})_{0}.\]
\item Codimension one faces of $\partial M^{+}_{z}(\theta_{\alpha}, \beta';\beta^{\#})_{1}$:\[M_{z'}(\theta_{\alpha},\beta'; \beta^{\#}_{1})_{0}\times \breve{M}_{z''}(\beta_{1}^{\#}, \beta^{\#})_{0}, \ \breve{M}_{z'}(\theta_{\alpha}, \beta)_{0}\times M_{z''}(\beta, \beta';\beta^{\#})_{0},\]\[\breve{M}_{z'}(\beta',\beta_{1}')_{0}\times M_{z''}(\theta_{\alpha},\beta'_{1};\beta^{\#})_{0}.\]
\end{itemize}

The relation (\ref{7.2}) decomposes into the following four identities:
\begin{eqnarray*}
\delta_{1}^{\#}m_{1}&=&\Delta_{1, 1}(d\otimes 1)+\Delta_{1,1}(\epsilon\otimes d')-\Delta_{1,2}(\epsilon v \otimes 1)+\Delta_{1,2}(\epsilon\otimes v')+\Delta_{1,3}(\epsilon\otimes \delta'_{1})+\Delta_{1,4}(\delta_{1}\otimes 1),\\
\delta^{\#}_{1}m_{2}&=&\Delta_{1,2}(d\otimes 1)-\Delta_{1,2}(\epsilon\otimes d')\\
\delta_{1}^{\#}m_{3}&=&\Delta_{1,2}(\epsilon\otimes \delta'_{2})+\Delta_{1,3}d+\delta_{1}\\
\delta_{1}^{\#}m_{4}&=&-\Delta_{1,2}(\delta_{2}\otimes 1)+\Delta_{1,4}d'
+\delta'_{1}
\end{eqnarray*}
Each relations are obtained by counting the boundary of compactified one-dimensional moduli spaces $M^{+}_{\gamma^{\#}, z}(\beta, \beta'; \theta^{\#}_{\alpha})_{1}$, $M^{+}_{z}(\beta, \beta'; \theta^{\#}_{\alpha})_{1}$ ,$M_{z}^{+}(\theta_{\alpha}, \beta';\theta^{\#}_{\alpha})_{1}$ and $M^{+}_{z}(\beta, \theta'_{\alpha};\theta_{\alpha}^{\#})_{1}$ for each paths $z$ and the argument is similar to the previous case. 
Note that $M(\theta_{\alpha}, \theta'_{\alpha}; \theta^{\#}_{\alpha})_{0}$ consists of unique reducible connection.
The relation (\ref{7.3}) reduces to 
\[m_{3}\delta_{2}+m_{4}\delta_{2}'=\delta^{\#}_{2}-d^{\#}\Delta_{2},\]
and this reduces to the counting of the boundary of $M^{+}_{z}(\theta_{\alpha}, \theta'_{\alpha}: \beta^{\#})_{1}$ whose codimension one faces are 
\[\breve{M}_{z'}(\theta_{\alpha}, \beta)_{0}\times M_{z''}(\beta, \theta'_{\alpha}; \beta^{\#})_{0},\ \breve{M}_{z}(\theta'_{\alpha}, \beta')_{0}\times M_{z''}(\beta, \beta'; \theta^{\#}_{\alpha})_{0},\]
\[\breve{M}_{z'}(\theta_{\alpha}, \theta'_{\alpha};\theta^{\#}_{\alpha})_{0}\times \breve{M}_{z''}(\theta^{\#}_{\alpha}, \beta^{\#})_{0}, \ M_{z'}(\theta_{\alpha}, \theta'_{\alpha};\beta^{\#}_{1})\times\breve{M}_{z''}(\beta_{1}^{\#}, \beta^{\#}).\]
The relation (\ref{7.4}) reduces to four identities:

\[d^{\#}\mu_{1}+\mu_{1}(d\otimes 1)+\mu_{1}(\epsilon \otimes d')-\mu_{2}(\epsilon v \otimes 1)+\mu_{2}(\epsilon\otimes v')+\mu_{3}(\epsilon\otimes \delta'_{1})+\mu_{4}(\delta_{1}\otimes 1)=v^{\#}m_{1}-m_{1}(v\otimes 1)+\delta^{\#}_{2}\Delta_{1,1},\]
\[d^{\#}\mu_{2}+(d\otimes 1)-\mu_{2}(\epsilon\otimes d')=v^{\#}m_{2}-m_{2}(v\otimes1)+m_{4}(\delta_{1}\otimes 1),\]
\[d^{\#}\mu_{3}+\mu_{2}(\epsilon\otimes \delta'_{2})+\mu_{3}d=v^{\#}m_{3}-m_{3}v+\delta_{2}^{\#}\Delta_{1,3}-\Delta_{2}\delta_{1},\]
\[d^{\#}\mu_{4}-\mu_{2}(\delta_{2}\otimes 1)+\mu_{4}d'=v^{\#}m_{4}-m_{1}(\delta_{1}\otimes 1)-m_{4}v'+\delta^{\#}_{2}\Delta_{1,3}-\Delta_{2}\delta'_{1}.\]

Each relations are obtained by counting the boundary of $M^{+}_{\gamma\gamma', z}(\beta, \beta';\beta^{\#})_{1}$, $M^{+}_{\gamma, z}(\beta, \beta'; \beta^{\#})_{1}$, $M^{+}_{\gamma, z}(\beta, \theta'_{\alpha}; \beta^{\#})_{1}$ and $M^{+}_{\gamma'}(\theta_{\alpha}, \beta';\beta^{\#})_{1}$. See also  \cite[Remark 6.11]{DS1}.

\end{proof}
\subsection{Step II.}
\[m'=[m'_{1}, m'_{2}, m'_{3}, m'_{4}]^{\rm T}:C^{\#}_{*}\rightarrow (C\otimes C')_{*}\oplus(C\otimes C')_{*-1}\oplus C_{*}\oplus C'_{*},\]
\[\mu'=[\mu'_{1},\mu'_{2},\mu'_{3},\mu'_{4}]^{\rm T}:C^{\#}_{*}\rightarrow(C\otimes C')_{*}\oplus(C\otimes C')_{*-1}\oplus C_{*}\oplus C'_{*},\]
\[\Delta'_{1}:C^{\#}_{1}\rightarrow\mathscr{S}.\]
\[\Delta'_{2}=[\Delta'_{2,1},\Delta'_{2,2},\Delta'_{2,3},\Delta'_{2,4}]^{\rm T}:\mathscr{S}\rightarrow(C\otimes C')_{-1}\oplus(C\otimes C')_{-2}\oplus C_{-1}\oplus C'_{-1},\]
Each components of above maps are defined as follows:

\begin{eqnarray*}
\langle m'_{1}(\beta^{\#}),\beta\otimes\beta' \rangle&=&\sum_{z}\#M_{z}(\beta^{\#};\beta, \beta')_{0}\lambda^{-\kappa(z)}T^{\nu(z)}\\
\langle m'_{2}(\beta^{\#}),\beta\otimes\beta' \rangle&=&\sum_{z}\#M_{\sigma, z}(\beta^{\#};\beta,\beta')_{0}\lambda^{-\kappa(z)}T^{\nu(z)}\\
\langle m'_{3}(\beta^{\#}),\beta \rangle&=&\sum_{z}\#M_{z}(\beta^{\#};\beta, \theta'_{\alpha})_{0}\lambda^{-\kappa(z)}T^{\nu(z)}\\
\langle m'_{4}(\beta^{\#}),\beta'\rangle&=&\sum_{z}\#M_{z}(\beta^{\#};\theta_{\alpha},\beta')_{0}\lambda^{-\kappa(z)}T^{\nu(z)}
\end{eqnarray*}

\begin{eqnarray*}
\langle\mu'_{1}(\beta\otimes\beta'),\beta^{\#} \rangle&=&\sum_{z}\#M_{\sigma, z}(\beta^{\#};\beta, \beta')_{0}\lambda^{-\kappa(z)}T^{\nu(z)}\\
\langle\mu'_{2}(\beta\otimes \beta'),\beta^{\#} \rangle&=&\sum_{z}\#M_{\sigma\sigma', z}(\beta^{\#};\beta,\beta')_{0}\lambda^{-\kappa(z)}T^{\nu(z)}\\
\langle\mu'_{3}(\beta),\beta^{\#}\rangle&=&\sum_{z}\#M_{\sigma, z}(\beta^{\#};\beta, \theta'_{\alpha})_{0}\lambda^{-\kappa(z)}T^{\nu(z)}\\
\langle\mu'_{4}(\beta'),\beta^{\#} \rangle&=&\sum_{z}\#M_{\sigma', z}(\beta^{\#};\theta_{\alpha},\beta')_{0}\lambda^{-\kappa(z)}T^{\nu(z)}
\end{eqnarray*}

\[\langle\Delta'_{1}(1),\beta^{\#}\rangle=\sum_{z}\#M_{z}(\beta^{\#};\theta_{\alpha},\theta'_{\alpha})_{0}\lambda^{-\kappa(z)}T^{\nu(z)}\]

\begin{eqnarray*}
\Delta'_{2,1}(\beta\otimes \beta')&=&\sum_{z}\#M_{z}(\beta^{\#};\beta, \beta')_{0}\lambda^{-\kappa(z)}T^{\nu(z)}\\
\Delta'_{2,2}(\beta\otimes \beta')&=&\sum_{z}\#M_{\sigma, z}(\theta^{\#}_{\alpha};\beta,\beta')_{0}\lambda^{-\kappa(z)}T^{\nu(z)}\\
\Delta'_{2,3}(\beta)&=&\sum_{z}\#M_{z}(\theta^{\#}_{\alpha};\beta,\theta'_{\alpha})_{0}\lambda^{-\kappa(z)}T^{\nu(z)}\\
\Delta'_{2,4}(\beta')&=&\sum_{z}\#M_{z}(\theta^{\#}_{\alpha};\theta'_{\alpha},\beta')_{0}\lambda^{-\kappa(z)}T^{\nu(z)}
\end{eqnarray*}

\begin{prp}
There are the following relations:
\begin{eqnarray}
d^{\otimes}\circ m'&=&m'\circ d^{\#}\\
\delta_{1}^{\otimes}\circ m'&=&\Delta'_{1}\circ d^{\#}+\delta_{1}^{\#}\\
m'\circ \delta_{2}^{\#}&=&\delta_{2}^{\otimes}-d^{\otimes}\circ \Delta'_{2},\\
d^{\otimes}\circ \mu' +\mu' \circ d^{\#}&=&v^{\otimes}\circ m' -m'\circ v^{\#}+\delta_{2}^{\otimes}\circ \Delta'_{1}-\Delta'_{2}\circ \delta_{1}^{\#}.
\end{eqnarray}
\end{prp}
\begin{proof}
The proof is similar to that of Proposition \ref{step1}. 
In this case, we consider the opposite cobordism $(W', S')$. 
\end{proof}
\subsection{Step III.}
Put $(W^{o}, S^{o}):=(W\circ W', S\circ S')$. We define compositions of paths $\rho^{\#}:=\gamma^{\#}\circ \sigma^{\#}$, $\rho:=\gamma\circ \sigma$ and $\rho':=\gamma'\circ \sigma'$. See Figure \ref{fig:paths on W^o}. 
We regard the configuration space of connections over $(W^{o}, S^{o})$ as quotient of the space of $SO(3)$-adjoint connections by determinant-$1$ gauge group $\mathcal{G}$.
Then there is an exact sequence
\[\mathcal{G}\hookrightarrow \mathcal{G}^{e}\twoheadrightarrow H^{1}(W^{o};\mathbb{Z}_{2}),\]
where $\mathcal{G}^{e}$ is a  $SO(3)$-gauge transformations and the last map gives the obstruction to lift an $SO(3)$-automorphism to $SU(2)$-automorphism over the $1$-skelton.
There is an action of $\mathcal{G}^{e}/\mathcal{G}\cong H^{1}(W^{o}, \mathbb{Z}_
{2})\cong \mathbb{Z}_{2}$ on the configuration space. Especially, there is an involution on
the moduli space $M(W^{o}, S^{o}; \beta^{\#}, \beta^{\#}_{1})_{d}$ and we denote its quotient by $M(W^{o}, S^{o};\beta^{\#}, \beta^{\#}_{1})_{d}^{e}$.  
We define
\[M_{\rho^{\#};z}(W^{o}, S^{o};\beta^{\#}, \beta^{\#}_{1})_{0}^{e}:=\{[A]\in M_{z}(W^{o}, S^{o};\beta^{\#}, \beta^{\#}_{1})_{1}^{e}|H^{\rho^{\#}}([A])=s\},\]
\[M_{\rho^{\#}\rho;z}(W^{o}, S^{o};\beta^{\#}, \beta^{\#}_{1})_{0}^{e}:=\{[A]\in M_{z}(W^{o}, S^{o};\beta^{\#}, \beta^{\#}_{1})_{2}^{e}|H^{\rho^{\#}}([A])=s, H^{\rho}([A])=s'\}.\]
The cardinality of these moduli spaces is half of that of the usual ones. Assume that $(W^{o}, S^{o})$ is equipped with a Riemannian metric with long neck along a cylinder $[0, 1]\times (Y\sqcup Y', K\sqcup K')$. Then  we have a good gluing relation
\begin{eqnarray*}M_{\rho^{\#}}(W^{o}, S^{o};\beta^{\#}, \beta^{\#}_{1})_{0}^{e}&=&\bigsqcup_{\beta, \beta'}M_{\sigma^{\#}}(W',S';\beta^{\#}, \beta, \beta')_{0}\times M(W, S;\beta, \beta';\beta_{1}^{\#})_{0}\\
&&\sqcup \bigsqcup_{\beta, \beta'}M(W',S';\beta^{\#}, \beta, \beta')_{0}\times M_{\gamma^{\#}}(W, S;\beta, \beta';\beta_{1}^{\#})_{0}\\
&&\sqcup\bigsqcup_{\beta'\in \mathfrak{C}'^{*}}M(W', S'; \beta^{\#}, \theta_{\alpha}, \beta')_{0}\times M(W, S;\theta_{\alpha}, \beta';\beta^{\#}_{1})_{0}\\
&&\sqcup\bigsqcup_{\beta\in \mathfrak{C}^{*}}M(W', S'; \beta^{\#},\beta, \theta_{\alpha})_{0}\times M(W, S;\beta, \theta'_{\alpha};\beta^{\#}_{1})_{0}
\end{eqnarray*}
Let $\tilde{m}_{(W^{o}, S^{o},\rho^{\#})}:\tilde{C}^{\#}_{*}\rightarrow \tilde{C}^{\#}_{*}$ be an $\mathcal{S}$-morphism whose components $m^{o}, \mu^{o}, \Delta^{o}_{1}$ and $\Delta^{o}_{2}$ are defined by
\begin{eqnarray*}
\langle m^{o}(\beta^{\#}),\beta^{\#}_{1}\rangle&=&\sum_{z}\#M_{\rho^{\#},z}(W^{o}, S^{o};\beta^{\#},\beta^{\#}_{1})^{e}_{0}\lambda^{-\kappa(z)}T^{\nu(z)}\\
\langle \mu^{o}(\beta^{\#}),\beta^{\#}_{1}\rangle&=&\sum_{z}\#M_{\rho^{\#}\rho,z}(W^{o}, S^{o};\beta^{\#},\beta^{\#}_{1})^{e}_{0}\lambda^{-\kappa(z)}T^{\nu(z)}\\
 \Delta^{o}_{1}(\beta^{\#})&=&\sum_{z}\#M_{\rho^{\#},z}(W^{o}, S^{o};\beta^{\#},\theta^{\#}_\alpha)^{e}_{0}\lambda^{-\kappa(z)}T^{\nu(z)}\\
 \Delta^{o}_{2}(1)\beta^{\#}&=&\sum_{z}\#M_{\rho^{\#},z}(W^{o}, S^{o};\beta^{\#},\theta^{\#}_\alpha)^{e}_{0}\lambda^{-\kappa(z)}T^{\nu(z)}.
\end{eqnarray*}

\begin{prp}
$\tilde{m}_{(W, S)}\circ \tilde{m}_{(W',S')}$ is $\mathcal{S}$-chain homotopic to $\tilde{m}_{(W\circ W', S\circ S' ;\rho^{\#})}$.
\end{prp}
\begin{proof}
\begin{figure}
    \centering
    \includegraphics[scale=0.8]{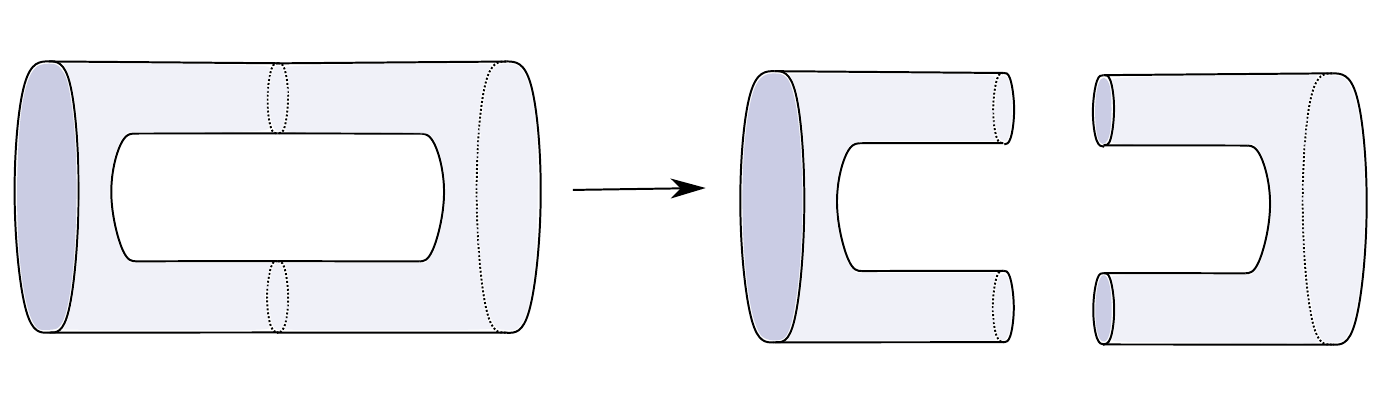}
    \caption{The family of metric $G^{o}$}
    \label{fig:G^o}
\end{figure}
\begin{figure}
    \centering
    \includegraphics[scale=1.3]{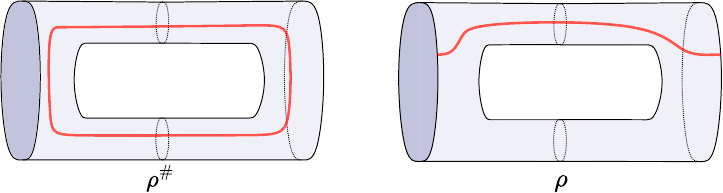}
    \caption{Paths on $(W^o, S^{o})$}
    \label{fig:paths on W^o}
\end{figure}
Let $G^{o}$ be a one parameter family of metric which stretch the cobordism $(W^{o}, S^{o})$ as in Figure \ref{fig:G^o}.
We modify the definition of $\mathcal{S}$-chain homotopy in Proposition 6.16 in \cite{DS1} as the following way:
\begin{eqnarray*}
\langle K^{o}(\beta^{\#}),\beta_{1}^{\#}\rangle&=&\sum_{z}\#\{[A]\in \bigcup_{g\in G^{o}}M_{z}^{g}(W^{o}, S^{o};\beta^{\#}, \beta_{1}^{\#})_{0}^{e}|H^{\rho^{\#}}([A])=s\}\lambda^{-\kappa(z)}T^{\nu(z)},\\
\langle L^{o}(\beta^{\#}),\beta^{\#}_{1}\rangle&=&\sum_{z}\#\{[A]\in \bigcup_{g\in G^{o}}M^{g}_{z}(W^{o}, S^{o}; \beta^{\#},\beta_{1}^{\#})_{0}^{e}|H^{\rho^{\#}}([A])=s, H^{\rho}([A])=t\}\lambda^{-\kappa(z)}T^{\nu(z)},\\
\langle M_{1}^{o}(\beta^{\#}),1\rangle&=&\sum_{z}\#\{[A]\in \bigcup_{g\in G^{o}}M_{z}^{g}(W^{o}, S^{o};\beta^{\#}, \theta^{\#})^{e}_{0}|H^{\rho^{\#}}([A])=s\}\lambda^{-\kappa(z)}T^{\nu(z)},\\
\langle M_{2}^{o}(1),\beta^{\#}\rangle&=&\sum_{z}\#\{[A]\in \bigcup_{g \in G^{o}}M^{g}_{z}(W^{o}, S^{o};\theta^{\#}_{\alpha},\beta^{\#})^{e}_{0}| H^{\rho^{\#}}([A])=s\}\lambda^{-\kappa(z)}T^{\nu(z)}.
\end{eqnarray*}
The rest of the argument is similar to  \cite[Proposition 6.16]{DS1} and we can check that \[H^{o}=\left[\begin{array}{ccc}
K^{o}&0&0\\
L^{o}&-K^{o}&M^{o}_{2}\\
M^{o}_{1}&0&0
\end{array}
\right]\]
gives an $\mathcal{S}$-chain homotopy from $\tilde{m}_{(W^{o}, S^{o})}$ to $\tilde{m}_{(W\sqcup W', S\sqcup S')}$.

\end{proof}
\begin{prp}
$\tilde{m}_{(W\circ W', S\circ S'; \rho^{\#})}$ is $\mathcal{S}$-chain homotopic to ${\rm id}_{\tilde{C}^{\#}}$ up to the multiplication of a unit element in $\mathscr{S}$.

\end{prp}
\begin{proof}
As in the proof of \cite[Proposition 6.17]{DS1}, we consider the decomposition,
\[(W^{o}, S^{o})=(W^{c}, S^{c})\cup (S^{1}\times D^{3}, S^{1}\times D^{1})\]
along $(S^{1}\times S^{2}, S^{1}\times {2{\rm pt}})$. 
We arrange the perturbation date on $(S^{1}\times D^{3}, S^{1}\times D^{1})$ and gluing region so that it is supported away from the moduli space of  flat connections.
Define the map $\tilde{m}^{+}$ as $\tilde{m}_{(W^{o}, S^{o}, \rho^{\#})}$ but using the metric on $(W^{o}, S^{o})$ which is stretched along the gluing region $(S^{1}\times S^{2}, S^{1}\times 2{\rm pt})$.
We write $m^{+}, \mu^{+}, \Delta^{+}_{1}$ and $\Delta^{+}_{2}$ for corresponding components of $\tilde{m}^{+}$.

Let $\mu_{1}$ be a generator of $S^{2}\setminus{\rm 2pt}$ factor and $\mu_{2}$ be a generator of $S^{1}$ factor in $\pi_{1}(S^{1}\times S^{2}\setminus S^{1}\times {\rm 2pt})$.
The equivalence classes of critical point set $\mathfrak{C}$ of the Chern-Simons functional on $(S^{1}\times S^{2}, S^{1}\times {\rm 2pt})$ can be identified with
\[\mathcal{R}_{\alpha}(S^{1}\times S^{2}\setminus S^{1}\times {\rm 2pt})=\{\beta \in {\rm Hom}(\pi_{1}, SU(2))|{\rm tr}\beta(\mu_{1})=2{\rm cos}(2\pi \alpha)\}/SU(2)\]
by the holonomy correspondence.
The character variety $\mathcal{R}_{\alpha}(S^{1}\times S^{2}\setminus S^{1}\times {\rm 2pt})$ is identified with $S^{1}$ as follows. 
Let $\mu_{1}$ be a generator of $\pi_{1}(S^{1}\times (S^{2}\setminus {\rm 2pt}))$ arising from $S^{2}\setminus {\rm 2pt}$ factor, and $\mu_{2}$ be another generator arising from $S^{1}$ factor.
Since ${\rm tr}\beta(\mu_{1})=2{\rm cos}(2\pi i \alpha)$, there is a element $g_{\beta}\in SU(2)$ with $g_{\beta}\beta(\mu_{1})g_{\beta}^{-1}=e^{2\pi i \alpha}\in S^{1}$ . Since $\mu_{1}$ and $\mu_{2}$ commute and $\alpha\neq 0,\frac{1}{2}$, there is $\theta(\beta)\in [0, 2\pi)$ and we have  $g_{\beta}\beta(\mu_{2})g_{\beta}^{-1}=e^{i\theta(\beta)}\in S^{1}$. The correspondence $\beta\mapsto e^{i\theta(\beta)}$ gives a bijection $\mathcal{R}_{\alpha}(S^{1}\times S^{2}\setminus S^{1}\times {\rm 2pt})\cong S^{1}$.

Let $A$ be a singular flat connection which is the extension of $\rho \in \mathfrak{C}$ over $(S^{1}\times D^{3}, S^{1}\times D^{1})$.
Since all elements in $\mathfrak{C}$ have $U(1)$-stabilizer, we have ${\rm dim}H^{0}(S^{1}\times D^{3}\setminus S^{1}\times D^{1}); {\rm ad}{A})=1$. We also have 
\[{\rm dim}H^{1}(S^{1}\times D^{3}\setminus S^{1}\times D^{1}; {\rm ad}{A})=1\]
by the computation of group cohomology of $\pi_{1}(S^{1}\times(D^{3}\setminus D^{1}))$.
Thus the critical point set $\mathfrak{C}=\mathcal{R}_{\alpha}(S^{1}\times (S^{2}\setminus {\rm 2pt}))$ is Morse-Bott non-degenerate. Consider a closed pair $(S^{1}\times S^{3},S^{1}\times S^{1})=(S^{1}\times D^{3}, S^{1}\times D^{1})\cup_{(S^{1}\times S^{2},S^{1}\times{\rm 2pt})}(S^{1}\times D^{3}, S^{1}\times D^{1})$. Then the gluing of index formula is  
\[2{\rm ind}D_{A}+{\rm dim}\mathfrak{C}+{\rm dim}{\rm Stab}(\rho)={\rm ind}D_{A\#_{\rho}A}.\]
Since ${\rm dim}\mathfrak{C}={\rm dim Stab}(\rho)=1$ and ${\rm ind}D_{A\#_{\rho}A}=0$ by the index formula for a closed pair, we have ${\rm ind}D_{A}=-1$. This implies that
\[{\rm dim}H^{2}(S^{1}\times D^{3}\setminus S^{1}\times D^{1});{\rm ad}{A})=0\]
and hence the  gluing theory is unobstructed at the flat connection.
Morse-Bott gluing theory tells us that  moduli space $M_{\rho^{\#}}(W^{o}, S^{o};\beta^{\#}, \beta_{1}^{\#})_{0}$ has the structure of union of fiber products as follows. 
\[M(W^{c}, S^{c};\beta^{\#}, \beta^{\#}_{1})_{d}\times_{\mathfrak{C}}M_{\rho^{\#}}(S^{1}\times D^{3}, S^{1}\times D^{1})^{\rm irred}_{d'}\ \ (d+d'=1)\]

\[M(W^{c}, S^{c};\beta^{\#}, \beta^{\#}_{1})_{1}\times_{\mathfrak{C}}M_{\rho^{\#}}(S^{1}\times D^{3}, S^{1}\times D^{1})^{\rm red}\]
The first case is excluded by the index reason.

Consider the restriction map 
\[r':M_{\rho^{\#}}(S^{1}\times D^{3}, S^{1}\times D^{1})^{\rm red}\rightarrow \mathfrak{C}.\]
By the holonomy condition $H^{\rho^{\#}}([A])=1$ on the moduli space $M(S^{1}\times D^{3}, S^{1}\times D^{1})$, the image of $r'$ consists of two points $\theta, \theta'\in \mathfrak{C}$. 
Hence, if the metric on $(W^{o}, S^{o})$ has long neck along $(S^{1}\times S^{2}, S^{2}\times 2{\rm pt})$, the moduli space $M_{\rho^{\#}}(W^{o}, S^{o};\beta^{\#}, \beta^{\#}_{1})$ is two copy of 
\[M(W^{c}, S^{c};\beta^{\#},\theta, \beta^{\#}_{1})_{0}\times M(S^{1}\times D^{3}, S^{1}\times D^{1};\theta)^{\rm red}.\]
In particular, 
\begin{eqnarray*}\label{unit}
&&\sum_{z}\#M_{\rho^{\#}, z}(W^{o},S^{o};\beta^{\#}, \beta^{\#}_{1})_{0}\lambda^{-\kappa(z)}T^{\nu(z)}\\&=&2
\sum_{z}\sum_{z'\circ z''=z}\#M_{z'}(S^{1}\times D^{3}, S^{1}\times D^{1};\theta)^{\rm red}\lambda^{-\kappa(z')}T^{\nu(z')}\cdot \#M_{z''}(W^{c}, S^{c};\beta^{\#}, \theta, \beta^{\#}_{1})_{0}\lambda^{-\kappa(z'')}T^{\nu (z'')}\\
&=&2\left(\sum_{k\leq 0}c_{k}Z^{k}\right)\cdot \sum_{z''}\#M_{z''}(W^{c}, S^{c};\beta^{\#}, \theta, \beta^{\#}_{1})_{0}\lambda^{-\kappa(z'')}T^{\nu (z'')}
\end{eqnarray*}
Since flat connections on $(S^{1}\times S^{2}, S^{1}\times 2{\rm pt})$ are uniquely extend to $(S^{1}\times D^{3}, S^{1}\times D^{1})$, $c_{0}=1$.

Since $2\#M_{\rho^{\#}}(W^{o}, S^{o};\beta^{\#}, \beta^{\#}_{1})_{0}^{e}=\#M_{\rho^{\#}}(W^{o}, S^{o};\beta^{\#}, \beta^{\#}_{1})_{0}$ , there is a unit element $C_{1}\in \mathscr{S}$ and we have 
\[\langle m^{+}(\beta^{\#}),\beta^{\#}_{1}\rangle=C_{1}\sum_{z}\#M_{z}(W^{c}, S^{c}; \beta^{\#}, \theta, \beta_{1}^{\#})_{0}\lambda^{-\kappa(z)}T^{\nu(z)}.\] Replacing $M_{\rho^{\#}}(W^{o},S^{o};\beta^{\#}, \beta^{\#}_{1})_{0}$ to $M_{\rho^{\#}\rho}(W^{o}, S^{o};\beta^{\#},\theta, \beta_{1}^{\#})_{0}$, $M_{\rho^{\#}}(W^{o}, S^{o};\beta^{\#}, \theta, \theta^{\#}_{\alpha})_{0}$, and $M_{\rho^{\#}}(W^{o},S^{o};\theta^{\#}_{\alpha}, \theta, \beta^{\#}_{1})_{0}$ in the argument above, we also have
\begin{eqnarray*}
\langle\mu^{+}(\beta^{\#}),\beta_{1}^{\#}\rangle&=&C_{1}\sum_{z}\#M_{\rho, z}(W^{c}, S^{c};\beta^{\#}, \theta,\beta^{\#}_{1})_{0}\lambda^{-\kappa(z)}T^{\nu(z)},\\
\langle \Delta^{+}_{1}(\beta^{\#}), 1\rangle&=&C_{1}\sum_{z}\#M_{z}(W^{o}, S^{o};\beta^{\#}, \theta, \theta^{\#}_{\alpha})_{0}\lambda^{-\kappa(z)}T^{\nu(z)},\\
\langle\Delta_{2}^{+}(1), \beta_{1}^{\#}\rangle&=&C_{1}\sum_{z}\#M_{z}(W^{o},S^{o};\theta^{\#}_{\alpha}, \theta, \beta^{\#}_{1})_{0}\lambda^{-\kappa(z)}T^{\nu(z)}.
\end{eqnarray*}
Replacing the pair $(S^{1}\times D^{3},S^{1}\times D^{1})$ with $(D^{2}\times S^{2},D^{2}\times 2{\rm pt})$, we obtain the product cobordism $[0, 1]\times (Y\# Y', K\# K')$. By stretching the metric on $[0, 1]\times (Y\#Y', K\# K')$ along the attaching domain, the moduli space $M([0,1]\times (Y\#Y', K\# K')\beta^{\#}, \beta^{\#}_{1})_{0}$ has the structure of the union of fiber products
\begin{equation}M(W^{c}, S^{c};\beta^{\#}, \beta^{\#}_{1})\times_{\mathfrak{C}}M(D^{2}\times S^{2}, D^{2}\times 2{\rm pt})^{red}. \label{fiber product}
\end{equation}
Let $A'$ be a extended flat connection on $(D^{2}\times S^{2}, D^{2}\times 2{\rm pt})$ of flat connection $\theta$.
Such $A'$ uniquely exists.
Moreover, it can be checked that the point $\Theta$ is unobstructed as follows. 
Consider the following closed pair 
\[(S^4, S^2):=(D^{2}\times S^{2}, D^{2}\times {\rm 2 pt})\cup_{(S^{1}\times S^{2}, S^{1}\times {\rm 2pt})}(S^{1}\times D^{3}, S^{1}\times D^{1})\]
and a glued reducible flat connection $A'\#_{\theta} A$ on $(S^4, S^2)$.
Then we have
\[{\rm ind}D_{A'\#_{\theta}A}={\rm ind}D_{A'}+{\rm dim}{\rm Stab}(\theta)+{\rm dim}\mathfrak{C}+{\rm ind}D_{A}.\]
Since $b^{1}(X)=b^{+}(X)=0$ and $S\cong S^{2}$, the index formula for a closed pair shows that ${\rm ind}D_{A'\#A}=-1$. 
Moreover, ${\rm ind}D_{A}=-1$ by the previous argument.
Thus we have ${\rm ind}D_{A'}=-2$.

Since ${\rm dim}H^{0}(D^{2}\times (S^{2}\setminus {2{\rm pt})}; {\rm ad}A')=1$ and ${\rm dim}H^{1}(D^{2}\times (S^{2}\setminus {2{\rm pt})}; {\rm ad}A')=0$ by the computation of group cohomology, ${\rm dim}H^{2}(D^{2}\times (S^{2}\setminus {2{\rm pt})}; {\rm ad}A')$ vanishes.

Now, a fiber product structure ($\ref{fiber product}$) implies that there is a unit element $C_{2}\in \mathscr{S}$ and
\[\langle m^{+}(\beta^{\#}),\beta^{\#}_{1}\rangle=C_{2}\sum_{z}\#M_{z}([0 ,1]\times (Y\# Y', K\# K');\beta^{\#}, \beta^{\#}_{1})_{0}\lambda^{-\kappa(z)}T^{\nu(z)}.\]
Similarly, we have 
\[\langle \mu^{+}(\beta^{\#}),\beta^{\#}_{1}\rangle=C_{2}\sum_{z}\#M_{\rho, z}([0 ,1]\times (Y\# Y', K\# K');\beta^{\#}, \beta^{\#}_{1})_{0}\lambda^{-\kappa(z)}T^{\nu(z)},\]
\[ \Delta_{1}^{+}(\beta^{\#})=C_{2}\sum_{z}\#M_{z}([0 ,1]\times (Y\# Y', K\# K');\beta^{\#}, \theta_{\alpha}^{\#})_{0}\lambda^{-\kappa(z)}T^{\nu(z)},\]
\[ \langle \Delta_{2}^{+}(1), \beta^{\#}\rangle=C_{2}\sum_{z}\#M_{z}([0 ,1]\times (Y\# Y', K\# K');\theta_{\alpha}^{\#}, \beta^{\#}_{1})_{0}\lambda^{-\kappa(z)}T^{\nu(z)}.\]
 Finally, there is a unit element $c\in \mathscr{S}$ and we have \[\tilde{m}^{+}=c\tilde{m}_{[0,1]\times(Y\#Y',K\#K')}.\] The right hand side is $\mathcal{S}$-chain homotopic to the identity since it is induced from the product cobordism.
By the construction, the unit element $c$ has the top term $1$, and hence $\tilde{m}_{(W\circ W', S\circ S';\rho^{\#})}$ is $\mathcal{S}$-chain homotopic to the identity up to the multiplication of a unit element in $\mathscr{S}$.
\end{proof}
\subsection{Step IV.}
Set $\bar{\rho}:=\sigma\circ\gamma$, $\bar{\rho}':=\sigma'\circ\gamma'$, $\tilde{\rho}:=\sigma'\circ\gamma$ , and $\tilde{\rho}':=\sigma\circ \gamma'$. See Figure \ref{fig:paths on W^I}.
\begin{figure}[hbtp]
    \centering
    \includegraphics[scale=1.0]{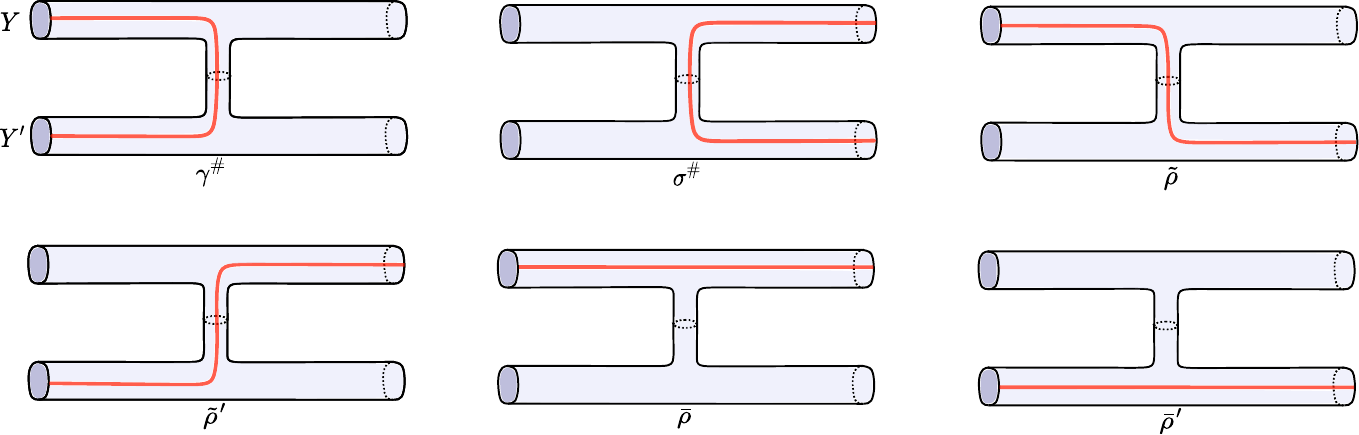}
    \caption{Paths on $(W^{I}, S^{I})$}
\label{fig:paths on W^I}
\end{figure}
\begin{prp}
$\tilde{m}_{(W', S')}\circ \tilde{m}_{(W, S)}$ is $\mathcal{S}$-chain homotopic to 
$\tilde{m}_{(W'\circ W, S'\circ S)}$.

\end{prp}
\begin{proof}
Let $G^{I}$ be a one parameter family of metric stretching $(W^{I}, S^{I}):=(W'\circ W, S'\circ S)$ along $(Y\# Y', K\#K')$.
Let  $\tilde{m}_{(W^{I}, S^{I})}$ be a cobordism map for $(W^{I},S^{I})$. We claim that there is a $\mathcal{S}$-chain homotopy $H^{I}$ such that
\[\tilde{d}^{\otimes}H^{I}+H^{I}\tilde{d}^{\otimes}=\tilde{m}_{(W', S')}\circ \tilde{m}_{(W, S)}-\tilde{m}_{(W^{I},S^{I})}.\]
Let us write each components of $H^{I}$ as\[H^{I}=\left[\begin{array}{ccc}
K^{I}&0&0\\
L^{I}&-K^{I}&M^{I}_{2}\\
M^{I}_{1}&0&0
\end{array}
\right]
\]

where $K^{I}$and $L^{I}$ are $4\times 4$ matrix. Before defining each components of these matrix, we introduce the following notation.
\[\mathfrak{m}^{I}_{z}(\beta, \beta', \beta_{1},\beta'_{1}):=\# \bigcup_{g\in G^{I}}M^{g}_{z}(W^{I}, S^{I}, \beta, \beta';\beta'_{1},\beta'_{1})_{-1},\]
\[\mathfrak{m}^{I}_{\circ_{1},\cdots, \circ_{d}; z}(\beta, \beta', \beta_{1},\beta'_{1}):=\#\{[A]\in \bigcup_{g\in G^{I}}M^{g}_{z}(W^{I}, S^{I}, \beta, \beta';\beta'_{1},\beta'_{1})_{d-1}|H^{\circ_{1}}([A])=s_{1},\cdots, H^{\circ_{d}}([A])=s_{d}\},\]
where $\circ_{1}, \cdots, \circ_{d}$ are elements in the set of paths $\{\gamma^{\#}, \sigma^{\#}, \bar{\rho}, \bar{\rho}', \tilde{\rho}, \tilde{\rho}'\}$. 
Then each components of $K^{I}, L^{I}, M_{1}^{I}$ and $M_{2}^{I}$ are given as follows:

\begin{itemize}
\item Components of $K^{I}$\begin{eqnarray*}
\langle K^{I}_{11}(\beta\otimes \beta'),\beta_{1}\otimes \beta'_{1}\rangle&=&\sum_{z}\mathfrak{m}_{\gamma^{\#};z}^{I}(\beta, \beta';\beta_{1}, \beta'_{1})\lambda^{-\kappa(z)}T^{\nu(z)}\\
\langle K^{I}_{12}(\beta\otimes \beta'),\beta_{1}\otimes \beta'_{1}\rangle&=&\sum_{z}\mathfrak{m}_{z}^{I}(\beta, \beta';\beta_{1},\beta'_{1})\lambda^{-\kappa(z)}T^{\nu(z)}\\
\langle K^{I}_{13}(\beta),\beta_{1}\otimes \beta'_{1}\rangle&=&\sum_{z}\mathfrak{m}_{z}^{I}(\beta, \theta_{\alpha}';\beta_{1},\beta'_{1})\lambda^{-\kappa(z)}T^{\nu(z)}\\
\langle K^{I}_{14}( \beta'),\beta_{1}\otimes \beta'_{1}\rangle&=&\sum_{z}\mathfrak{m}_{z}^{I}(\theta_{\alpha}, \beta';\beta_{1},\beta'_{1})\lambda^{-\kappa(z)}T^{\nu(z)}
\end{eqnarray*}\begin{eqnarray*}
\langle K^{I}_{21}(\beta\otimes \beta'),\beta_{1}\otimes \beta'_{1}\rangle&=&\sum_{z}\mathfrak{m}_{\gamma^{\#}, \sigma^{\#};z}^{I}(\beta, \beta';\beta_{1}, \beta'_{1})\lambda^{-\kappa(z)}T^{\nu(z)}\\
\langle K^{I}_{22}(\beta\otimes \beta'),\beta_{1}\otimes \beta'_{1}\rangle&=&\sum_{z}\mathfrak{m}_{\sigma^{\#};z}^{I}(\beta, \beta';\beta_{1},\beta'_{1})\lambda^{-\kappa(z)}T^{\nu(z)}\\
\langle K^{I}_{23}(\beta),\beta_{1}\otimes \beta'_{1}\rangle&=&\sum_{z}\mathfrak{m}_{\sigma^{\#};z}^{I}(\beta, \theta_{\alpha}';\beta_{1},\beta'_{1})\lambda^{-\kappa(z)}T^{\nu(z)}\\
\langle K^{I}_{24}( \beta'),\beta_{1}\otimes \beta'_{1}\rangle&=&\sum_{z}\mathfrak{m}_{\sigma^{\#};z}^{I}(\theta_{\alpha}, \beta';\beta_{1},\beta'_{1})\lambda^{-\kappa(z)}T^{\nu(z)}
\end{eqnarray*}
\begin{eqnarray*}
\langle K^{I}_{31}(\beta\otimes \beta'),\beta_{1}\rangle&=&\sum_{z}\mathfrak{m}_{\gamma^{\#};z}^{I}(\beta, \beta';\beta_{1}, \theta'_{\alpha})\lambda^{-\kappa(z)}T^{\nu(z)}\\
\langle K^{I}_{32}(\beta\otimes \beta'),\beta_{1}\rangle&=&\sum_{z}\mathfrak{m}_{z}^{I}(\beta, \beta';\beta_{1},\theta'_{\alpha})\lambda^{-\kappa(z)}T^{\nu(z)}\\
\langle K^{I}_{33}(\beta),\beta_{1}\rangle&=&\sum_{z}\mathfrak{m}_{z}^{I}(\beta, \theta_{\alpha}';\beta_{1},\theta'_{\alpha})\lambda^{-\kappa(z)}T^{\nu(z)}\\
\langle K^{I}_{34}( \beta'),\beta_{1}\rangle&=&\sum_{z}\mathfrak{m}_{z}^{I}(\theta_{\alpha}, \beta';\beta_{1},\theta'_{\alpha})\lambda^{-\kappa(z)}T^{\nu(z)}
\end{eqnarray*}
\begin{eqnarray*}
\langle K^{I}_{41}(\beta\otimes \beta'),\beta'_{1}\rangle&=&\sum_{z}\mathfrak{m}_{\gamma^{\#};z}^{I}(\beta, \beta'; \theta_{\alpha} ,\beta'_{1})\lambda^{-\kappa(z)}T^{\nu(z)}\\
\langle K^{I}_{42}(\beta\otimes \beta'),\beta'_{1}\rangle&=&\sum_{z}\mathfrak{m}_{z}^{I}(\beta, \beta';\theta'_{\alpha},\beta'_{1})\lambda^{-\kappa(z)}T^{\nu(z)}\\
\langle K^{I}_{43}(\beta),\beta'_{1}\rangle&=&\sum_{z}\mathfrak{m}_{z}^{I}(\beta, \theta_{\alpha}';\theta'_{\alpha},\beta'_{1})\lambda^{-\kappa(z)}T^{\nu(z)}\\
\langle K^{I}_{44}( \beta'),\beta'_{1}\rangle&=&\sum_{z}\mathfrak{m}_{z}^{I}(\theta_{\alpha}, \beta';\theta'_{\alpha},\beta_{1})\lambda^{-\kappa(z)}T^{\nu(z)}
\end{eqnarray*}
\item Components of $L^{I}$
\begin{eqnarray*}
\langle L^{I}_{11}(\beta\otimes \beta'),\beta_{1}\otimes \beta'_{1}\rangle&=&\sum_{z}\mathfrak{m}_{\gamma^{\#},\bar{\rho};z}^{I}(\beta, \beta';\beta_{1}, \beta'_{1})\lambda^{-\kappa(z)}T^{\nu(z)}\\
\langle L^{I}_{12}(\beta\otimes \beta'),\beta_{1}\otimes \beta'_{1}\rangle&=&\sum_{z}\mathfrak{m}_{\bar{\rho};z}^{I}(\beta, \beta';\beta_{1},\beta'_{1})\lambda^{-\kappa(z)}T^{\nu(z)}\\
\langle L^{I}_{13}(\beta),\beta_{1}\otimes \beta'_{1}\rangle&=&\sum_{z}\mathfrak{m}_{\bar{\rho};z}^{I}(\beta, \theta_{\alpha}';\beta_{1},\beta'_{1})_\lambda^{-\kappa(z)}T^{\nu(z)}\\
\langle L^{I}_{14}( \beta'),\beta_{1}\otimes \beta'_{1}\rangle&=&\sum_{z}\mathfrak{m}_{\bar{\rho}';z}^{I}(\theta_{\alpha}, \beta';\beta_{1},\beta'_{1})\lambda^{-\kappa(z)}T^{\nu(z)}
\end{eqnarray*}\begin{eqnarray*}
\langle L^{I}_{21}(\beta\otimes \beta'),\beta_{1}\otimes \beta'_{1}\rangle&=&\sum_{z}\mathfrak{m}_{\sigma^{\#},\bar{\rho};z}^{I}(\beta, \beta';\beta_{1}, \beta'_{1})\lambda^{-\kappa(z)}T^{\nu(z)}\\
\langle L^{I}_{22}(\beta\otimes \beta'),\beta_{1}\otimes \beta'_{1}\rangle&=&\sum_{z}\mathfrak{m}_{\sigma^{\#},\bar{\rho};z}^{I}(\beta, \beta';\beta_{1},\beta'_{1})\lambda^{-\kappa(z)}T^{\nu(z)}\\
\langle L^{I}_{23}(\beta),\beta_{1}\otimes \beta'_{1}\rangle&=&\sum_{z}\mathfrak{m}_{\sigma^{\#},\bar{\rho};z}^{I}(\beta, \theta_{\alpha}';\beta_{1},\beta'_{1})\lambda^{-\kappa(z)}T^{\nu(z)}\\
\langle L^{I}_{24}( \beta'),\beta_{1}\otimes \beta'_{1}\rangle&=&\sum_{z}\mathfrak{m}_{\sigma^{\#},\bar{\rho}';z}^{I}(\theta_{\alpha}, \beta';\beta_{1},\beta'_{1})\lambda^{-\kappa(z)}T^{\nu(z)}
\end{eqnarray*}
\begin{eqnarray*}
\langle L^{I}_{31}(\beta\otimes \beta'),\beta_{1}\rangle&=&\sum_{z}\mathfrak{m}_{\gamma^{\#},\bar{\rho};z}^{I}(\beta, \beta';\beta_{1}, \theta'_{\alpha})\lambda^{-\kappa(z)}T^{\nu(z)}\\
\langle L^{I}_{32}(\beta\otimes \beta'),\beta_{1}\rangle&=&\sum_{z}\mathfrak{m}_{\bar{\rho};z}^{I}(\beta, \beta';\beta_{1},\theta'_{\alpha})\lambda^{-\kappa(z)}T^{\nu(z)}\\
\langle L^{I}_{33}(\beta),\beta_{1}\rangle&=&\sum_{z}\mathfrak{m}_{\bar{\rho}';z}^{I}(\beta, \theta_{\alpha}';\beta_{1},\theta'_{\alpha})\lambda^{-\kappa(z)}T^{\nu(z)}\\
\langle L^{I}_{34}( \beta'),\beta_{1}\rangle&=&\sum_{z}\mathfrak{m}_{\tilde{\rho}'; z}^{I}(\theta_{\alpha}, \beta';\beta_{1},\theta'_{\alpha})\lambda^{-\kappa(z)}T^{\nu(z)}
\end{eqnarray*}
\begin{eqnarray*}
\langle L^{I}_{41}(\beta\otimes \beta'),\beta'_{1}\rangle&=&\sum_{z}\mathfrak{m}_{\gamma^{\#} ,\bar{\rho}';z}^{I}(\beta, \beta'; \theta_{\alpha},\beta'_{1})\lambda^{-\kappa(z)}T^{\nu(z)}\\
\langle L^{I}_{42}(\beta\otimes \beta'),\beta'_{1}\rangle&=&\sum_{z}\mathfrak{m}_{\tilde{\rho};z}^{I}(\beta, \beta';\theta'_{\alpha},\beta'_{1})\lambda^{-\kappa(z)}T^{\nu(z)}\\
\langle L^{I}_{43}(\beta),\beta'_{1}\rangle&=&\sum_{z}\mathfrak{m}_{\tilde{\rho};z}^{I}(\beta, \theta_{\alpha}';\theta'_{\alpha},\beta'_{1})\lambda^{-\kappa(z)}T^{\nu(z)}\\
\langle L^{I}_{44}( \beta'),\beta'_{1}\rangle&=&\sum_{z}\mathfrak{m}_{\bar{\rho}';z}^{I}(\theta_{\alpha}, \beta';\theta'_{\alpha},\beta_{1})\lambda^{-\kappa(z)}T^{\nu(z)}
\end{eqnarray*}
\item Components of $M^{I}_{1}$\begin{eqnarray*}
 M^{I}_{1,1}(\beta\otimes\beta')&=&\sum_{z}\mathfrak{m}_{\gamma^{\#};z}^{I}(\beta, \beta';\theta_{\alpha}, \theta'_{\alpha})\lambda^{-\kappa(z)}T^{\nu(z)}\\
 M^{I}_{1,2}(\beta\otimes\beta')&=&\sum_{z}\mathfrak{m}_{z}^{I}(\beta, \beta';\theta_{\alpha}, \theta'_{\alpha})\lambda^{-\kappa(z)}T^{\nu(z)}\\
 M^{I}_{1,3}(\beta)&=&\sum_{z}\mathfrak{m}_{z}^{I}(\beta, \theta'_{\alpha};\theta_{\alpha}, \theta'_{\alpha})\lambda^{-\kappa(z)}T^{\nu(z)}\\
 M^{I}_{1,4}(\beta')&=&\sum_{z}\mathfrak{m}_{z}^{I}(\theta_{\alpha}, \beta';\theta_{\alpha}, \theta'_{\alpha})\lambda^{-\kappa(z)}T^{\nu(z)}
\end{eqnarray*}
\item Components of $M^{I}_{2}$\begin{eqnarray*}
 \langle M^{I}_{2,1}(1),\beta\otimes\beta'\rangle&=&\sum_{z}\mathfrak{m}_{z}^{I}(\theta_{\alpha}, \theta'_{\alpha};\beta, \beta')\lambda^{-\kappa(z)}T^{\nu(z)}\\
 \langle M^{I}_{2,2}(1),\beta\otimes\beta'\rangle&=&\sum_{z}\mathfrak{m}_{\sigma^{\#};z}^{I}(\theta_{\alpha}, \theta'_{\alpha};\beta, \beta')\lambda^{-\kappa(z)}T^{\nu(z)}\\
 \langle M^{I}_{2,3}(1),\beta\rangle&=&\sum_{z}\mathfrak{m}_{z}^{I}(\theta_{\alpha}, \theta'_{\alpha};\beta, \theta'_{\alpha})\lambda^{-\kappa(z)}T^{\nu(z)}\\
 \langle M^{I}_{2,4}(1),\beta'\rangle&=&\sum_{z}\mathfrak{m}_{z}^{I}(\theta_{\alpha}, \theta'_{\alpha};\theta_{\alpha}, \beta')\lambda^{-\kappa(z)}T^{\nu(z)}
\end{eqnarray*}
\end{itemize}
Then we can check that there are following identities:
\begin{eqnarray}
d^{\otimes}K^{I}+K^{I}d^{\otimes}&=&m'm-m^{I}\label{7.17}\\
v^{\otimes}K^{I}-d^{\otimes}L^{I}+\delta_{2}^{\otimes}M^{I}_{1}+L^{I}d^{\otimes}-K^{I}v^{\otimes}+M_{2}^{I}\delta_{1}^{\otimes}&=&\mu'm+m'\mu+\Delta_{2}'\Delta_{1}-\mu^{I}\\
\delta_{1}^{\otimes}K^{I}+M_{1}^{I}d^{\otimes}&=&\Delta'_{1}m+\Delta_{1}-\Delta^{I}_{1}\\
-d^{\otimes}M_{2}^{I}-K^{I}\delta_{2}^{\otimes}&=&m'\Delta_{2}+\Delta_{2}'-\Delta_{2}^{I}
\end{eqnarray}
Identities above are proved by counting oriented boundaries of corresponding moduli spaces. 
For example, such moduli spaces  for identity (\ref{7.17}) are given as in Table \ref{tab1}.
Other identities can be proved in a similar way.
\begin{table}[h]
\centering
\caption{}
\label{tab1}
\begin{tabular}{|c|c|}
\hline
The component of (\ref{7.17})& The corresponding family of moduli space\\
\hline\hline
(1, 1) &$\{[A]\in \bigcup_{g\in G^{I}}M_{z}^{g}(W^{I}, S^{I};\beta, \beta';\beta_{1}, \beta'_{1})_{0}|H^{\gamma^{\#}}([A])=s\}$\\ 
(1, 2)&$\bigcup_{g\in G^{I}}M_{z}^{g}(W^{I}, S^{I};\beta, \beta';\beta_{1},\beta'_{1})_{0}$\\
(1, 3)&$\bigcup_{g\in G^{I}}M_{z}^{g}(W^{I}, S^{I};\beta, \theta_{\alpha}';\beta_{1},\beta'_{1})_{0}$\\
(1, 4)&$\bigcup_{g\in G^{I}}M_{z}^{g}(W^{I}, S^{I};\theta_{\alpha}, \beta';\beta_{1},\beta'_{1})_{0}$\\
\hline
(2, 1)&$\{[A]\in \bigcup_{g\in G^{I}}M_{z}^{g}(W^{I}, S^{I};\beta, \beta';\beta_{1}, \beta'_{1})_{2}|H^{\gamma^{\#}}([A])=s,H^{\sigma^{\#}}([A])=t\}$\\
(2, 2)&$\{[A]\in\bigcup_{g\in G^{I}}M_{z}^{g}(W^{I}, S^{I};\beta, \beta';\beta_{1},\beta'_{1})_{1}|H^{\sigma^{\#}}([A])=s\}$\\
(2, 3)&$\{[A]\in \bigcup_{g\in G^{I}}M_{z}^{g}(W^{I}, S^{I};\beta, \theta_{\alpha}';\beta_{1},\beta'_{1})_{1}|H^{\sigma^{\#}}([A])=s\}$\\
(2, 4)&$\{[A]\in\bigcup_{g\in G^{I}}M_{z}^{g}(W^{I}, S^{I};\theta_{\alpha}, \beta';\beta_{1},\beta'_{1})_{1}|H^{\sigma^{\#}}([A])=s\}$\\
\hline
(3, 1)&$\{[A]\in \bigcup_{g\in G^{I}}M_{z}^{g}(W^{I}, S^{I};\beta, \beta';\beta_{1}, \theta'_{\alpha})_{1}|H^{\gamma^{\#}}([A])=s\}$\\
(3, 2)&$\bigcup_{g\in G^{I}}M_{z}^{g}(W^{I}, S^{I};\beta, \beta';\beta_{1},\theta'_{\alpha})_{0}$\\
(3, 3)&$\{[A]\in\bigcup_{g\in G^{I}}M_{z}^{g}(W^{I}, S^{I};\beta, \theta_{\alpha}';\beta_{1},\theta'_{\alpha})_{1}|H^{\bar{\rho}'}([A])=t\}$\\
(3, 4)&$\bigcup_{g\in G^{I}}M_{z}^{g}(W^{I}, S^{I};\theta_{\alpha}, \beta';\beta_{1},\theta'_{\alpha})_{0}$\\
\hline
(4, 1)&$\{[A]\in \bigcup_{g\in G^{I}}M_{z}^{g}(W^{I}, S^{I};\beta, \beta'; \theta_{\alpha},\beta'_{1})_{1}|H^{\gamma^{\#}}([A])=s\}$\\
(4, 2)&$\bigcup_{g\in G^{I}}M_{z}^{g}(W^{I}, S^{I};\beta, \beta';\theta'_{\alpha},\beta'_{1})_{0}$\\
(4, 3)&$\bigcup_{g\in G^{I}}M_{z}^{g}(W^{I}, S^{I};\beta, \theta_{\alpha}';\theta'_{\alpha},\beta'_{1})_{0}$\\
(4, 4)&$\bigcup_{g\in G^{I}}M_{z}^{g}(W^{I}, S^{I};\theta_{\alpha}, \beta';\theta'_{\alpha},\beta_{1},)_{0}$\\
\hline
\end{tabular}
\end{table}
\end{proof}
\begin{prp}
$\tilde{m}_{(W'\circ W, S'\circ S)}$ is $\mathcal{S}$-chain homotopic to ${\rm id}_{\tilde{C}^{\otimes}}$.
\end{prp}
\begin{proof}
Consider a family of metric $G'^{I}$ on $(W^{I}, S^{I})$ which stretch the cobordism along $(S^{3}, S^{1})$ as in Figure\ref{fig:long neck W^I}.
\begin{figure}[hbtp]
    \centering
    \includegraphics[scale=0.5]{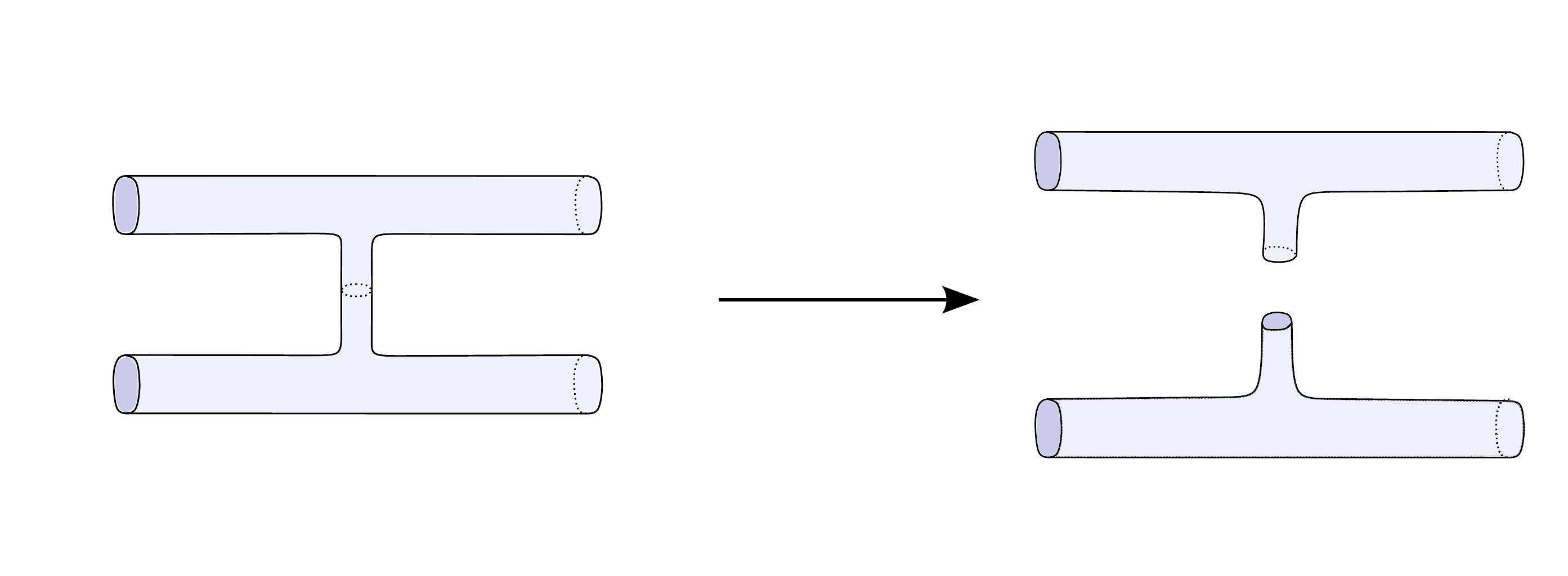}
   \caption{Family of metric $G'^{I}$}
    \label{fig:long neck W^I}
\end{figure}
Let $\tilde{m}^{I}$ be a map defined by long stretched metric on $(W^{I}, S^{I})$.
The family of  metric $G'^{I}$ gives $\mathcal{S}$-chain homotopy between $\tilde{m}_{(W^{I}, S^{I})}$ and $\tilde{m}^{I}$.
Let $(W^{c'}, S^{c'})$ be a disjoint union  \[(Y\times [0, 1]\setminus D^{4}, K\times [0, 1]\setminus D^{2})\sqcup(Y'\times [0, 1]\setminus D^{4}, K'\times [0, 1]\setminus D^{2}).\]
$\tilde{m}^{I}$ can also be defined by counting instantons on $(W^{c'}, S^{c'})$.
We will show that $\tilde{m}^{I}$ is an isomorphism of $\mathcal{S}$-complex.
We obtain a pair of cylinder $[0, 1]\times (Y\sqcup Y', K\sqcup K')$ by  gluing back two pair of disk $(D^{4}, D^{2})$ to $(W^{c'}, S^{c'})$.
Consider a character variety $\mathfrak{C}'$ with a holonomy parameter $\alpha$ on $(S^{3}, S^{1})$. For $0<\alpha<\frac{1}{2}$, $\mathfrak{C}'$ is a
one point  set which consists of a unique flat reducible $\theta_{\alpha}$ and  the moduli space $M(D^{4}, D^{2};\theta_{\alpha})_{0}$ also consists of one element $\Theta_{\alpha}$ which is a unique extension of $\theta_{\alpha}$ to $D^{4}\setminus D^{2}$. 
The computation of  group cohomology of $\pi_{1}(S^{3}\setminus S^{1})$ shows that ${\rm dim}H^{1}(S^{3}\setminus S^{1}; {\rm ad}\theta_{\alpha})=0$. Taking double of $(D^{4}, D^{2})$, we have a relation of indexes
\[2{\rm ind}D_{\Theta_{\alpha}}+1={\rm ind}D_{\Theta_{\alpha}\#{\Theta_{\alpha}}}\]
${\rm ind}D_{\Theta_{\alpha}\#\Theta_{\alpha}}=-1$ by the index formula for a closed pair $(S^{4}, S^{2})$. 
Thus we have ${\rm ind}D_{\Theta_{\alpha}}=-1$ and $H^{2}(D^{4}\setminus D^{2};{\rm ad}\Theta_{\alpha})=0$. 
In particular, the gluing along $\mathfrak{C}'$ is unobstructed.
Morse-Bott gluing argument shows that 
\[M([0, 1]\times Y, [0, 1]\times K; \beta, \beta_{1})_{d}=M([0, 1]\times Y\setminus D^{4}, [0, 1]\times K\setminus D^{2};\beta, \theta_{\alpha},\beta')_{d}\]
and similarly for a pair $(Y', K')$.
Thus we have 
\begin{eqnarray*}
\#M^{g^{\infty}}_{z}(W^{I}, S^{I},\beta, \beta';\beta_{1}, \beta'_{1})&=&\#M_{z'}(W^{I}, S^{I};\beta, \theta_{\alpha}, \beta_{1})\cdot\#M_{z''}(W^{I},S^{I};\beta', \theta_{\alpha}, \beta'_{1})\\
&=&\#M_{z'}(Y\times [0, 1], K\times [0, 1];\beta, \beta_{1})\cdot\#M_{z''}(Y'\times [0, 1], K'\times [0, 1];\beta', \beta'_{1}).\\
\end{eqnarray*}
Thus $\tilde{m}_{(W^{I}, S^{I})}$ is $\mathcal{S}$-chain homotopic to a morphism $\tilde{m}_{\rm prod}$ which is induced from the product cobordism $(Y\sqcup Y', K\sqcup K')\times [0,1]$.
The $\mathcal{S}$-morphism $\tilde{m}_{\rm prod}$ is an isomorphism of $\mathcal{S}$-complex (see \cite[Lemma 6.29]{DS1}), and in fact $\mathcal{S}$-chain homotopic to the identity by the formal argument.
\end{proof}

\bibliographystyle{plain}
\bibliography{article}

\noindent\textsc{
Depatrtment of Mathematics, 
Kyoto University,\\
Kyoto, 606-8502, Japan.}
\\ \\
\noindent{E-mail:\ imori.hayato.67m@st.kyoto-u.ac.jp}

\end{document}